\documentclass[12pt]{amsart}
\usepackage[T1]{fontenc}

\usepackage[square,numbers]{natbib}
\usepackage{hyperref}
\usepackage{graphicx}
\usepackage{tikz-cd}
\usepackage{amsmath}
\usepackage{xfrac} 

\usepackage{caption}
\usepackage{subcaption}

\newtheorem{theorem}{Theorem}[section]
\newtheorem{lemma}[theorem]{Lemma}
\newtheorem{corollary}[theorem]{Corollary}

\newtheorem{conjecture}[theorem]{Conjecture}
\theoremstyle{definition}

\theoremstyle{remark}
\newtheorem{remark}{Remark}

\newcommand{\HON}{{\tt HONEY}(\tau_n)}
\newcommand{\MOB}{{\tt M \ddot{O} BIUS}(\widetilde{\tau}_{n},\delta)}
\newcommand{\BDR}{{\tt BDRY}(\widetilde{\tau}_{n}, \delta)}
\newcommand{\R}{\mathbb{R}}
\newcommand{\C}{\mathbb{C}}
\newcommand{\Z}{\mathbb{Z}}
\newcommand{\Q}{\mathbb{Q}}
\newcommand{\N}{\mathbb{N}}
\newcommand{\Parn}{{\rm Par}_n}
\newcommand{\ParnQ}{{\rm Par}_n^{\mathbb{Q}}}

\newcommand{\NLsat}{{\rm NL}\text{-}{\rm sat}(n)}
\newcommand{\modtwo}{\hspace{0.2em} ( {\rm mod}\hspace{0.2em} 2)}

\title{Proof of the Newell-Littlewood saturation conjecture}
\author{Jaewon Min}
\address{Dept.~of Mathematics, University of Illinois at Urbana-Champaign, Urbana, IL 61801}
\email{jaewonm2@illinois.edu}
\date{August 30, 2024}

\begin{document}

\maketitle

\vspace{-1ex}

\begin{abstract}
By inventing the notion of \emph{honeycombs}, A.~Knutson and T.~Tao proved the saturation conjecture for Littlewood-Richardson coefficients. The Newell-Littlewood numbers are a generalization of the Littlewood-Richardson coefficients. By introducing honeycombs on a M{\"o}bius strip, we prove the saturation conjecture for Newell-Littlewood numbers posed by S.~Gao, G.~Orelowitz and A.~Yong.  
\end{abstract}

\setcounter{tocdepth}{2}

\tableofcontents


\section{Introduction}\label{sec1}


\subsection{Background} The irreducible polynomial representations $V_\lambda$ of $\text{GL}_n\C$ are indexed by the set of partitions 
\begin{equation}\label{eqn1.1}
\Parn := \{ \lambda = (\lambda_1 ,\cdots , \lambda_n ) \in \Z^n \mid \lambda_1 \geq \cdots \geq \lambda_n \geq 0  \};
\end{equation}
see, \textit{e.g.},  \cite{Ful04}. For each $\mu, \nu \in \Parn$, 
\begin{equation}\label{eqn1.2}
V_\mu \otimes V_\nu \cong \bigoplus_{\lambda \in {\rm Par}_n} V_\lambda^{\oplus c_{\mu , \nu}^{\lambda}}.
\end{equation}
The tensor product multiplicities $c_{\mu, \nu}^{\lambda}$ are the \textbf{Littlewood-Richardson coefficients}. 

For each $k \in \N:=\{1,2,3,\ldots\}$ and $\lambda \in {\rm Par}_n$, let $k \lambda := (k \lambda_1 , \cdots , k\lambda_n)$. 

\begin{theorem}[Saturation of Littlewood-Richardson coefficients {\rm\citep{Tao99}}]\label{thm1.1}
Let $\lambda, \mu, \nu \in \Parn$. If there exists $k \in \N$ such that $c_{k \mu , k \nu}^{k \lambda} > 0$, then $c_{\mu , \nu}^{\lambda} > 0$.
\end{theorem}

A.~Knutson and T.~Tao proved Theorem~\ref{thm1.1} using \textit{honeycombs} \citep{Tao99}. Honeycombs are combinatorial objects used to count Littlewood-Richardson coefficients. This paper concerns a generalization of Theorem~\ref{thm1.1} and its proof. 

The significance of the saturation theorem stems from \textit{Horn's conjecture} \citep{Hor62} which gives a recursive description of linear inequalities, called \emph{Horn's inequalities}, on the eigenvalues of $n\times n$ Hermitian matrices $A$, $B$ and $A+B$. Theorem~\ref{thm1.1} combined with earlier work of  A.~A.~Klyachko \citep{Kly98} proved Horn's conjecture; see W.~Fulton's survey \citep{Ful00}.


\subsection{Main result}

We generalize Theorem~\ref{thm1.1} and its proof to the  \textbf{Newell-Littlewood numbers}, which are defined, using the Littlewood--Richardson coefficients, as
follows:
\begin{equation}\label{eqn1.4}
N_{\lambda, \mu , \nu}:= \sum_{\alpha , \beta , \gamma \in \Parn} c_{\beta , \gamma}^{\lambda} c_{\gamma , \alpha}^{\mu} c_{\alpha , \beta}^{\nu} \quad (\lambda , \mu , \nu \in \Parn).
\end{equation}

For each $\lambda \in \Parn$, let $|\lambda | := \lambda_1 + \cdots + \lambda_n$. If $c_{\mu, \nu}^\lambda \neq 0$, then $|\mu| + |\nu | = |\lambda |$. According to \citep[Lemma 2.2]{Yon21},
\begin{equation}\label{eqn1.6}
|\mu| + |\nu | = |\lambda | \quad \Rightarrow \quad N_{\lambda , \mu , \nu} = c_{\mu , \nu}^{\lambda}.
\end{equation}
Thus, Newell-Littlewood numbers  generalize  Littlewood-Richardson coefficients.

In 2021, S.~Gao, G.~Orelowitz and A.~Yong \citep[Conjecture~5.5, 5.6]{Yon21} conjectured a generalization of Theorem \ref{thm1.1}.
In \emph{ibid.}, this conjecture was proved for the special cases that $\lambda = \mu = \nu$ \cite[Theorem~4.1]{Yon21} and for $n=2$ \cite[Theorem~4.1]{Yon21}. In \citep[Corollary~6.1]{Yon22}, S.~Gao, G.~Orelowitz, N.~Ressayre, and A.~Yong gave a computational proof of the cases when $n \leq 5$.
Our main result is a complete proof of said conjecture from \citep[Conjecture~1.1]{Yon21}, 
by modifying the proof of Theorem \ref{thm1.1} in \citep{Tao99}.

\begin{theorem}[Newell-Littlewood saturation {\rm \cite[Conjecture 5.5, 5.6]{Yon21}}]\label{thm1.2}
Let  $\lambda,\mu,\nu \in \Parn$ satisfying $|\lambda|+|\mu|+|\nu| \equiv 0 \modtwo$. 
If there exists $k \in \N$ such that $N_{k\lambda,k\mu,k\nu} >0$, then $N_{\lambda,\mu,\nu} >0$.
\end{theorem}

This follows from the technical center of this paper, Theorem~\ref{thm3.2} in Subsection \ref{sub3.2}.

In view of \eqref{eqn1.6}, Theorem \ref{thm1.2} immediately implies the saturation of Littlewood-Richardson coefficients. Actually, our method can be used to modify some arguments in the proof of Theorem \ref{thm1.1} in \citep{Tao99}; see Remark \ref{rmk5.1} in Subsection \ref{sub5.2}.

We now discuss consequences of proving Theorem \ref{thm1.2}. Analogous to the Horn's inequalities, S.~Gao, 
G.~Orelowitz and A.~Yong 
\citep[Theorem~1.3]{Yon212} defined \textit{extended Horn inequalities} (which we will not restate here) and proved that they are necessary conditions for $N_{\lambda , \mu , \nu} >0$. Additionally, they conjectured the converse; our paper also confirms this conjecture. 

\begin{corollary}{\rm \citep[Conjecture 1.4]{Yon212}}
If $(\lambda , \mu , \nu) \in (\Parn)^3$ satisfies the extended Horn inequalities and $|\lambda| + |\mu| + |\nu| \equiv 0 \modtwo$, then $N_{\lambda , \mu , \nu} >0$.
\end{corollary}
\begin{proof}
Due to \citep[Corollary 8.5]{Yon22}, this follows from Theorem \ref{thm1.2}.
\end{proof}

Therefore, the extended Horn inequalities and $|\lambda| + |\mu| + |\nu| \equiv 0 \modtwo$ completely determine the set
\begin{equation}
{\rm NL} := \{ ( \lambda , \mu , \nu) \in (\Parn)^3 \mid N_{\lambda , \mu , \nu} > 0 \}.
\end{equation}

Another application is  to the eigenvalues of a family of complex matrices. Let
\begin{equation}
\ParnQ := \{ \lambda = (\lambda_1 ,\cdots , \lambda_n ) \in \Q^n \mid \lambda_1 \geq \cdots \geq \lambda_n \geq 0  \},
\end{equation}
\begin{equation}\label{eqn1.5}
\NLsat := \{ (\lambda , \mu , \nu) \in (\ParnQ)^3 \mid \exists k >0, \quad N_{k \lambda, k \mu , k \nu} >0 \}.
\end{equation}
In \citep[Proposition 3.1]{Yon22}, S.~Gao, G.~Orelowitz, N.~Ressayre and A.~Yong proved that $\NLsat$ describes an analogue of the Horn problem for matrices in $\mathfrak{sp}_{2n} \C \cap \mathfrak{u}_{2n} \C$. Theorem~\ref{thm1.2} shows that ${\rm NL}$ also controls the same thing.

Lastly, Theorem \ref{thm1.2} is related to the conjecture suggested in \citep[Section~7]{Tao99}. Given a split reductive group $G$ over $\C$, it has a root system and its irreducible representation is indexed by a dominant integral weight $\lambda$. Write the dual weight as $\lambda^*$ and the tensor product multiplicities by $c_{\mu, \nu}^{\lambda}(G)$.

\begin{theorem}{\rm\citep[Theorem 1.1]{Kap08}}\label{thm1.3}
Let $G$ be a split reductive group over $\C$ and $\lambda , \mu , \nu$ be dominant integral weights such that $\lambda^* + \mu + \nu$ is in the root lattice. Then there exists $k_G \in \N$ with following property:
\begin{equation}
\exists k \in \N \text{ such that } c_{k\mu , k\nu}^{k \lambda}(G) >0 \quad \Rightarrow \quad c_{k_G \mu , k_G \nu}^{k_G \lambda}(G) >0.
\end{equation}
\end{theorem}

\begin{conjecture}{\rm\citep[Conjecture 1.4]{Kap06}}\label{con1.1}
If the root system of $G$ is simply laced, then $k_G$ can be chosen as $1$.
\end{conjecture}

In particular, we are interested in the cases when $G = {\rm SO}_{2n+1}\C, \hspace{0.2em} {\rm Sp}_{2n} \C, \hspace{0.2em} {\rm SO}_{2n} \C$. In \citep[Theorem 1.1]{Kap08}, M.~Kapovich and J.~J.~Millson proved that $k_G = 4$. Additionally, P.~Belkale and S.~Kumar \citep[Theorem 6, 7]{Bel07} proved that $k_G=2$ if $G$ is ${\rm SO}_{2n+1}\C$ or ${\rm Sp}_{2n} \C$. S.~V.~Sam \citep[Theorem 1.1]{Sam12} proved that $k_G=2$ when $G = {\rm SO}_{2n+1}\C, \hspace{0.2em}{\rm Sp}_{2n} \C,\hspace{0.2em}{\rm SO}_{2n} \C$, by using quiver representations, extending the proof of Theorem \ref{thm1.1} given by H.~Derksen and  J.~Weyman~\citep{Der00}.

The possibility that $k_G=1$ when $G = {\rm SO}_{2n}  \C$ remains open. For recent work concerning  ${\rm SO}_{2n}  \C$ and ${\rm Spin}_{2n}  \C$, see, \textit{e.g.}, \citep{KKM09, Kie21}.

Let $G = {\rm SO}_{2n+1}\C,\hspace{0.2em} {\rm Sp}_{2n} \C,\hspace{0.2em}{\rm SO}_{2n} \C$. For the classical Lie groups, irreducible representations are indexed by the set of partitions $\Parn$; see, \textit{e.g.}, \citep{Ful04, Koi87}. $l(\lambda)$ denotes the number of non-zero components of $\lambda = (\lambda_1 , \cdots , \lambda_n)$. According to \citep[Theorem 3.1]{Koi89}, 
\begin{equation}
l(\mu)+l(\nu) \leq n \quad \Rightarrow \quad N_{\lambda, \mu , \nu} = c_{\mu , \nu}^{\lambda}(G).
\end{equation}
The condition imposed on $\mu , \nu \in \Parn$ is called the \textit{stable range}. The next result is an immediate consequence of Theorem~\ref{thm1.2}:

\begin{corollary}
Let $G = {\rm SO}_{2n+1}\C, \hspace{0.2em}{\rm Sp}_{2n} \C,\hspace{0.2em}{\rm SO}_{2n} \C$. Suppose $\lambda, \mu , \nu \in \Parn$ and $l(\mu)+l(\nu) \leq n$. If there exists $k \in \N$ such that $c_{k\mu, k \nu}^{k \lambda}(G) >0$, then $c_{\mu , \nu}^\lambda(G) >0$. 
\end{corollary}

Thus, $k_G$ from Conjecture \ref{con1.1} may be taken as $1$ for $G = {\rm SO}_{2n+1}\C, \hspace{0.2em}{\rm Sp}_{2n} \C,\hspace{0.2em}{\rm SO}_{2n} \C$ if $(\lambda , \mu , \nu)$ is in the stable range.


\subsection{Overview}

In Section \ref{sec2}, we review the construction of \textit{honeycombs} from \citep[Section 2]{Tao99}. Roughly speaking, when given a directed graph which looks like a ``bee hive'', a honeycomb is a map assigning each vertex of the graph to a vector in a plane. We recapitulate how honeycombs compute Littlewood-Richardson coefficients.

In Section \ref{sec3}, we define \textit{M{\"o}bius honeycombs} to be honeycombs ``embroidered'' on a M{\"o}bius strip, building on the setup of Section~2. Since a M{\"o}bius strip cannot be embedded into a plane, we rigorously define the concept by using its covering space. We then prove that they compute Newell-Littlewood numbers. We state our main result about M\"obius honeycombs, Theorem \ref{thm3.2}, and show that Theorem \ref{thm1.2} quickly follows from it.

In Section \ref{sec4}, we construct a particular M{\"o}bius honeycomb, which is an extremal point of related polytope, analogous to \citep[Section 5]{Tao99}. We do this by formulating a linear functional which provides ``height'' to the elements of the polytope. The modification rules of honeycombs from \citep{Tao99} can still be used, but only on limited part which we color in white. 

In Section \ref{sec5}, we observe that there can be a non-oriented loop defined in M{\"o}bius honeycombs, unlike in \citep{Tao99}. Handling these non-oriented loops precisely concerns the condition $|\lambda|+|\mu|+|\nu| \equiv 0 \modtwo$ given in Theorem \ref{thm1.2};
the reason comes down to the fact that the fundamental group of $\R{\mathbb P}^2$ is $\Z / 2 \Z$ . We end the section by proving Theorem~\ref{thm3.2}.


\section{Honeycombs}\label{sec2}

In this section, we review \textit{honeycombs} from \citep[section 2]{Tao99}. In particular, we recall how honeycombs compute Littlewood-Richardson coefficients.

\subsection{Tinkertoys}\label{sub2.1}

Let $B$ be a finite dimensional real vector space. Let $\Gamma$ be a directed graph with $V_{\Gamma}$ and $E_{\Gamma}$ being the set of vertices and edges, respectively. 

A \textbf{tinkertoy} $\tau$ is a triple $(B,\Gamma,d)$  consisting of $B$, $\Gamma$ and a map 
\begin{equation}
d: E_{\Gamma}\rightarrow B.
\end{equation}
The map $d$ is called the \textbf{direction map}. A \textbf{subtinkertoy} $(B,\Delta,d|_{E_{\Delta}})\leq (B,\Gamma,d)$ is a tinkertoy where $\Gamma$ is replaced by
an induced subgraph $\Delta \leq \Gamma$.

In \cite{Tao99}, as it will be in this paper, $B$ is fixed to be the two-dimensional real vector space
\begin{equation}\label{eqn2.1}
B:=\{(x,y,z) \in \mathbb{R}^{3} \mid x+y+z = 0\}.
\end{equation}
Let the \textbf{lattice points} of $B$ be
\begin{equation}\label{eqn2.2}
B_{\mathbb Z}:= \{(x,y,z) \in \mathbb{Z}^{3} \mid x+y+z = 0\}.
\end{equation}
For each $a \in \mathbb{R}$, define three lines in the plane $B$
\begin{subequations}\label{eqn2.3}
\begin{equation}
(a,*,*):=\{(x,y,z) \in B \mid x=a\},
\end{equation}
\begin{equation}
(*,a,*):=\{(x,y,z) \in B \mid y=a\},
\end{equation}
\begin{equation}
(*,*,a):=\{(x,y,z) \in B \mid z=a\}.
\end{equation}
\end{subequations}
Each value $a \in \R$  is  the \textbf{constant coordinate} of the associated line. If 
$a\in {\mathbb Z}$, that line is a \textbf{lattice line}. In Figure \ref{fig1}, we depict
$B_{\mathbb Z}$ and its lattice lines.

\begin{figure}
\centering
\includegraphics[scale=0.45]{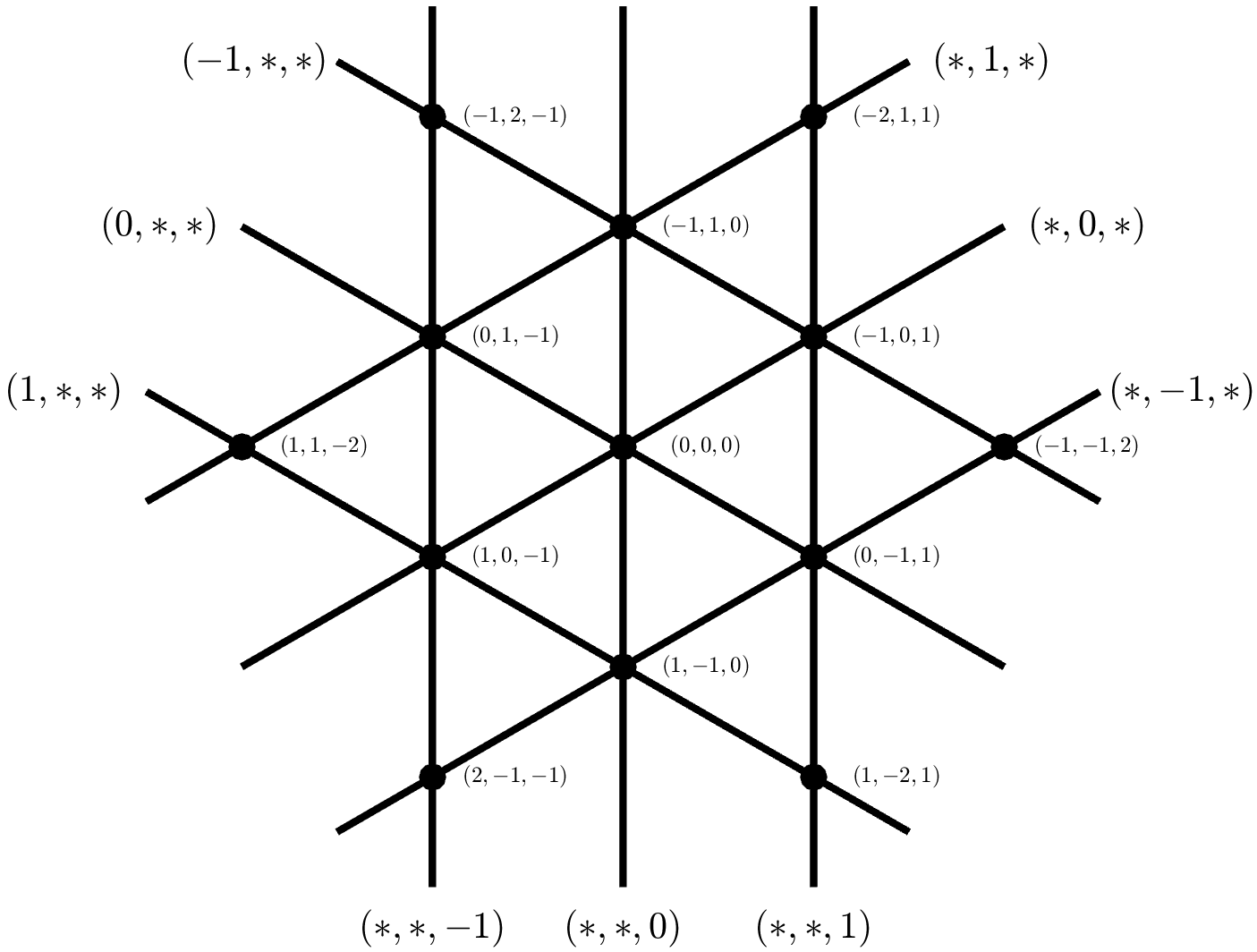}
\caption{$B_{\mathbb Z}$ and its lattice lines}
\label{fig1}
\end{figure}

Next, we  define a directed graph $\Gamma_{\infty}$. Its vertices and edges are
\begin{subequations}
\begin{equation}
V_{\Gamma_\infty} = \{\widetilde{A}_{i,j} \mid i,j \in \Z \} \cup \{\widetilde{B}_{i,j} \mid i,j \in \Z \},
\end{equation}
\begin{align}
E_{\Gamma_\infty} &= \{ (\widetilde{A}_{i,j}, \widetilde{B}_{i,j}) \mid i,j \in \Z \} \\
& \cup \{(\widetilde{A}_{i,j}, \widetilde{B}_{i-1,j}) \mid i,j \in \Z \} \\
& \cup \{ (\widetilde{A}_{i,j}, \widetilde{B}_{i-1,j-1}) \mid i,j \in \Z \}.
\end{align}
\end{subequations}
Here, we denote a directed edge from $U$ to $W$ as $(U,W)$. Consequently, $\widetilde{A}_{i,j}$ has three outgoing edges whereas $\widetilde{B}_{i,j}$ has three incoming edges. See the depiction of $\Gamma_{\infty}$ in Figure \ref{fig2}. 

Lastly, define a direction map $d: E_{\Gamma_{\infty}} \rightarrow B$ by mapping
\begin{align}
(\widetilde{A}_{i,j}, \widetilde{B}_{i-1,j-1}) &\mapsto (0,-1,1), \\
(\widetilde{A}_{i,j}, \widetilde{B}_{i-1,j}) &\mapsto (1,0,-1), \\
(\widetilde{A}_{i,j}, \widetilde{B}_{i,j}) & \mapsto (-1,1,0).   
\end{align}
As in Figure \ref{fig2}, $d$ maps each southeast edges to $(0,-1,1)$, southwest edges to $(1,0,-1)$, and north edges to $(-1,1,0)$. Now, the \textbf{infinite honeycomb tinkertoy} $\tau_{\infty}$ is the triple $\tau_{\infty}:=(B,\Gamma_{\infty},d)$.

\begin{figure}
\centering
\includegraphics[scale = 0.52]{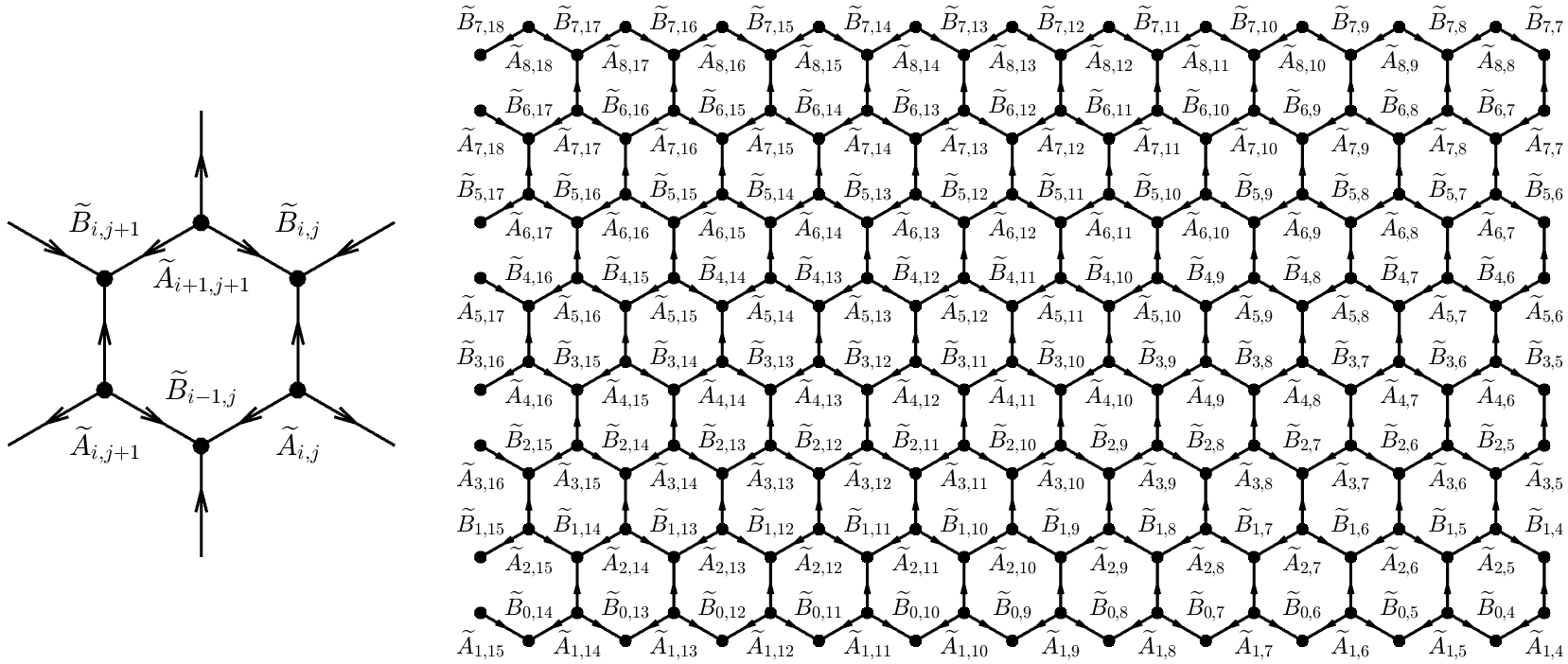}
\caption{The graph $\Gamma_\infty$ of the infinite honeycomb tinkertoy $\tau_\infty$.}
\label{fig2}
\end{figure}

Define the \textbf{$GL_{n}$ honeycomb tinkertoy} $\tau_{n}:= (B, \Delta_{n}, d|_{E_{\Delta_{n}}})$  of $\tau_{\infty}$ as follows.\footnote{From now on, we assume $n \in \N$ without saying so.} The graph $\Delta_{n}$ is the induced subgraph of $\Gamma_n$ using the subset of vertices
\begin{equation}\label{eqn2.5}
V_{\Delta_{n}}:= \{ \widetilde{A}_{i,j} \mid 1 \leq i < j \leq n \} \cup \{ \widetilde{B}_{i,j} \mid 0 \leq i < j \leq n \} .
\end{equation}
$E_{\Delta_{n}}$ is the resulting edge set.  A vertex that is not connected to three edges is a \textbf{boundary vertex}. There are exactly $3(n-1)$-many boundary vertices in $\Delta_n$: for $0 \leq i \leq n-1$,
\begin{itemize}
\item $\widetilde{B}_{i,n}$ : not connected to edge $e$ such that $d(e) = (0,-1,1)$,
\item $\widetilde{B}_{i,i+1}$ : not connected to edge $e$ such that $d(e) = (1,0,-1)$,
\item $\widetilde{B}_{0,i+1}$ : not connected to edge $e$ such that $d(e) = (-1,1,0)$.
\end{itemize}
See Figure \ref{fig3} for the case $n=5$.

\begin{figure}
\centering
\includegraphics[scale = 0.6]{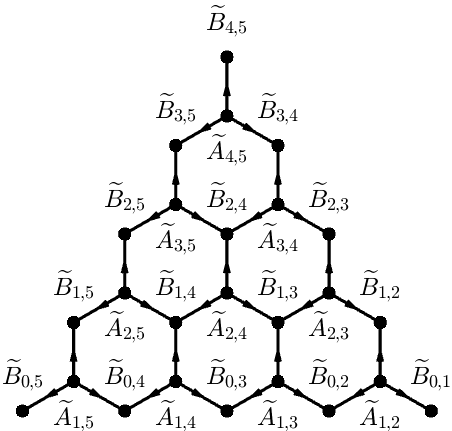}
\caption{The graph $\Delta_{5}$ of the $GL_{5}$ honeycomb tinkertoy $\tau_{5}$.}
\label{fig3}
\end{figure}


\subsection{Configurations}\label{sub2.2}

Let $B$ be a finite dimensional real vector space and $\tau = (B, \Gamma , d)$ be a tinkertoy. A \textbf{configuration $h$ of a tinkertoy $\tau$} is a function $h:V_{\Gamma} \rightarrow B$ satisfying
\begin{equation}\label{eqn2.6}
h({\sf head}(e))-h({\sf tail}(e)) \in \{ a \cdot v \in B \mid a\in {\mathbb R}_{\geq 0}, v= d(e)\}.
\end{equation}
A configuration $h$ of a $GL_n$ honeycomb tinkertoy $\tau_n$ is a \textbf{honeycomb} \citep{Tao99}.

Assuming that $B$ is a plane as in \eqref{eqn2.1}, draw a picture of a configuration $h$ of $\tau = (B, \Gamma , d)$ by marking the position of $h(P)$ in $B$ for all $P \in V_\Gamma$. In addition, if vertices $P$ and $Q$ are connected by a directed edge $e$, then connect $h(P)$ and $h(Q)$ by a line segment. For instance, Figure \ref{fig4} illustrates two honeycombs $h_{1},h_{2}$. Observe, $h_{1}$ and $h_{2}$ are ``distortions'' of the graph $\Delta_{5}$ in Figure \ref{fig3}; the edge directions are the same, but lengths may differ. 

\begin{figure}
\centering
\includegraphics[scale = 0.5]{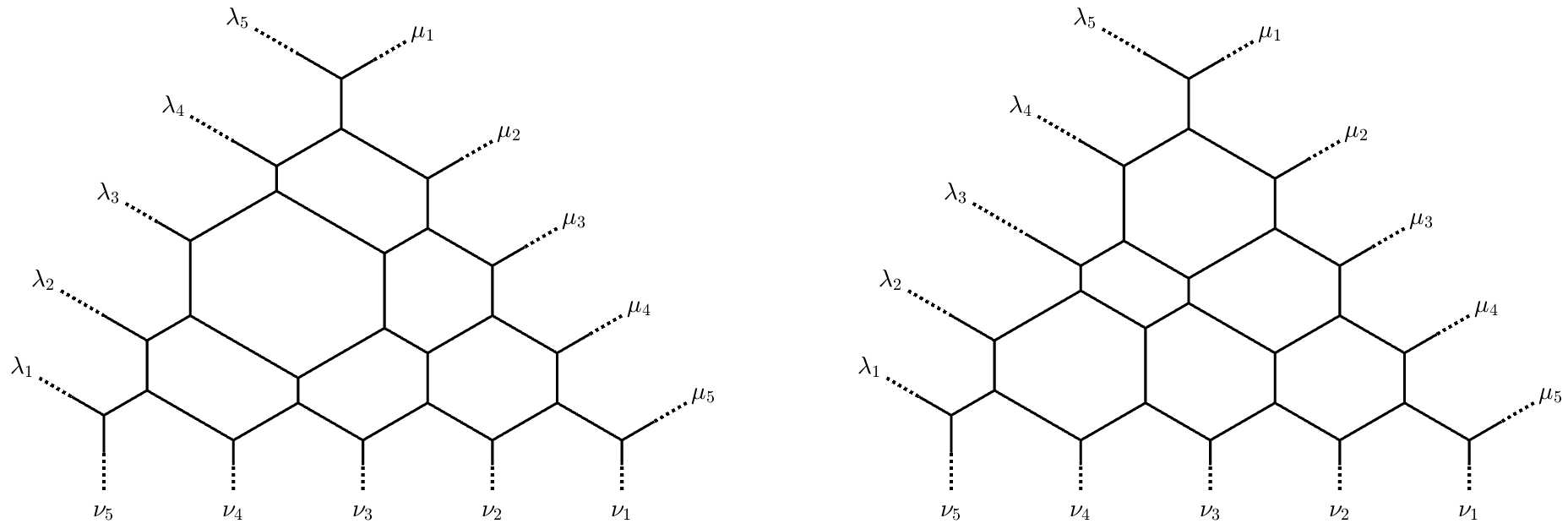}
\caption{Configurations $h_{1}$ and $h_{2}$ of $GL_{5}$ honeycomb tinkertoy $\tau_{5}$.}
\label{fig4}
\end{figure}

Regard $h=(h(v))_{v \in V_\Gamma}$ as an element of a vector space
\begin{equation}\label{eqn2.7}
B^{V_\Gamma} = \prod_{v \in V_{\Gamma}} B(v),  
\end{equation}
where $B(v)=B$ for each $v\in V_{\Gamma}$. 

For configurations $h_{1},h_{2}$ and $c_{1}, c_{2} \in \mathbb{R}$, define $c_{1} \cdot h_{1}+c_{2} \cdot h_{2}$ as a function $V_\Gamma \rightarrow B$ mapping
\begin{equation}\label{eqn2.13}
(c_{1} \cdot h_{1}+c_{2} \cdot h_{2})(v):=c_{1} \cdot h_{1}(v) + c_{2} \cdot h_{2}(v), \quad v \in V_\Gamma.
\end{equation}

\begin{lemma}\label{lem2.1}
Let $\tau = (B , \Gamma , d)$ be a tinkertoy, $h_1 , h_2$ be configurations of $\tau$ and $c_1 , c_2 \in \R_{\geq 0}$. Then $c_1 \cdot h_1 + c_2 \cdot h_2$ is a configuration of $\tau$.
\end{lemma}
\begin{proof}
Let $e \in E_\Gamma$. Denote $w_h:={\sf head}(e)$ and $w_t:={\sf tail}(e)$ and $v:=d(e)$. From \eqref{eqn2.6}, there exists $a_1 , a_2 \in \R_{\geq 0}$ such that
\begin{equation}
h_1(w_h) - h_1 (w_t) = a_1 v, \quad h_2(w_h) - h_2 (w_t) = a_2 v. 
\end{equation}
Denote $h:=c_1 \cdot h_1 + c_2 \cdot h_2$. Then
\begin{equation}
h(w_h) - h(w_t) = a_1 c_1 v + a_2 c_2 v.
\end{equation}
Hence, $h = c_1 \cdot h_1 + c_2 \cdot h_2$ satisfies \eqref{eqn2.6}, proving that it is a configuration of $\tau$.
\end{proof} 

Fix $GL_{n}$ honeycomb tinkertoy $\tau_n = (B , \Delta_n , d|_{\Delta_n } )$. As in Figure \ref{fig3}, the boundary vertices of $\Delta_n$ are $\widetilde{B}_{i,n} , \widetilde{B}_{i,i+1} , \widetilde{B}_{0,i+1} $ for each $0 \leq i \leq n-1$. Let $h : V_{\Delta_n} \rightarrow B$ be a honeycomb. Denote
\begin{equation}
\HON := \{ h \in B^{V_{\Delta_n}} \mid h \text{ is a configuration of }\tau_n\}.
\end{equation}
The \textbf{boundary map} is
\begin{equation}\label{eqn2.8}
\partial : \HON \rightarrow \mathbb{R}^{3n}, \quad h \mapsto (\lambda_1 , \cdots , \lambda_n, \mu_1 , \cdots, \mu_n, \nu_1 , \cdots , \nu_n).
\end{equation}
Here, for each $1\leq i \leq n$, $\lambda_i, \mu_i , \nu_i$ are chosen by
\begin{itemize}
\item $\lambda_i$ : $x$ coordinate of $h(\widetilde{B}_{i-1,n})$,
\item $\mu_i$ : $y$ coordinate of $h(\widetilde{B}_{n-i,n-i+1})$,
\item $\nu_i$ : $z$ coordinate of $h(\widetilde{B}_{0,i})$.
\end{itemize}
From now on, write $(\lambda_1 , \cdots , \lambda_n, \mu_1 , \cdots, \mu_n, \nu_1 , \cdots , \nu_n)$ as $(\lambda , \mu , \nu)$. 
By definition of $h$, $h(\widetilde{B}_{i-1,n})$ is on the line $(\lambda_i , * , *)$. Similarly, $h(\widetilde{B}_{n-i,n-i+1}) \in (*, \mu_i , *)$ and $h(\widetilde{B}_{0,i}) \in (*, * , \nu_i)$. The dotted lines indexed by $\lambda_i , \mu_i , \nu_i$ in Figure \ref{fig4} are these lines. For example, Figure \ref{fig4} depicts $h_1,h_2\in \HON$ with $\partial h_1 = \partial h_2 = (\lambda , \mu , \nu)$.

It is immediate from \eqref{eqn2.13} that
\begin{equation}\label{eqn2.14}
\partial (c_1 \cdot h_1 + c_2 \cdot h_2) = c_1 \cdot \partial (h_1) + c_2 \cdot \partial (h_2), \quad (h_1 , h_2 \in \HON, c_1 ,c_2 \in \R_{\geq 0}).
\end{equation}


\subsection{Littlewood-Richardson coefficients}

For each $\lambda = (\lambda_1 , \cdots , \lambda_n) \in \R^n$, write its dual weight $\lambda^*:= (- \lambda_n , \cdots , - \lambda_1)$.

\begin{theorem}{\rm {\citep[Theorem 4]{Tao99}}}\label{thm2.1}
Let $\lambda,\mu,\nu \in \Parn$ and $\tau_n = (B, \Delta_n , d|_{\Delta_n} )$ be the $GL_n$ honeycomb tinkertoy. Then $c_{\mu , \nu}^{\lambda}$ counts the
number of honeycombs $h \in \HON$ satisfying:
\begin{itemize}
\item $\partial (h) = (\mu^* , \nu^*, \lambda )$, and
\item $\forall v \in V_{\Delta_n}$, $h(v) \in B_\Z$.
\end{itemize}
\end{theorem}

There is a relationship between \textit{Berenstein-Zelevinsky patterns}
\citep{Ber89, Ber92} and honeycombs; see \citep{Tao99} for further discussion.

\begin{theorem}{\rm {\citep[Theorem 2]{Tao99}}} \label{thm2.2}
Let $\tau_n = (B, \Delta_n , d|_{\Delta_n})$ be the $GL_n$ honeycomb tinkertoy and $h \in \HON$ such that $\partial (h) \in \mathbb{Z}^{3n}$. Then there exists $g \in \HON$ such that:
\begin{itemize}
\item $\partial (g) = \partial (h)$, and
\item $\forall v \in V_{\Delta_n}$, $g(v) \in B_\Z$.
\end{itemize}
\end{theorem}

\begin{proof}[Proof of Theorem \ref{thm1.1}]
Suppose $\lambda, \mu, \nu \in \Parn$ and $k \in \N$ such that $c_{k \mu, k \nu}^{k \lambda}>0$. By Theorem \ref{thm2.1}, there exists $h \in \HON$ such that
\begin{equation}\label{eqn2.9}
\partial (h) = (k \mu^* , k \nu^* , k\lambda).
\end{equation}
Since $k>0$, $\frac{1}{k} h \in \HON$ by Lemma \ref{lem2.1}. Due to \eqref{eqn2.14},
\begin{equation}\label{eqn2.10}
\partial \left(\frac{1}{k} h \right) =(\mu^* , \nu^* , \lambda).
\end{equation}
Apply Theorem \ref{thm2.2} to find $g \in \HON$ such that $\partial (g) = \partial (\frac{1}{k}h)$ and $g(v) \in B_\Z$ for all $v \in V_{\Delta_n}$. By Theorem \ref{thm2.1} once more, $c_{\mu, \nu}^{\lambda}>0$.
\end{proof}


\section{M{\"o}bius honeycombs}\label{sec3}

In this section, we introduce a new concept, \textit{M{\"o}bius honeycombs}. We prove that the number of M{\"o}bius honeycombs is the same as Newell-Littlewood numbers, analogous to (and, in fact generalizing) how honeycombs compute Littlewood-Richardson coefficients (Theorem~\ref{thm2.1}).


\subsection{M{\"o}bius honeycomb tinkertoys}

Recall, in Subsection \ref{sub2.1}, the infinite honeycomb tinkertoy $\tau_{\infty} = (B, \Gamma_{\infty},d)$ was defined. Define a subgraph $\widetilde{\Gamma}_n$ of $\Gamma_\infty$ induced by the vertices
\begin{equation}\label{eqn3.1}
V_{\widetilde{\Gamma}_{n}}:=\{ \widetilde{A}_{i,j} \in V_{\Gamma_\infty} \mid 0 \leq i \leq n \} \cup \{ \widetilde{B}_{i,j} \in V_{\Gamma_\infty} \mid 0 \leq i \leq n \} .
\end{equation}
We define the \textbf{M{\"o}bius honeycomb tinkertoy} as the subtinkertoy 
\[\widetilde{\tau}_{n}:=(B,\widetilde{\Gamma}_{n}, d|_{E_{\widetilde{\Gamma}_{n}}})\] of $\tau_\infty$. From now on, the direction map $d|_{E_{\widetilde{\Gamma}_n}}: E_{\widetilde{\Gamma}_n} \rightarrow B$ is denoted simply as $d$.

$\widetilde{\Gamma}_{n}$ is an infinite strip composed of $(n-1)$-number of layers of hexagons. For instance, $\widetilde{\Gamma}_{5}$ is depicted in Figure \ref{fig5}. There are vertices connected to exactly one edge in Figure \ref{fig5}, namely $\widetilde{A}_{0,j}, \widetilde{B}_{n,j}$ for $j \in \Z$. Such vertices of $\widetilde{\Gamma}_{n}$ are the \textbf{boundary vertices} in $\widetilde{\Gamma}_n$.

\begin{figure}
\centering
\begin{subfigure}[b]{\textwidth}
\centering
\includegraphics[scale = 0.52]{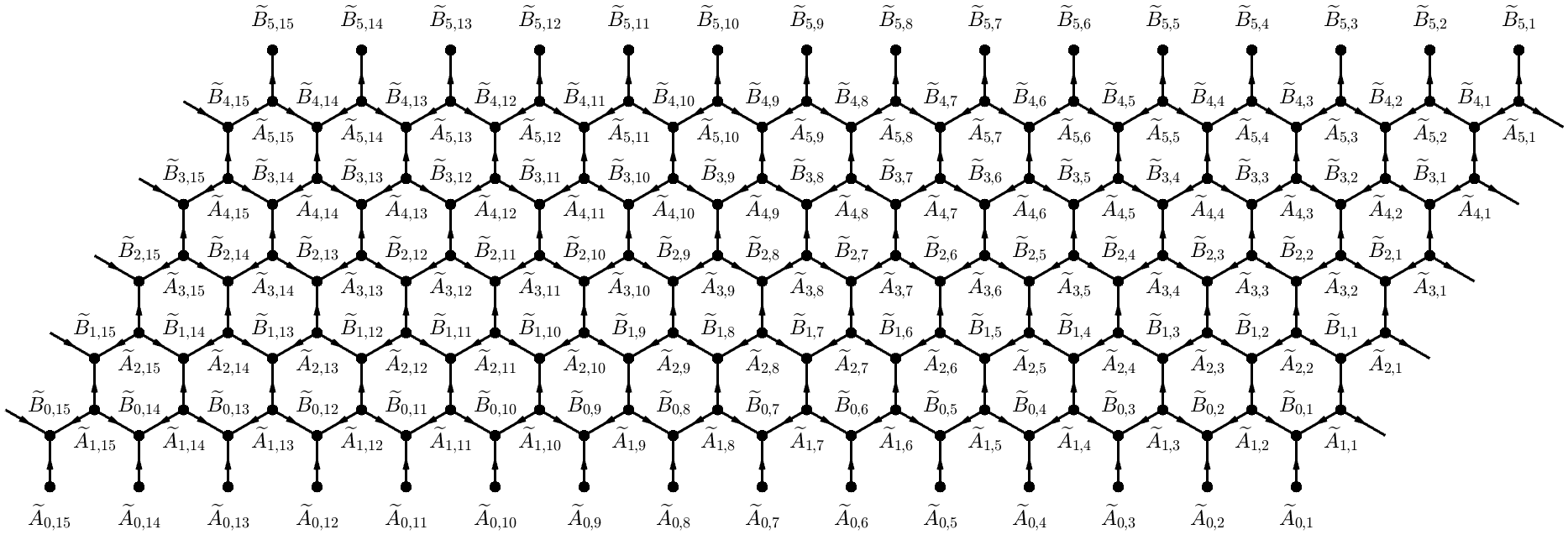}
\caption{The directed graph $\widetilde{\Gamma}_{5}$ of the M{\"o}bius honeycomb tinkertoy $	\widetilde{\tau}_{5}$.}
\label{fig5}
\end{subfigure}

\begin{subfigure}[b]{\textwidth}
\centering
\includegraphics[scale = 0.52]{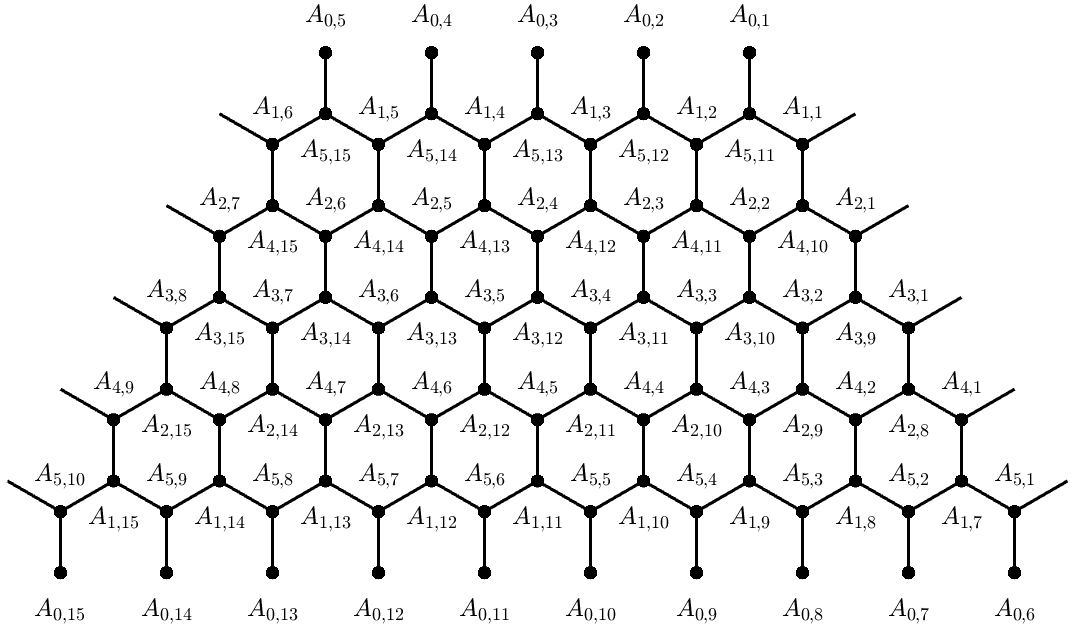}
\caption{The graph $\Gamma_{5}$.}
\label{fig6}
\end{subfigure}
\caption{$\widetilde{\Gamma}_5$ and its quotient graph $\Gamma_5$.}
\label{fig56}
\end{figure}

We now define a graph $\Gamma_{n}$, which will be a ``quotient graph'' of 
$\widetilde{\Gamma}_n$. Intuitively, ``slice'' $\widetilde{\Gamma}_n$ into pieces by using trapezoids as in Figure \ref{fig41}. We want to identify all trapezoids as one, which corresponds to the quotient graph $\Gamma_n$. For instance, four bold vertices of $\widetilde{\Gamma}_5$ in Figure \ref{fig41} are identified as a vertex of $\Gamma_5$. 

To be precise, identify the vertices of $\widetilde{\Gamma}_{n}$ using the equivalence relation $\sim$ defined by
\begin{subequations}\label{eqn3.3}
\begin{equation}
\widetilde{A}_{i,j} \sim \widetilde{B}_{-i+n,-i+j+2n}, \quad (i,j \in \Z, 0 \leq i \leq n)
\end{equation}
and
\begin{equation}
\widetilde{B}_{i,j} \sim \widetilde{A}_{-i+n,-i+j+2n}. \quad (i,j \in \Z, 0 \leq i \leq n).
\end{equation}
\end{subequations}

\begin{figure}
\centering
\begin{subfigure}[b]{\textwidth}
\centering
\includegraphics[scale = 0.44]{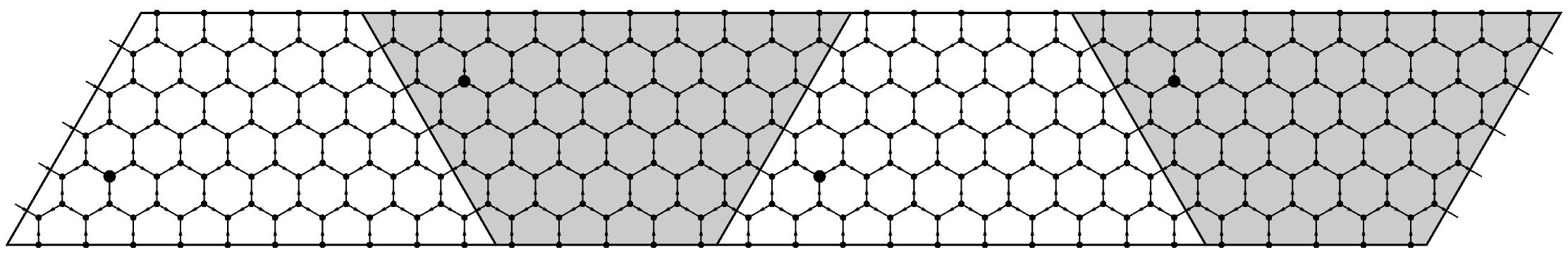}
	\caption{$\widetilde{\Gamma}_{5}$ and $\Gamma_5$ .}
\label{fig41}
\end{subfigure}

\begin{subfigure}[b]{\textwidth}
\centering
\includegraphics[scale = 0.44]{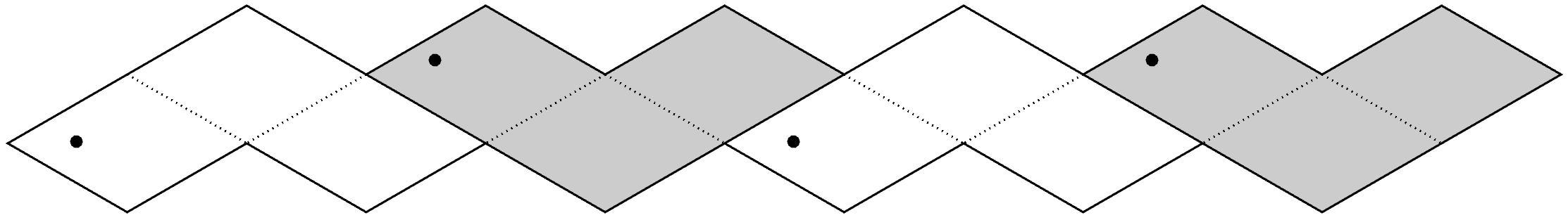}
\caption{$\widetilde{B}_\delta$ and $B_\delta$.}
\label{fig43}
\end{subfigure}

\begin{subfigure}[b]{\textwidth}
\centering
\includegraphics[scale = 0.44]{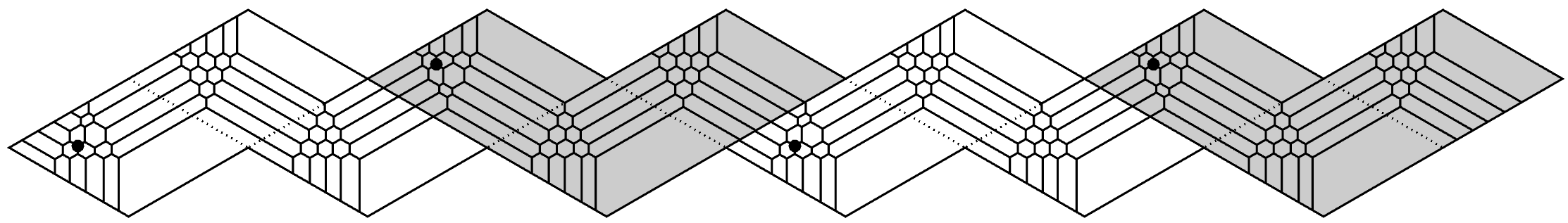}
\caption{Images of $\widetilde{h}$ and $h$.}
\label{fig44}
\end{subfigure}
\caption{Equivalence relations.}
\label{fig414344}
\end{figure}

The vertices of $\Gamma_n$ are representatives of the equivalence classes $[\widetilde{P}]$ for each $\widetilde{P} \in V_{\widetilde{\Gamma}_n}$; we have the quotient map induced by the equivalence relation: 
\begin{equation}
	p_v : V_{\widetilde{\Gamma}_{n}} \rightarrow V_{\Gamma_{n}}, \quad \widetilde{P} \mapsto [\widetilde{P}].
\end{equation}

Next, we define an equivalence relation $\equiv$ on the edges in $\widetilde{\Gamma}_n$. Write a directed edge $\widetilde{e} = ({\sf tail}(\widetilde{e}), {\sf head}(\widetilde{e}))$. For each $\widetilde{e} = (\widetilde{A},\widetilde{B})$ and $\tilde{e}' = (\widetilde{A}' , \widetilde{B}')$, set
\begin{equation}
	\widetilde{e} \equiv \tilde{e}' \iff \widetilde{A} \sim \widetilde{A}' , \widetilde{B} \sim \widetilde{B}' \text{ or } \widetilde{A} \sim \widetilde{B}' , \widetilde{B} \sim \widetilde{A}'.
\end{equation}
The edges of $\Gamma_n$ are representatives of equivalence classes $[\widetilde{e}]$ for each $\widetilde{e} \in E_{\widetilde{\Gamma}_n}$. Here, $[\widetilde{e}]$ is a \emph{non-directed} edge connecting $p_v({\sf tail}(\widetilde{e}))$ and $p_v ({\sf head}(\widetilde{e}))$. We denote a non-directed edge $e = \{ A, B \}$ if $e$ connects vertices $A$ and $B$. The quotient map is defined by
\begin{equation}
	p_e : E_{\widetilde{\Gamma}_n} \rightarrow E_{\Gamma_n}, \quad \widetilde{e} \mapsto [\widetilde{e}].
\end{equation}

From \eqref{eqn3.3}, $\widetilde{A}_{i,j} \sim \widetilde{A}_{i,j+3n}$ and $\widetilde{B}_{i,j} \sim \widetilde{B}_{i,j+3n}$ for all indices. Therefore, there are $3n(n+1)$-many equivalence classes in $V_{\widetilde{\Gamma}_n}$, each represented by $\widetilde{A}_{i,j}$ for $0 \leq i \leq n, 1 \leq j \leq 3n$. Set $A_{i,j} := p_v (\widetilde{A}_{i,j})$ for $0\leq i \leq n, 1 \leq j \leq 3n$. Then the elements of $\Gamma_n$ are indexed by
\begin{subequations}\label{eqn3.23}
\begin{equation}
V_{\Gamma_n}= \{ A_{i,j} \mid i , j \in  \Z, 0\leq i \leq n, 1 \leq j \leq 3n \},
\end{equation}
\begin{align}
E_{\Gamma_n} &=  \{ \{A_{i,j},A_{-i+n,-i+j+2n} \} \mid 0 \leq i \leq n, \hspace{1em} 1 \leq j \leq i+n \}\\
& \cup \{ \{A_{i,j},A_{-i+n+1,-i+j+2n}\} \mid 1 \leq i \leq n, \hspace{1em} 1 \leq j \leq i+n \}\\
& \cup \{ \{A_{i,j},A_{-i+n+1,-i+j+2n+1}\} \mid 1 \leq i \leq n, \hspace{1em} 1 \leq j \leq i+n-1 \}.
\end{align}
\end{subequations}

In summary, $\Gamma_n$ is a finite graph embedded in a M{\"o}bius strip. For instance, consider $\Gamma_5$ in Figure \ref{fig6}. Following \eqref{eqn3.23}, each of the vertices $A_{1,1},A_{2,1},A_{3,1},A_{4,1}, A_{5,1}$ are connected to $A_{5,10},A_{4,9},A_{3,8},A_{2,7},A_{1,6}$, respectively.

For $1 \leq j \leq 3n$, we call $A_{0,j}$ \textbf{boundary vertices} in $\Gamma_n$.


\subsection{Definition of M{\"o}bius honeycombs}\label{sub3.2}

Fix $\delta \in \mathbb{N}$. For each $k \in \mathbb{Z}$, define subsets of $B$
\begin{align}\label{eqn3.15}
D_{\delta}^{(2k)}& := \{ (x,y,z) \in B \mid (k-1)\delta \leq x \leq k \delta, \hspace{1em} (k-1)\delta \leq y \leq k\delta \}, \\
D_{\delta}^{(2k+1)}&:= \{ (x,y,z) \in B \mid (k-1)\delta \leq x \leq k \delta, \hspace{1em}  k\delta \leq y \leq (k+1)\delta \},
\end{align}
\begin{equation}\label{eqn3.4}
\widetilde{B}_{\delta}:=\bigcup_{k \in \mathbb{Z}} D_{\delta}^{(k)}.
\end{equation}
$\widetilde{B}_{\delta}$ is depicted in Figure \ref{fig7}, as an infinite zigzag strip. Here, $D_{\delta}^{(k)}$ is a rhombus. In Figure \ref{fig7}, there are six rhombi, which are $D_{\delta}^{(0)}, D_{\delta}^{(-1)} , \cdots , \cdots D_{\delta}^{(-5)} $, from the left to the right.

We want to define a quotient space $B_\delta$ of $\widetilde{B}_\delta$. Intuitively, we ``slice'' $\widetilde{B}_\delta$ into pieces and identify them into one to construct $B_\delta$. See Figure \ref{fig43}. The four bold points are identified as one element in $B_\delta$.

To write a formal definition, define an equivalence relation on $B$, namely
\begin{equation}\label{eqn3.5}
(x,y,z) \sim (y-2\delta, x-\delta, z+3\delta).\footnote{Momentarily, we will justify this overload
of the use the symbol
$\sim$. See \eqref{eqn3.8}.} 
\end{equation}
Denote the quotient map by $q : B \rightarrow \sfrac{B}{\sim}$. Define $B_{\delta} := q( \widetilde{B}_{\delta} )$. $\widetilde{B}_\delta$ is an infinite strip whereas $B_\delta$ is a M{\"o}bius strip; see Figure \ref{fig78}.

\begin{figure}
\centering
\begin{subfigure}[b]{\textwidth}
\centering
\includegraphics[scale = 0.45]{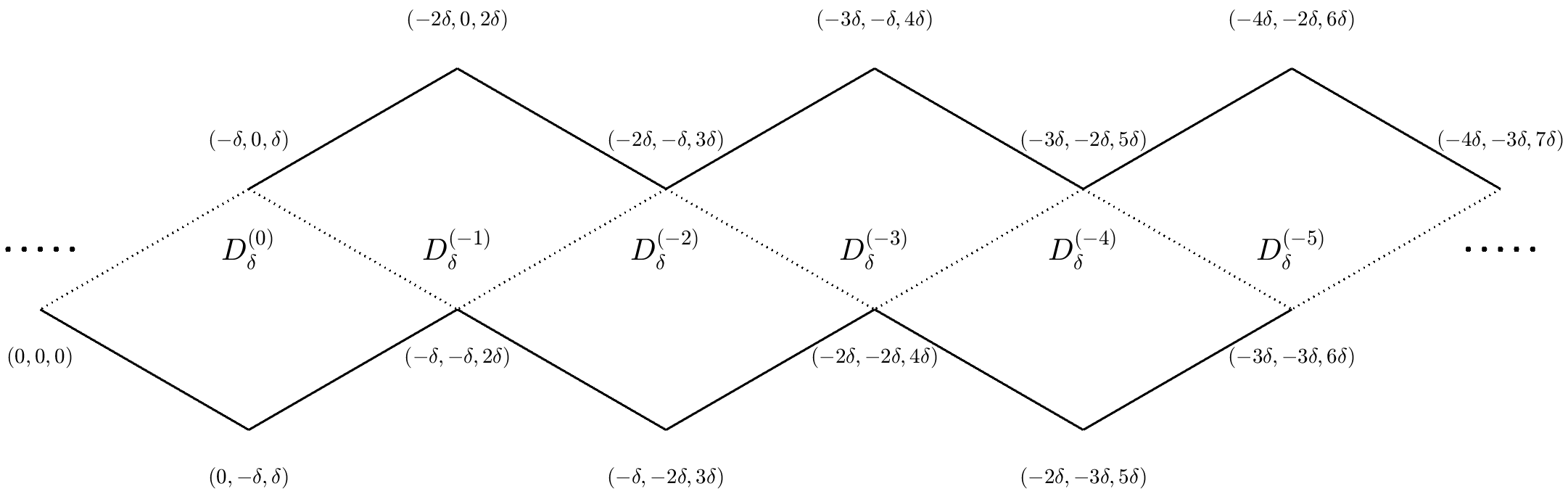}
\caption{$\widetilde{B}_{\delta}$ contained in $B$.}
\label{fig7}
\end{subfigure}

\begin{subfigure}[b]{\textwidth}
\centering
\includegraphics[scale = 0.45]{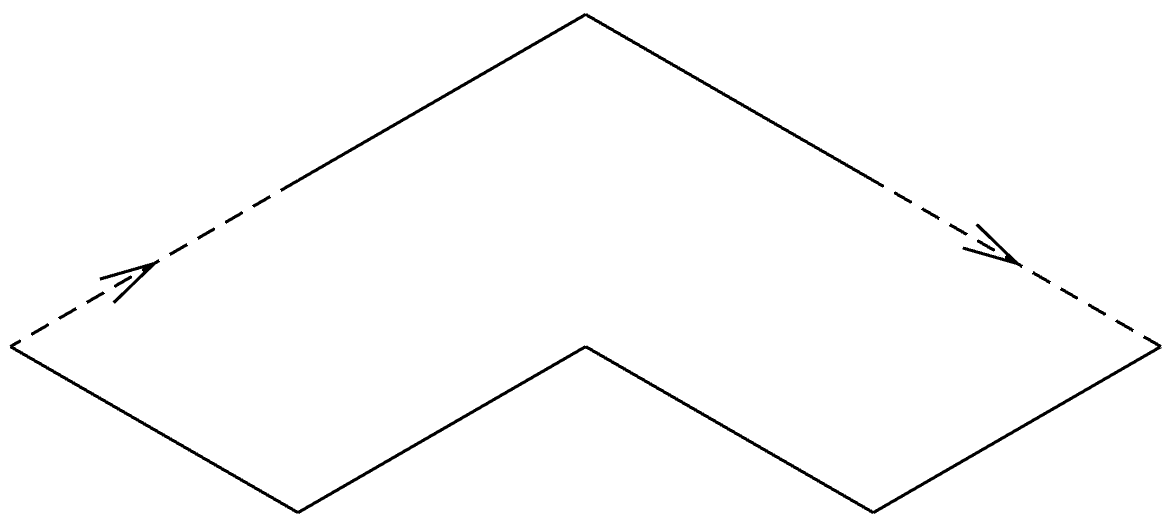}
\caption{A M{\"o}bius strip $B_{\delta}$. The pair of arrows indicate the gluing.}
\label{fig8}
\end{subfigure}
\caption{$\widetilde{B}_\delta$ and its quotient space $B_\delta$.}
\label{fig78}
\end{figure}

By the equivalence relation on $B$, $D_{\delta}^{(k)}$ is identified to $D_{\delta}^{(k-3)}$ for all $k \in \mathbb{Z}$. For instance, $D_{\delta}^{(0)}$ and $D_{\delta}^{(-3)}$, $D_{\delta}^{(-1)}$ and $D_{\delta}^{(-4)}$, $D_{\delta}^{(-2)}$ and $D_{\delta}^{(-5)}$ are identified by the map $q$ in Figure \ref{fig7}.

Now, consider the M{\"o}bius honeycomb tinkertoy $\widetilde{\tau}_{n} = (B, \widetilde{\Gamma}_{n},d)$. Let $\widetilde{h} : V_{\widetilde{\Gamma}_{n}} \rightarrow B$ be a configuration of a tinkertoy $\widetilde{\tau}_{n}$. By definition \eqref{eqn2.6}, $\widetilde{h}$ is required to satisfy
\begin{equation}\label{eqn3.9}
\tag{MH1}
\widetilde{h}({\sf head}(\widetilde{e}))-\widetilde{h}({\sf tail}(\widetilde{e})) \in \{ a \cdot v \in B \mid a \geq 0, v= d(\widetilde{e})\}, \quad \widetilde{e} \in E_{\widetilde{\Gamma}_n}.
\end{equation}
$\widetilde{h}$ is an element of the infinite-dimensional vector space
\begin{equation}\label{eqn3.6}
B^{V_{\widetilde{\Gamma}_n}} =  \prod_{\widetilde{v} \in V_{\widetilde{\Gamma}_{n}}} B(\widetilde{v}), 
\end{equation}
where $B(\widetilde{v}) = B$ for each $\widetilde{v} \in V_{\widetilde{\Gamma}_n}$.

For fixed $\delta \in \mathbb{N}$, $\widetilde{h}$ is required to satisfy additional conditions explained below. Consider $\widetilde{A}_{0,1}, \widetilde{A}_{0,2} ,\cdots , \widetilde{A}_{0,3n}$, which are representatives of equivalence classes of boundary vertices. For instance, in Figure \ref{fig5}, these vertices are on the lowest level, from the right to the left. $\widetilde{h}$ is required to satisfy
\begin{align*}\label{eqn3.7}
\widetilde{h} (\widetilde{A}_{0,j}) & \in \{(-2 \delta , 2 \delta - \xi , \xi) \mid 4\delta \leq \xi \leq 5\delta \}, \quad (1 \leq j \leq n) \\
\tag{MH2}
\widetilde{h} (\widetilde{A}_{0,j}) & \in \{(- \delta ,  \delta - \xi , \xi) \mid 2\delta \leq \xi \leq 3\delta \}, \quad (n+1 \leq j \leq 2n) \\
\widetilde{h} (\widetilde{A}_{0,j}) & \in \{(0 ,  - \xi , \xi) \mid 0 \leq \xi \leq \delta \}, \quad (2n+1 \leq j \leq 3n).
\end{align*}

When $n=5$, for each $1 \leq j \leq 5$, $\widetilde{A}_{0,j}$ should be mapped to the line segment connecting $(-2\delta, -3\delta, 5\delta)$ and $(-2\delta, -2\delta, 4\delta)$, which is in the boundary of $D_\delta^{(-4)}$; see Figure~\ref{fig7}. The cases of $6 \leq j \leq 10$ and $11\leq j \leq 15$ can be interpreted in similar fashion. 

The last condition on $\widetilde{h}$ is
\begin{equation}\label{eqn3.8}
\tag{MH3}
\widetilde{P}_1 \sim \widetilde{P}_2 \hspace{0.5em} \in V_{\widetilde{\Gamma}_n} \quad \Rightarrow \quad \widetilde{h}(\widetilde{P}_1) \sim \widetilde{h}(\widetilde{P}_2) \hspace{0.5em} \in B.
\end{equation}

For fixed $\delta \in \mathbb{N}$, $\widetilde{h} : V_{\widetilde{\Gamma}_{n}} \rightarrow B$ is a \textbf{M{\"o}bius honeycomb} if $\widetilde{h}$ satisfies \eqref{eqn3.9}, \eqref{eqn3.7} and \eqref{eqn3.8}. Denote $\MOB$ as the subset of $B^{V_{\widetilde{\Gamma}_n}}$ consisting M{\"o}bius honeycombs. In \eqref{eqn3.7}, write $\xi_i$ as the $z$-coordinate of $\widetilde{h}(\widetilde{A}_{0,i})$ and define the \textbf{boundary map}
\begin{equation}
\partial : \MOB \rightarrow \R^{3n}, \quad \widetilde{h} \mapsto (\xi_1 , \cdots, \xi_{3n}).
\end{equation}

\begin{theorem}\label{thm3.1}
Let $\lambda,\mu,\nu \in \text{\emph{Par}}_{n}$ and $\delta \in \N$ such that $\delta \geq \lambda_1, \mu_1, \nu_1$. Then $N_{\lambda,\mu,\nu}$ counts the number of M{\"o}bius honeycombs $\widetilde{h} \in \MOB$ satisfying:
\begin{itemize}
\item $\partial (\widetilde{h}) = (\lambda_{1}+4\delta,\cdots, \lambda_{n}+4\delta, \mu_{1}+2\delta, \cdots, \mu_{n}+2\delta, \nu_{1}, \cdots, \nu_{n})$, and
\item $\forall \widetilde{W} \in V_{\widetilde{\Gamma}_n}$, $\widetilde{h} ( \widetilde{W} ) \in B_{\Z}$.
\end{itemize}
\end{theorem}

\begin{figure}
\centering
\includegraphics[scale = 0.42]{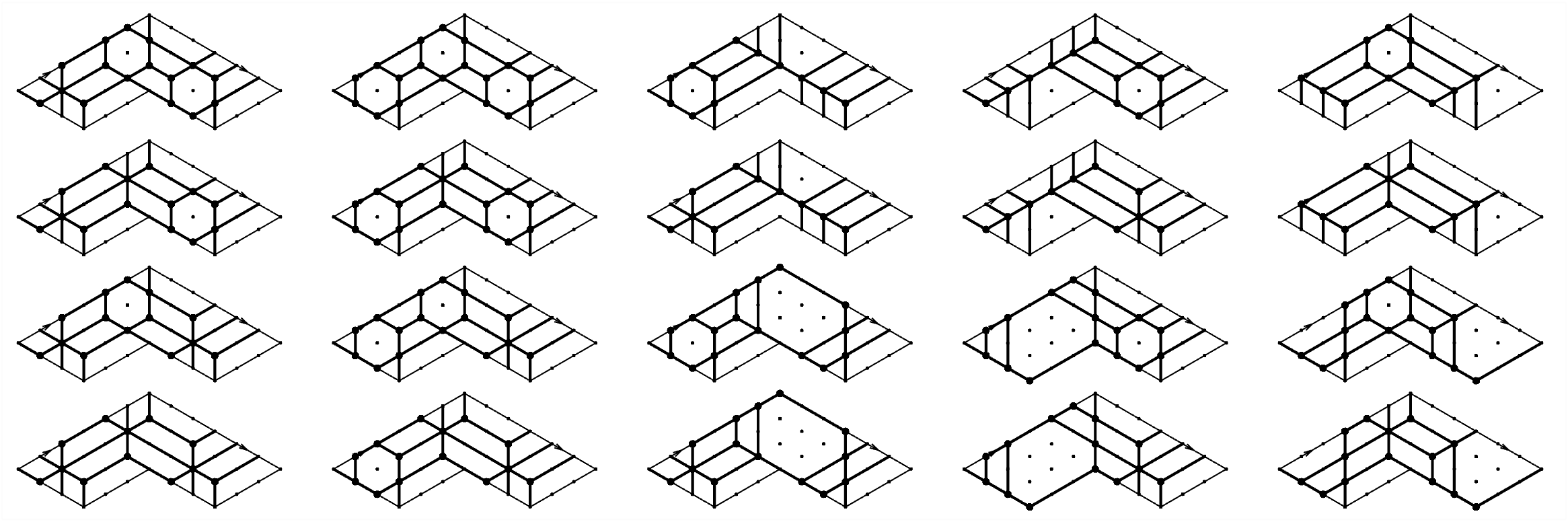}
\caption{$n=3$, $\delta = 3$, $\lambda = \mu = \nu = (3,2,1)$. Then $N_{\lambda, \mu , \nu} = 20$.}
\label{fig57}
\end{figure}

For instance, let $n=3$ and $\lambda = \mu = \nu = (3,2,1)$. Since $\lambda_1 = \mu_1 = \nu_1 =3$, take $\delta =3$. In Figure \ref{fig57}, the number of M{\"o}bius honeycombs satisfying the conditions is $20$. Therefore, $N_{\lambda , \mu , \nu} = 20$. 

\begin{theorem}\label{thm3.2}
Let $\delta \in \mathbb{N}$. Let $\widetilde{h} \in \MOB$ such that $\partial (\widetilde{h}) = (\xi_{1}, \cdots, \xi_{3n}) \in \mathbb{Z}^{3n}$ and $\sum_{1\leq j \leq 3n}\xi_{j} \equiv 0 \modtwo $. Then there exists $\widetilde{g} \in \MOB$ such that:
\begin{itemize}
\item $\partial (\widetilde{g}) = \partial (\widetilde{h})$, and
\item $\forall \widetilde{W} \in V_{\widetilde{\Gamma}_n}$, $\widetilde{g} ( \widetilde{W} ) \in B_{\Z}$.
\end{itemize}
\end{theorem}

We prove Theorem \ref{thm3.1} and Theorem \ref{thm3.2} in Subsection \ref{sub3.4} and Subsection \ref{sub5.4}, respectively.

\begin{proof}[Proof of Theorem \ref{thm1.2}]
Choose $\delta \in \N$ such that $\delta \geq \lambda_1,\mu_1, \nu_1$. Apply Theorem \ref{thm3.1}: from $N_{k\lambda , k\mu , k\nu}>0$, there exists $\widetilde{h} \in {{\tt M \ddot{O} BIUS}(\widetilde{\tau}_{n}, k\delta)}$ satisfying
\begin{equation}
\partial (\widetilde{h}) = (k\lambda_{1}+4k\delta,\cdots, k\lambda_{n}+4k\delta, k\mu_{1}+2k\delta, \cdots, k\mu_{n}+2k\delta, k\nu_{1}, \cdots, k\nu_{n}).
\end{equation}

Due to Lemma \ref{lem3.0}, $\frac{1}{k} \widetilde{h} \in \MOB$ and
\begin{equation}
\partial \left( \frac{1}{k}\widetilde{h} \right) = (\lambda_{1}+4\delta,\cdots, \lambda_{n}+4\delta, \mu_{1}+2\delta, \cdots, \mu_{n}+2\delta, \nu_{1}, \cdots, \nu_{n}).
\end{equation}
In particular, $\partial \left( \frac{1}{k}\widetilde{h} \right) \in \Z^{3n}$ and the sum of components is $|\lambda| + |\mu| + |\nu| + 6n\delta$, which is an even integer. Apply Theorem \ref{thm3.2} to find $\widetilde{g} \in \MOB$ such that
\begin{equation}
\partial \left( \frac{1}{k}\widetilde{h}  \right) = \partial (\widetilde{g}) \hspace{0.5em} \text{and} \hspace{0.5em} \forall \widetilde{W} \in V_{\widetilde{\Gamma}_n}, \hspace{0.5em} \widetilde{g} (\widetilde{W}) \in B_\Z.
\end{equation}

Again, due to the existence of $\widetilde{g} \in \MOB$, $N_{\lambda , \mu , \nu} >0$ follows from Theorem \ref{thm3.1}, completing the proof.
\end{proof}


\subsection{Newell-Littlewood numbers}\label{sub3.4}

In this subsection, we prove Theorem \ref{thm3.1}.

\begin{figure}
\centering
\begin{subfigure}[b]{\textwidth}
\centering
\includegraphics[scale = 0.45]{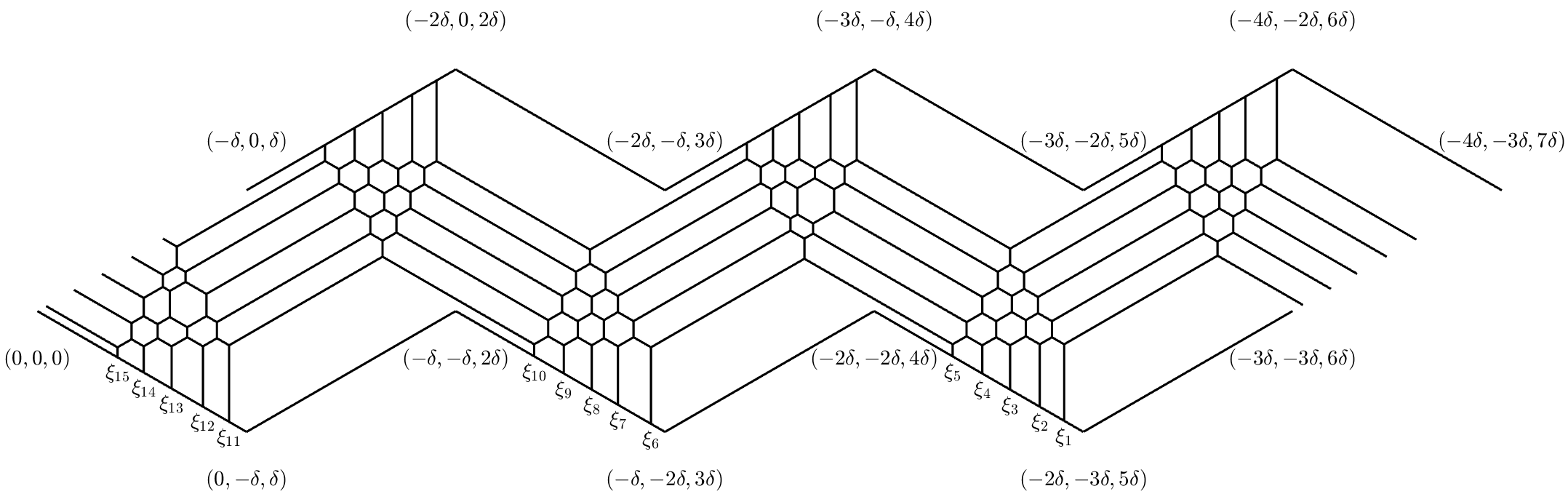}
\caption{Image of $\widetilde{h}$ contained in $\widetilde{B}_{\delta}$ when $n=5$.}
\label{fig9}
\end{subfigure}

\begin{subfigure}[b]{\textwidth}
\centering
\includegraphics[scale = 0.45]{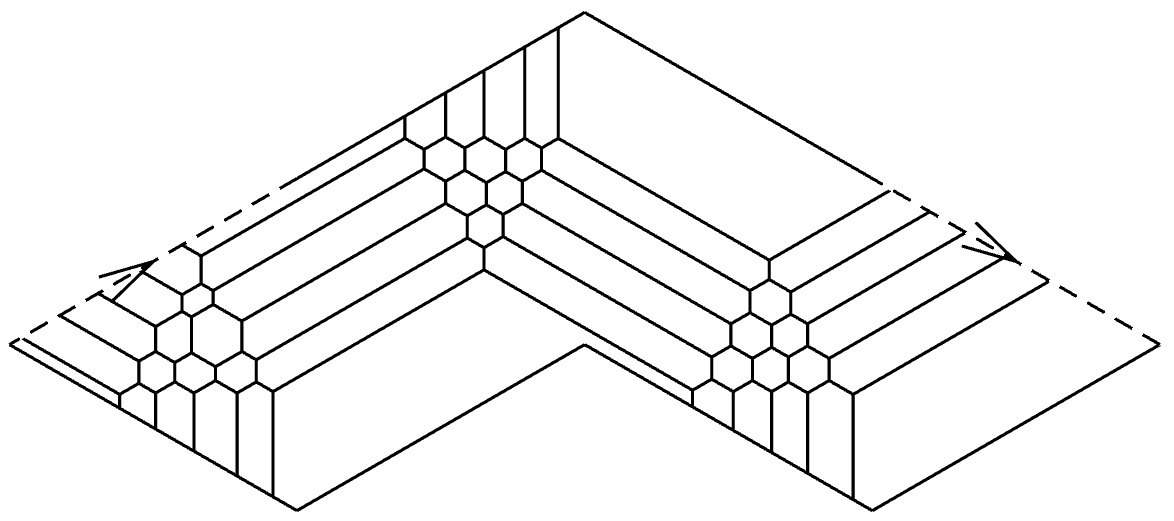}
\caption{Image of $h$ contained in $B_{\delta}$ when $n=5$.}
\label{fig11}
\end{subfigure}
\caption{$\widetilde{h}$ and its associated map $h$.}
\label{fig911}
\end{figure}

\begin{figure}
\centering
\begin{subfigure}[t]{0.35\textwidth}
\centering
	\includegraphics[width=\textwidth]{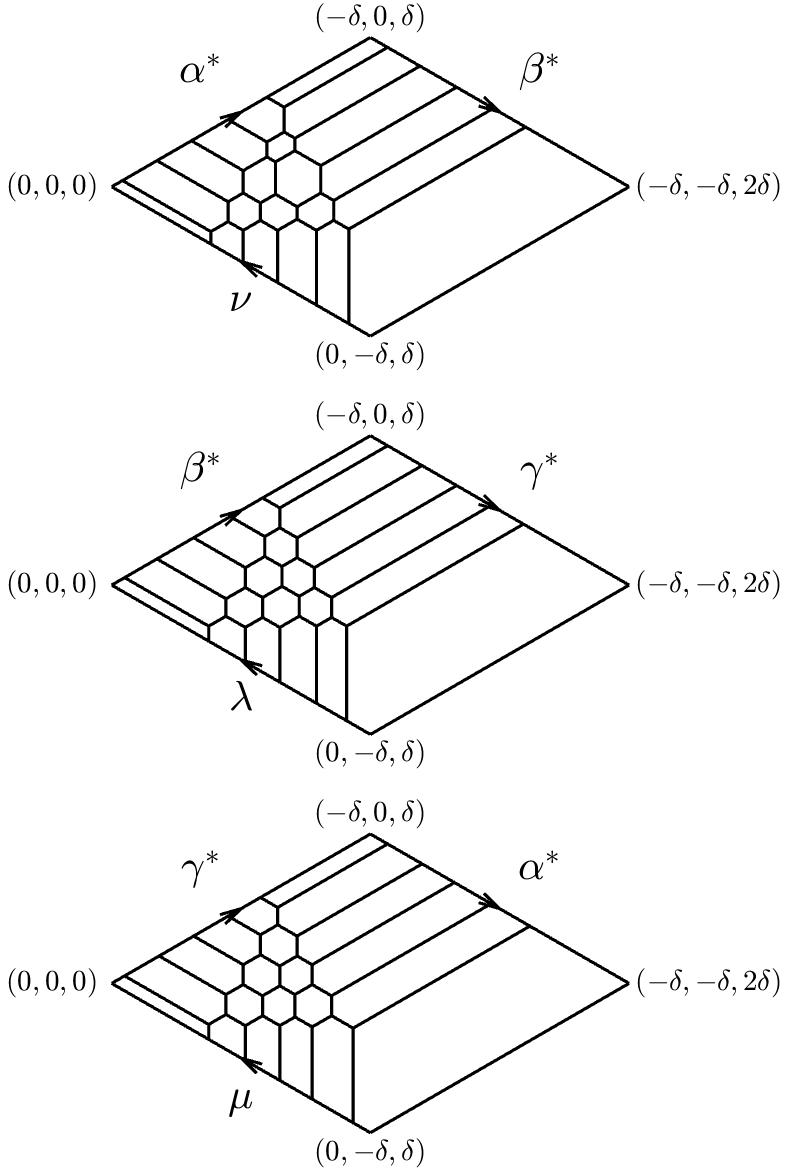}
	\caption{Image of $h_{\lambda},h_{\mu},h_{\nu}$ contained in $D_{\delta}^{(0)}$ when $n=5$.}
	\label{fig12}
\end{subfigure}
\hfill
\begin{subfigure}[t]{0.61\textwidth}
\centering
\includegraphics[width=\textwidth]{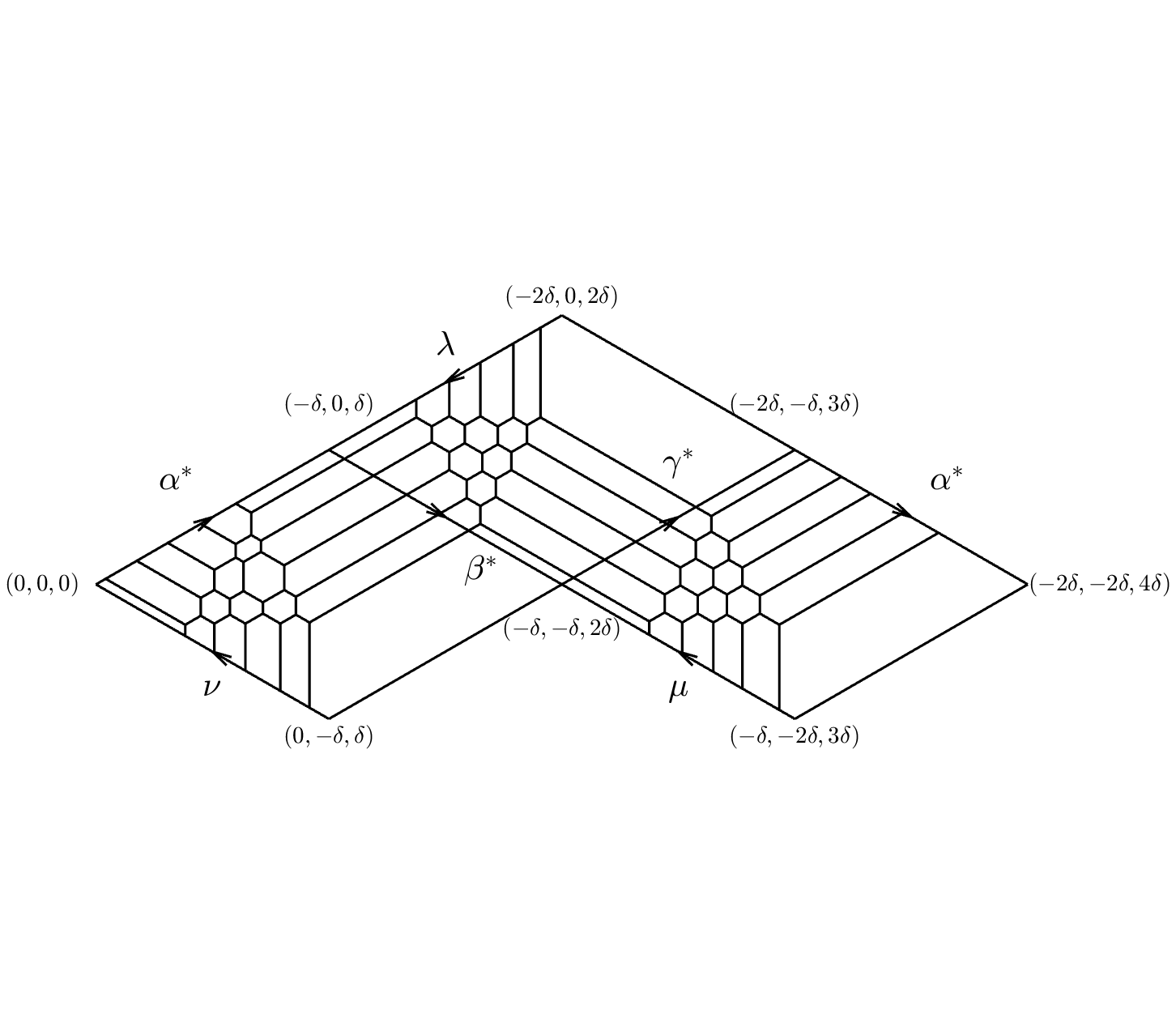}
	\caption{Gluing $h_{\lambda},h_{\mu},h_{\nu}$ to obtain $\widetilde{h}$ when $n=5$.}
	\label{fig13}
\end{subfigure}
\caption{Honeycombs and M{\"o}bius honeycombs.}
\label{fig1213}
\end{figure}

\begin{proof}[Proof of Theorem \ref{thm3.1}]
By Theorem \ref{thm1.2}, $c_{\beta ,\gamma}^{\lambda}c_{\gamma , \alpha}^{\mu} c_{\alpha, \beta}^{\nu}$ is the number of ordered triples $(h_{\lambda}, h_{\mu}, h_{\nu})$ satisfying:
\begin{itemize}
\item $h_{\lambda} , h_{\mu}, h_{\nu} \in \texttt{HONEY}(\tau_{n})$,
\item $\partial (h_{\lambda}) = (\beta^{*},\gamma^{*},\lambda)$, $\partial (h_{\mu}) = (\gamma^{*}, \alpha^{*},\mu)$, $\partial (h_{\nu}) = (\alpha^{*}, \beta^{*},\nu)$, and
\item $\forall v \in V_{\Delta_n}$, $h_\lambda (v) , h_\mu (v) , h_\nu (v) \in B_\Z$.
\end{itemize}
If $c_{\beta ,\gamma}^{\lambda}c_{\gamma , \alpha}^{\mu} c_{\alpha, \beta}^{\nu} \neq 0$, then $\delta \geq \alpha_1, \beta_1, \gamma_1$ follows from $\delta \geq \lambda_1, \mu_1, \nu_1$. As a result,
\begin{equation}\label{eqn3.31}
\forall v \in V_{\Delta_n}, \quad h_\lambda (v) , h_\mu (v) , h_\nu (v) \in D_\delta^{(0)}.
\end{equation}
This is depicted in Figure \ref{fig12}. We have infinite copies of three different types of rhombi depicted in Figure \ref{fig12}. Each type of rhombi is arranged in $B$ as follows.
\begin{itemize}
\item $h_\lambda$ rhombus: $\cdots , D_{\delta}^{(-4)}, D_{\delta}^{(-1)}, D_{\delta}^{(2)}, D_{\delta}^{(5)}, \cdots$
\item $h_\mu$ rhombus: $\cdots , D_{\delta}^{(-5)}, D_{\delta}^{(-2)}, D_{\delta}^{(1)}, D_{\delta}^{(4)}, \cdots$
\item $h_\nu$ rhombus: $\cdots , D_{\delta}^{(-6)}, D_{\delta}^{(-3)}, D_{\delta}^{(0)}, D_{\delta}^{(3)}, \cdots$
\end{itemize}

Gluing pieces along the line segments $\alpha^{*}, \beta^{*}$ and $\gamma^{*}$, we have $\widetilde{h} \in \MOB$ satisfying the given conditions. Therefore, the number of $\widetilde{h} \in \MOB$ satisfying the given conditions is greater than or equal to $N_{\lambda,\mu,\nu}$.

Conversely, suppose we have $\widetilde{h} \in \MOB$ satisfying the given conditions. Then the image of vertices under $\widetilde{h}$ can be depicted, for instance, in Figure \ref{fig9}. Slice $\widetilde{B}_{\delta}$ into $D_{\delta}^{(k)}$. Due to Lemma \ref{lem3.1}, we can retrieve Figure \ref{fig12} from Figure \ref{fig9}. In other words, we construct $h_\lambda , h_\mu , h_\nu \in \HON$ from $\widetilde{h} \in \MOB$, satisfying \eqref{eqn3.31}. Since the images of vertices under $\widetilde{h}$ lie on the lattice points, so do the images under $h_\lambda , h_\mu , h_\nu$. Also, each component of $\alpha,\beta,\gamma$ is non-negative, due to \eqref{eqn3.31}. Therefore, $\alpha ,\beta, \gamma \in \Parn$. This proves that $N_{\lambda, \mu, \nu}$ is greater than or equal to the number of $\widetilde{h} \in \MOB$ satisfying the given conditions.
\end{proof}


\section{Largest-lifts}\label{sec4}

Recall that Theorem \ref{thm1.1} follows from Theorem \ref{thm2.2}, asserting there exists $g \in \HON$ of which image contained in $B_\Z$. To find such a $g$, A.~Knutson and T.~Tao identified a section of $\HON$ as a convex polytope embedded in a finite dimensional vector space equipped with a linear functional. $g$ is chosen as the unique maximum in that polytope with respect to the linear functional. They proved that the image of $g$ has a simple pattern, concluding that it is contained in $B_\Z$. They called $g$ a \textit{largest-lift}.

In this section, we construct an analogue of largest-lift in $\MOB$ as a candidate for $\widetilde{g}$ in Theorem \ref{thm3.2}.\footnote{From now on, we assume $\delta \in \N$ without saying so.} Then we construct a two-colored graph by using edge contraction on $\Gamma_n$, based on $\widetilde{g}$. Using this graph, we prove that the image of $\widetilde{g}$ also has a simple pattern, analogous to \citep[Section 5]{Tao99}.


\subsection{Construction of largest-lifts}

A \textbf{hexagon} $\widetilde{\alpha}_{i,j}$ in $\widetilde{\Gamma}_n$ is
\begin{subequations}\label{eqn2.11}
\begin{equation}
\widetilde{\alpha}_{i,j}:= \{ \widetilde{A}_{i,j}, \widetilde{B}_{i,j}, \widetilde{A}_{i+1,j+1},\widetilde{B}_{i,j+1},\widetilde{A}_{i,j+1},\widetilde{B}_{i-1,j} \} \subseteq V_{\widetilde{\Gamma}_n}.
\end{equation}
For a depiction, see the left-hand picture of Figure \ref{fig2}. The set of hexagons in $\widetilde{\Gamma}_n$ is
\begin{equation}
H_{\widetilde{\Gamma}_n}:=\{\widetilde{\alpha}_{i,j} \mid 1 \leq i \leq n-1, j \in \Z\}.
\end{equation}
\end{subequations}

Similarly, define a hexagon $\alpha_{i,j}$ in $\Gamma_n$
\begin{subequations}\label{eqn2.12}
\begin{equation}
\alpha_{i,j}:=\left\{ p_v(\widetilde{A}_{i,j}), p_v(\widetilde{B}_{i,j}), p_v(\widetilde{A}_{i+1,j+1}), p_v(\widetilde{B}_{i,j+1}),p_v(\widetilde{A}_{i,j+1}), p_v(\widetilde{B}_{i-1,j}) \right\} .
\end{equation}
The set of hexagons in $\Gamma_n$ is 
\begin{equation}
H_{\Gamma_n}:=\{ \alpha_{i,j} \mid 1 \leq i\leq n-1, 1 \leq j \leq n+i \}.
\end{equation}
\end{subequations}
Define
\begin{align}
p_h : H_{\widetilde{\Gamma}_n} & \rightarrow H_{\Gamma_n}, \\
\{ \widetilde{A},\widetilde{B}, \widetilde{C}, \widetilde{D},\widetilde{E},\widetilde{F} \} & \mapsto \{ p_v(\widetilde{A}), p_v(\widetilde{B}), p_v(\widetilde{C}), p_v(\widetilde{D}), p_v(\widetilde{E}), p_v(\widetilde{F}) \} . 
\end{align}

For $(x,y,z),  (x',y',z') \in B$, define a metric $l$ in $B$
\begin{equation}\label{eqn3.20}
l \left( (x,y,z),(x',y',z') \right) := \frac{1}{\sqrt{2}} \sqrt{(x-x')^2 + (y-y')^2 + (z-z')^2}.
\end{equation}
The metric $l$ is scaled so that the distance between consecutive lattice points is $1$. Suppose $\widetilde{h} \in \MOB$ and $\widetilde{e} \in E_{\widetilde{\Gamma}_n}$. Let
\begin{equation}\label{eqn3.26}
{\sf length}(\widetilde{h}; \widetilde{e}):= l\left( \widetilde{h}({\sf head}(\widetilde{e})), \widetilde{h}({\sf tail}(\widetilde{e}) )  \right);
\end{equation}
${\sf length}$ measures each line segment in Figure \ref{fig44}. From Lemma \ref{lemZ.6},
\begin{equation}\label{eqn4.4}
p_e(\widetilde{e}) = p_e(\tilde{e}') \quad \Rightarrow \quad {\sf length}(\widetilde{h}; \widetilde{e}) = {\sf length}(\widetilde{h}; \tilde{e}').
\end{equation}

Let $\widetilde{e}_1, \cdots , \widetilde{e}_6$ be six edges surrounding $\widetilde{\alpha} \in H_{\widetilde{\Gamma}_n}$. Define
\begin{equation}\label{eqn3.28}
{\sf perimeter}(\widetilde{h} ; \widetilde{\alpha}):= \sum_{i=1}^6 {\sf length}(\widetilde{h} ; \widetilde{e}_i).
\end{equation}
From \eqref{eqn4.4},
\begin{equation}\label{eqn4.5}
p_h(\widetilde{\alpha}) = p_h(\widetilde{\alpha}') \quad \Rightarrow \quad {\sf perimeter}(\widetilde{h}; \widetilde{\alpha}) = {\sf perimeter}(\widetilde{h}; \widetilde{\alpha}').
\end{equation}

\begin{lemma}\label{lemA.1}
The following map is an injection:
\begin{equation}\label{eqnA.1}
\iota : \MOB \rightarrow \R^{\frac{3}{2}n(n-1)} \times \R^{3n}, \quad \widetilde{h}  \mapsto \left( (p_{i,j})_{1 \leq i \leq n-1, \hspace{0.2em} 1 \leq j \leq n+i} , (\xi_j)_{1\leq j \leq 3n} \right),
\end{equation}
where $p_{i,j} := {\sf perimeter}(\widetilde{h} ; \widetilde{\alpha}_{i,j})$ and $(\xi_j)_{1\leq j \leq 3n} := \partial \widetilde{h}$.
\end{lemma}
\begin{proof}
See Appendix \ref{secA}.
\end{proof}

\begin{lemma}\label{lem4.4}
$\iota \left( \MOB \right)$ is a convex polytope.
\end{lemma}
\begin{proof}
Using Lemma \ref{lemZ.3}, $\iota$ is extended to a $\R$-linear map between vector spaces. Combined with Corollary \ref{coroC.1}, $\iota (\MOB)$ is determined by finite number of linear equations and inequalities. Convexity follows from Lemma \ref{lem3.0}. The image is bounded by Lemma \ref{lem3.1}.
\end{proof}

Let $w : H_{\Gamma_n} \rightarrow \R$ be a map satisfying a condition as follows: for each $\alpha \in H_{\Gamma_n}$ and $\alpha_1 , \cdots , \alpha_k \in H_{\Gamma_n}$ adjacent to $\alpha$,
\begin{equation}\label{eqn4.6}
w(\alpha) > \frac{1}{6} \left( w(\alpha_1) + \cdots + w(\alpha_k) \right).
\end{equation}
In other words, $w$ assigns real numbers to hexagons so that each number is greater than the average of the surrounding six numbers, as in Figure \ref{fig14}. Indeed such a map exists: let
\begin{equation}\label{eqn4.8}
w: H_{\Gamma_n} \rightarrow \R , \quad \alpha_{i,j} \mapsto i (n-i).
\end{equation}
Then $w$ satisfies \eqref{eqn4.6}.

\begin{figure}
\centering
\includegraphics[scale = 0.52]{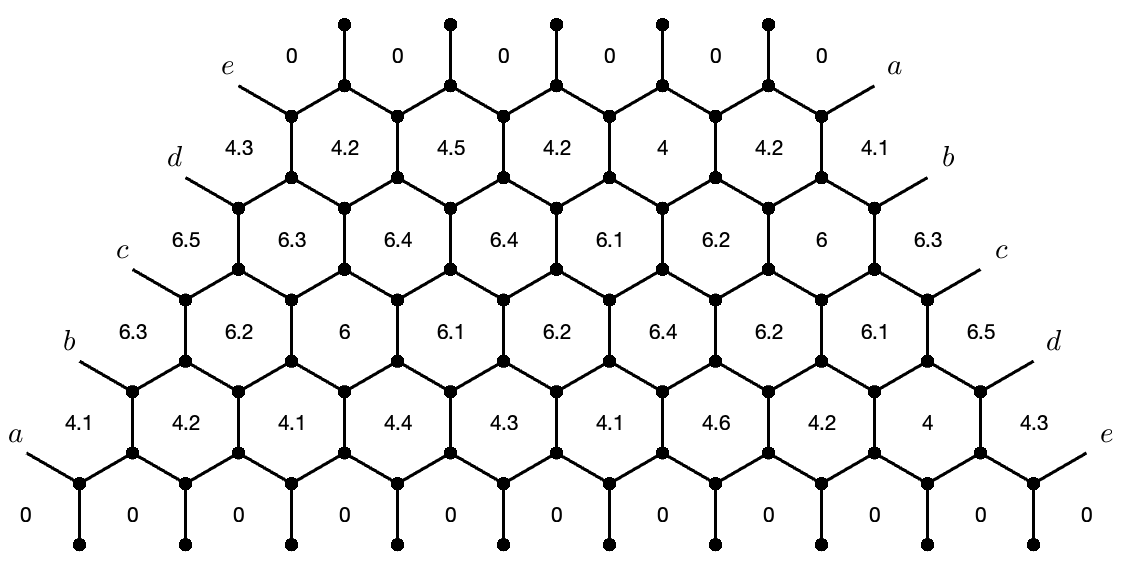}
\caption{Assigning $w$ to the graph $\Gamma_{5}$.}
\label{fig14}
\end{figure}

Using this concept, define a \textbf{weighted perimeter} of $\widetilde{h}$ by
\begin{equation}\label{eqn3.21}
{\sf wperim}(\widetilde{h}) := \sum_{\alpha \in H_{\Gamma_n}} w(\alpha) \times {\sf perimeter} (\widetilde{h} ; \widetilde{\alpha}).
\end{equation}
Here, for each $\alpha \in H_{\Gamma_n}$, $\widetilde{\alpha} \in H_{\widetilde{\Gamma}_n}$ satisfies that $p_h(\widetilde{\alpha}) = \alpha$; this is well-defined due to \eqref{eqn4.5}. 

Denote $\BDR := \partial (\MOB)$. Let $\xi \in \BDR$. We say $\widetilde{h} \in \MOB$ is a \textbf{largest-lift} of $\xi$ if $\partial (\widetilde{h}) = \xi$ and there exists $w : H_{\Gamma_n} \rightarrow \R $ and corresponding weighted perimeter so that
\begin{equation}
\widetilde{g} \in \MOB \text{ such that }  \partial (\widetilde{h}) = \partial (\widetilde{g}), \hspace{0.5em} \widetilde{h} \neq \widetilde{g} \quad \Rightarrow \quad {\sf wperim}(\widetilde{h}) > {\sf wperim} (\widetilde{g}) .
\end{equation}

\begin{lemma}\label{lem3.3}
Let $\xi \in \BDR$. Then there exists a largest-lift of $\xi$.
\end{lemma}

\begin{proof}
By Lemma \ref{lem4.4},
\begin{equation}\label{eqn4.7}
\iota (\MOB) \cap \left( \R^{\frac{3}{2}n(n-1)} \times \{ \xi \} \right)
\end{equation}
is a convex polytope. Regarding $\vec w=(w(\alpha_{i,j}))_{\alpha_{i,j} \in H_{\Gamma_n}}$ as a vector in $\R^{\frac{3}{2}n(n-1)}$, ${\sf wperim}$ defines the ``height'' of the elements in \eqref{eqn4.7} \emph{via} the dot product with the vector $\vec w$.

Notice there exists a sufficiently small open neighborhood $O$ of the $\vec w \in \R^{\frac{3}{2}n(n-1)}$ in \eqref{eqn4.8}  such that any $\vec w'\in O$ satisfies \eqref{eqn4.6}. Therefore, we may assume $\vec w$ so that there is the unique ``highest'' element in \eqref{eqn4.7}, with respect to $\vec w$. Due to Lemma \ref{lemA.1}, $\iota$ is injective, proving that the ``highest'' element corresponds to a largest-lift under $\iota$.
\end{proof}

Suppose in the image of $\widetilde{h} \in \MOB$, we have a hexagon $\widetilde{\alpha}$ depicted in Figure \ref{fig15} with six surrounding edges of $\widetilde{\alpha}$ and six ``spoke'' edges of positive lengths. Let $g:= w(p_h (\widetilde{\alpha} ))$. Similarly, let $a,b,c,d,e,f$ be the assigned values of surrounding hexagons of $\widetilde{\alpha}$. Clearly, one can inflate the image of $\widetilde{\alpha}$ by a sufficiently small $\epsilon >0$, as in Figure \ref{fig15}, and obtain another $\widetilde{h}' \in \MOB$. Then the perimeter of $\widetilde{\alpha}$ increases by $6\epsilon$, whereas those of surrounding hexagons decrease by $\epsilon$, respectively. That is,
\begin{equation}\label{eqn4.9}
{\sf wperim}(\widetilde{h}') - {\sf wperim}(\widetilde{h}) = 6\epsilon g - \epsilon (a+ b + c+ d+ e+ f).\footnote{This comes from \citep[Lemma 9]{Tao99}.}
\end{equation}
Since $w$ is assigned in \eqref{eqn4.6} so that $g > \frac{1}{6}(a+b+c+d+e+f)$, we have ${\sf wperim} (\widetilde{h}) < {\sf wperim} (\widetilde{h}')$. In short, inflating a hexagon increases the value of ${\sf wperim}$.

We now formulate inflation rigorously. Let $(\xi_i , p_{i,j}):= \iota (\widetilde{h})$. For each $\alpha \in H_{\Gamma_n}$, define
\begin{equation}
\xi_j':=\xi_j , \quad p_{i,j}':=
\begin{cases}
p_{i,j} + 6\epsilon &  \alpha_{i,j} = \alpha \\
p_{i,j}-\epsilon & \alpha_{i,j} \text{ is adjacent to }\alpha \\
p_{i,j} & \text{otherwise}
\end{cases}
.
\end{equation}
If $(\xi_j ' , p_{i,j}' ) \in \iota (\MOB)$, let $\widetilde{h}' := \iota^{-1}(\xi_j' , p_{i,j}')$. Then we say that $\widetilde{h}' $ is obtained from $\widetilde{h}$ by \textbf{inflating a hexagon $\alpha$} of $\Gamma_n$ by $\epsilon$.

\begin{figure}
\centering
\includegraphics[scale = 0.35]{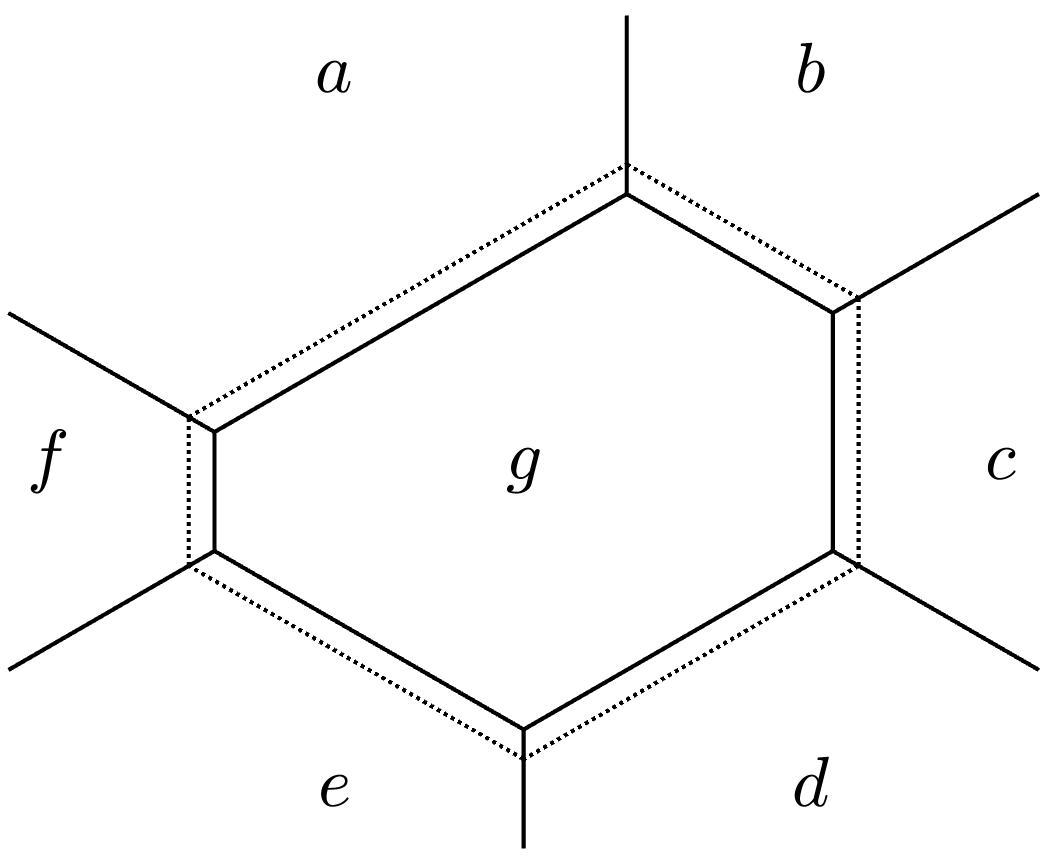}
\caption{Inflating a hexagon increases the value of ${\sf wperim}$.}
\label{fig15}
\end{figure}

\begin{lemma}\label{lem3.4}
Let $\widetilde{h}$ be a largest-lift of $\xi \in \BDR$. Then
\begin{itemize}
\item If $\widetilde{h}_1, \widetilde{h}_2 \in \partial^{-1}(\xi)$, $c_1 , c_2 \in \R_{\geq 0}$ such that $c_1 + c_2 =1$ and $\widetilde{h} = c_1 \cdot \widetilde{h}_1 + c_2 \cdot \widetilde{h}_2$, then $\widetilde{h} = \widetilde{h}_1$ or $\widetilde{h} = \widetilde{h}_2$.
\item No hexagon can be inflated to obtain another M{\"o}bius honeycomb from $\widetilde{h}$.
\end{itemize}
\end{lemma}

\begin{proof}
Suppose $\widetilde{h} \neq \widetilde{h}_1$ and $\widetilde{h} \neq \widetilde{h}_2$. Due to Lemma \ref{lemZ.3}, for each $\widetilde{\alpha} \in H_{\widetilde{\Gamma}_n}$
\begin{equation}
{\sf perimeter} (\widetilde{h} ; \widetilde{\alpha}) = c_1 \cdot {\sf perimeter} (\widetilde{h}_1 ; \widetilde{\alpha}) + c_2 \cdot {\sf perimeter} (\widetilde{h}_2 ; \widetilde{\alpha}).
\end{equation}
Therefore,
\begin{equation}
{\sf wperim}(\widetilde{h}) = c_1 \cdot {\sf wperim}(\widetilde{h}_1) + c_2 \cdot {\sf wperim}(\widetilde{h}_2).  
\end{equation}
However, since $\widetilde{h}$ is a largest-lift, ${\sf wperim}(\widetilde{h}) > {\sf wperim}(\widetilde{h}_i)$ for $i =1,2$, leading to contradiction. Hence, $\widetilde{h} = \widetilde{h}_1$ or $\widetilde{h} = \widetilde{h}_2$.

If a hexagon in the image of $\widetilde{h}$ can be inflated, then the value of ${\sf wperim}$ increases due to \eqref{eqn4.9}. Note that inflating a hexagon does not change the boundary vertices. Hence, no hexagon of $\Gamma_n$ can be inflated as in Figure \ref{fig15} to obtain another M{\"o}bius honeycomb.
\end{proof}

We conclude that in the image of a largest-lift $\widetilde{h}$, there is no such hexagon as in Figure \ref{fig15}. In other words, one ``spoke'' edge has zero length, making it impossible to inflate the hexagon in the middle.


\subsection{Coloring}\label{sub4.2}

While the image of a M{\"o}bius honeycomb consists of infinite copies of M\"obius strips, as depicted
in Figure \ref{fig414344}, it is sufficient to deal with just one of them. That is, for each $\widetilde{h} \in \MOB$, define
\begin{equation}
h : V_{\Gamma_n} \rightarrow B_\delta, \quad W \mapsto (q \circ \widetilde{h} \circ  p_v^{-1})(W) .
\end{equation}
This is well-defined because of \eqref{eqn3.8}. We call $h$ the \textbf{associated map} of $\widetilde{h}$. 

We call an edge $e \in E_{\Gamma_n}$ \textbf{degenerate} if its endpoints are mapped to the same point in $B_\delta$ under $h$. Also, we call a vertex $W \in V_{\Gamma_n}$ \textbf{degenerate} if $W$ is one of the endpoints of a degenerate edge. 

Suppose we are given $\Gamma_n$ as input. Contract each degenerate edge $e \in E_{\Gamma_n}$,
\emph{i.e.}, delete $e$ and merge its endpoints. The resulting graph may have multiple edges. Next, merge such multiple edges into a single edge. This procedure outputs a simple graph $\Gamma_n (h)$. It makes sense to define the identification map between vertices
\begin{equation}
\rho_h : V_{\Gamma_n} \rightarrow V_{\Gamma_n (h)}.
\end{equation}
If an edge of $e \in E_{\Gamma_n (h)}$ was obtained by merging $m\geq 1$ number of edges of $\Gamma_n$, we say $e$ has \textbf{multiplicity} $m$. 

$W \in V_{\Gamma_n (h)}$ is called \textbf{boundary vertex} of $\Gamma_n (h)$ if $W = \rho_h(U)$ for some $U$ which is a boundary vertex of $\Gamma_n$.

\begin{lemma}\label{lem4.5}
Let $\widetilde{h} \in \MOB$ and its associated map $h : V_{\Gamma_n} \rightarrow B_\delta$ be defined. If $W$ is not a boundary vertex of $\Gamma_n (h)$, then it is one of five types in Figure \ref{fig16}, up to rotation, each number denoting multiplicities of adjoining edges.
\end{lemma}
\begin{proof}
Figure \ref{fig16} is identical to \citep[Figure 9]{Tao99}. According to \citep[Lemma 3]{Tao99}, the \textit{diagram} of a honeycomb is the image of the honeycomb in the vector space, which are one of the types in Figure \ref{fig16}. This is, in fact, a distortion of $\Gamma_n(h)$, which means that the vertices of $\Gamma_n (h)$ are also classified by Figure \ref{fig16}. More specifically, by applying edge contraction on all non-boundary edges in \citep[Figure 8]{Tao99}, we have Figure \ref{fig16}. Here, we omit the cases when the vertex is a boundary vertex, since the degree of a boundary vertex of $\Gamma_n$ is $1$, unlike \citep[Figure 8]{Tao99}.
\end{proof}

\begin{figure}
\centering
\includegraphics[scale=0.7]{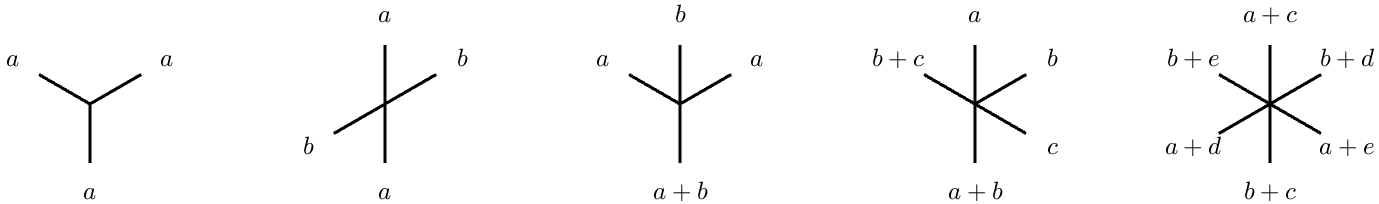}
\caption{The \textbf{Y}, crossing, rake, 5-valent, and 6-valent vertices.}
\label{fig16}
\end{figure}

Let $\widetilde{h} \in \MOB$. For each edge $\widetilde{e} \in E_{\widetilde{\Gamma}_n}$, there exists $a \in \R$ such that a line segment from $\widetilde{h}({\sf tail}(\widetilde{e}))$ to $\widetilde{h}({\sf head}(\widetilde{e}))$ is contained in one of the lines in \eqref{eqn2.3} due to \eqref{eqn3.9}. Write ${\sf const}(\widetilde{h} ; \widetilde{e}) = a$. More specifically, suppose $\widetilde{W}$ is an endpoint of $\widetilde{e}$ and 
let $(x,y,z)=\widetilde{h}(\widetilde{W})$. Then set
\begin{equation}\label{eqnA.2}
{\sf const}(\widetilde{h} ; \widetilde{e}):= \left\{
\begin{array}{ll}
x & \text{if }d(\widetilde{e})=(0,-1,1) \\
y & \text{if }d(\widetilde{e})=(1,0,-1).  \\
z & \text{if }d(\widetilde{e})=(-1,1,0)
\end{array}
\right. 
\end{equation}
Here, ${\sf const}(h; \widetilde{e})$ is well-defined regardless of which endpoint is chosen. Let $\widetilde{W}$ be a vertex of $\widetilde{\Gamma}_n$.
\begin{itemize}
\item If $\widetilde{h}(\widetilde{W}) \notin B_\Z$, then color $\widetilde{W}$ white.
\item Otherwise, color $\widetilde{W}$ black.
\end{itemize}
Let $\widetilde{e}$ be an edge of $\widetilde{\Gamma}_n$.
\begin{itemize}
\item If ${\sf const}( \widetilde{h} ; \widetilde{e}) \notin \Z$, then color $\widetilde{e}$ white.
\item Otherwise, color $\widetilde{e}$ black.
\end{itemize}

If $p_v(\widetilde{W}_1) = p_v(\widetilde{W}_2)$, then $\widetilde{W}_1$ and $\widetilde{W}_2$ are in the same color due to Lemma \ref{lemZ.1}. Similarly, if $p_e(\widetilde{e}_1) = p_e(\widetilde{e}_2)$, then $\widetilde{e}_1$ and $\widetilde{e}_2$ are in the same color due to Lemma \ref{lemZ.5}. Therefore, we can define coloring on $\Gamma_n$ from $\widetilde{\Gamma}_n$ as follows.

\begin{itemize}
\item For each vertex $W$ in $\Gamma_n$, color it the same as $\widetilde{W}$ where $p_v(\widetilde{W}) = W$. 
\item For each edge $e$ in $\Gamma_n$, color it the same as $\widetilde{e}$ where $p_e(\widetilde{e}) = e$. 
\end{itemize}
In the left-hand picture of Figure \ref{fig36}, the vertices and edges of $\Gamma_n$ are colored in black and white. 

\begin{figure}
\centering
\includegraphics[scale=0.33]{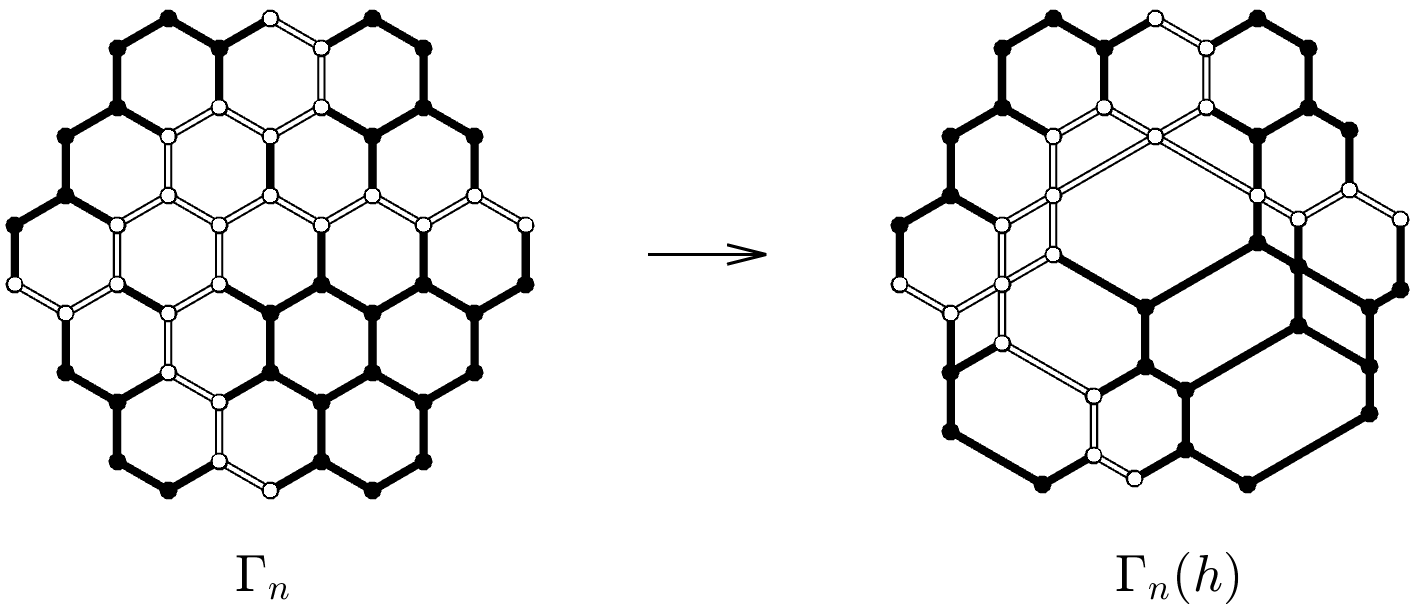}
\caption{Coloring the graph $\Gamma_n$ and $\Gamma_n (h)$.}
\label{fig36}
\end{figure}

Let $h: V_{\Gamma_n} \rightarrow B_\delta$ be the associated map of $\widetilde{h}$. Here, vertices and edges are merged together whenever they are mapped to the same elements under $h$. Since the coloring is based on the image of $\widetilde{h}$, vertices and edges are in the same color before they are merged together. Hence, the coloring of $\Gamma_n$ induces a 
coloring of vertices and edges in $\Gamma_n (h)$, as in Figure \ref{fig36}.

White vertices are precisely those in $\widetilde{\Gamma}_n$ not mapped to $B_\Z$ under $\widetilde{h}$. We view this as a deficiency, and our next step is to study them.

\begin{lemma}\label{lem4.3}
Let $\widetilde{h} \in \MOB$ be chosen so that $\partial \widetilde{h} \in \Z^{3n}$. Suppose $\widetilde{W} \in V_{\widetilde{\Gamma}_n}$ is a white vertex in $\widetilde{\Gamma}_n$. Then $\widetilde{h} (\widetilde{W})$ is in the interior of $\widetilde{B}_{\delta}$. 
\end{lemma}

\begin{proof}
Suppose $\widetilde{h}(\widetilde{W})$ is on the boundary of $\widetilde{B}_{\delta}$. Using Lemma \ref{lemZ.4}, there exists a boundary vertex $\widetilde{W}'$ such that $\widetilde{h}(\widetilde{W}) = \widetilde{h}(\widetilde{W}')$. Therefore, $\widetilde{W}$ and $\widetilde{W}'$ are in the same color. However, the components of $\partial \widetilde{h} = ( \xi_{1}, \cdots, \xi_{3n} )$ are integers and $\delta \in \N$. Hence, the color of boundary vertices is black, a contradiction.
\end{proof}

\begin{lemma}\label{lem4.6}
Let $\widetilde{h} \in \MOB$ be chosen so that $\partial \widetilde{h} \in \Z^{3n}$. Let $h:V_{\Gamma_n} \rightarrow B_\delta$ be the associated map of $\widetilde{h}$. Then after coloring, the only possible cases of white vertices in $\Gamma_n (h)$ are those displayed in Figure \ref{fig46}, up to rotation and reflection.
\end{lemma}

\begin{proof}
First, we claim that all boundary vertices of $\Gamma_n (h)$ are colored in black. Suppose it is not true. Then some boundary vertices of $\Gamma_n$ are colored in white. This contradicts Lemma \ref{lem4.3}, proving our claim. Therefore, a white vertex in $\Gamma_n (h)$ is one of five types in Figure \ref{fig16}, due to Lemma \ref{lem4.5}.

Let $W$ be a white vertex of $\Gamma_n (h)$. Let $\widetilde{W} \in V_{\widetilde{\Gamma}_n}$ be chosen so that $(\rho_h \circ p_v) (\widetilde{W}) = W$. Let $(x,y,z):= \widetilde{h} (\widetilde{W})$. Since $(x,y,z) \notin B_\Z$, at least two of the coordinates are non-integers, due to $x+y+z = 0$. Also, $\widetilde{h}(\widetilde{W})$ is contained in lines $(x, * ,* )$, $(*,y,*)$ and $(*,*,z)$. 

If we assume that $W$ is a 6-valent, these three lines correspond to three lines passing through $W$ in $\Gamma_n (h)$. The constant coordinates $x$, $y$ and $z$ determine the color of the lines. In other words, at least two of the lines are in white, whereas the others are in black. This means there are two cases of 6-valent as in Figure \ref{fig46}. Similarly, we find all possible cases of coloring \textbf{Y}, crossing, rake and 5-valent, as in Figure \ref{fig46}.
\end{proof}

\begin{figure}
\centering
\includegraphics[scale=0.58]{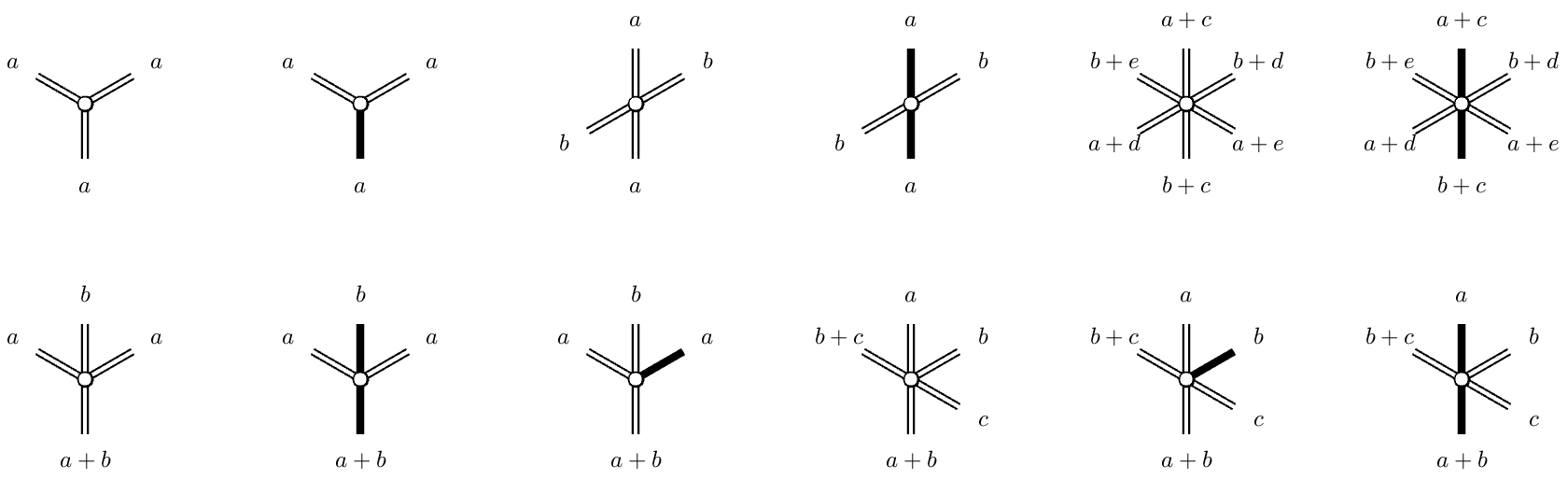}
\caption{Possible cases of white vertex in $\Gamma_n (h)$.}
\label{fig46}
\end{figure}

\begin{theorem}\label{lem4.2}
Let $\widetilde{h} \in \MOB$ be chosen so that $\partial \widetilde{h} \in \Z^{3n}$. Let $h: V_{\Gamma_n} \rightarrow B_\delta$ be the associated map of $\widetilde{h}$. Assume that $\widetilde{h}$ is a largest-lift. Then after coloring, a white vertex of $\Gamma_n(h)$ is one of five types in Figure \ref{fig23}, up to rotation and reflection.\footnote{The lemma, proof, and eventual application is analogous to \citep[Theorem 2]{Tao99}.}
\end{theorem}

\begin{figure}
\centering
\includegraphics[scale=0.6]{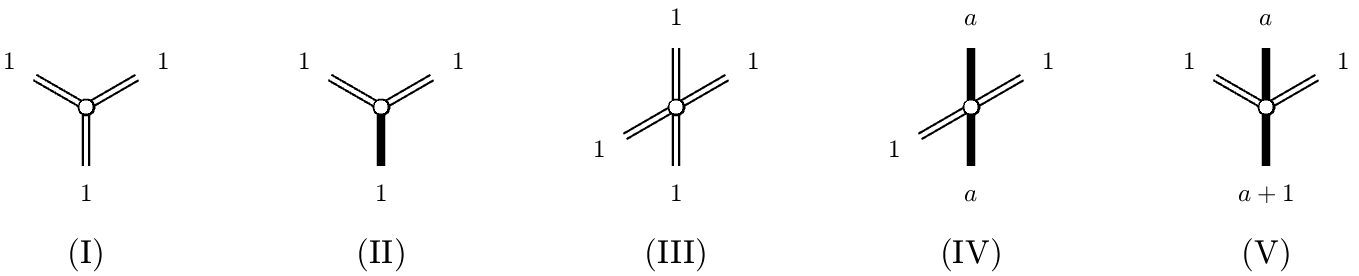}
\caption{Possible cases of white vertex in $\Gamma_n (h)$ when $\widetilde{h}$ is a largest-lift.}
\label{fig23}
\end{figure}
\begin{figure}
\centering
\begin{subfigure}[b]{\textwidth}
\centering
\includegraphics[scale=0.50]{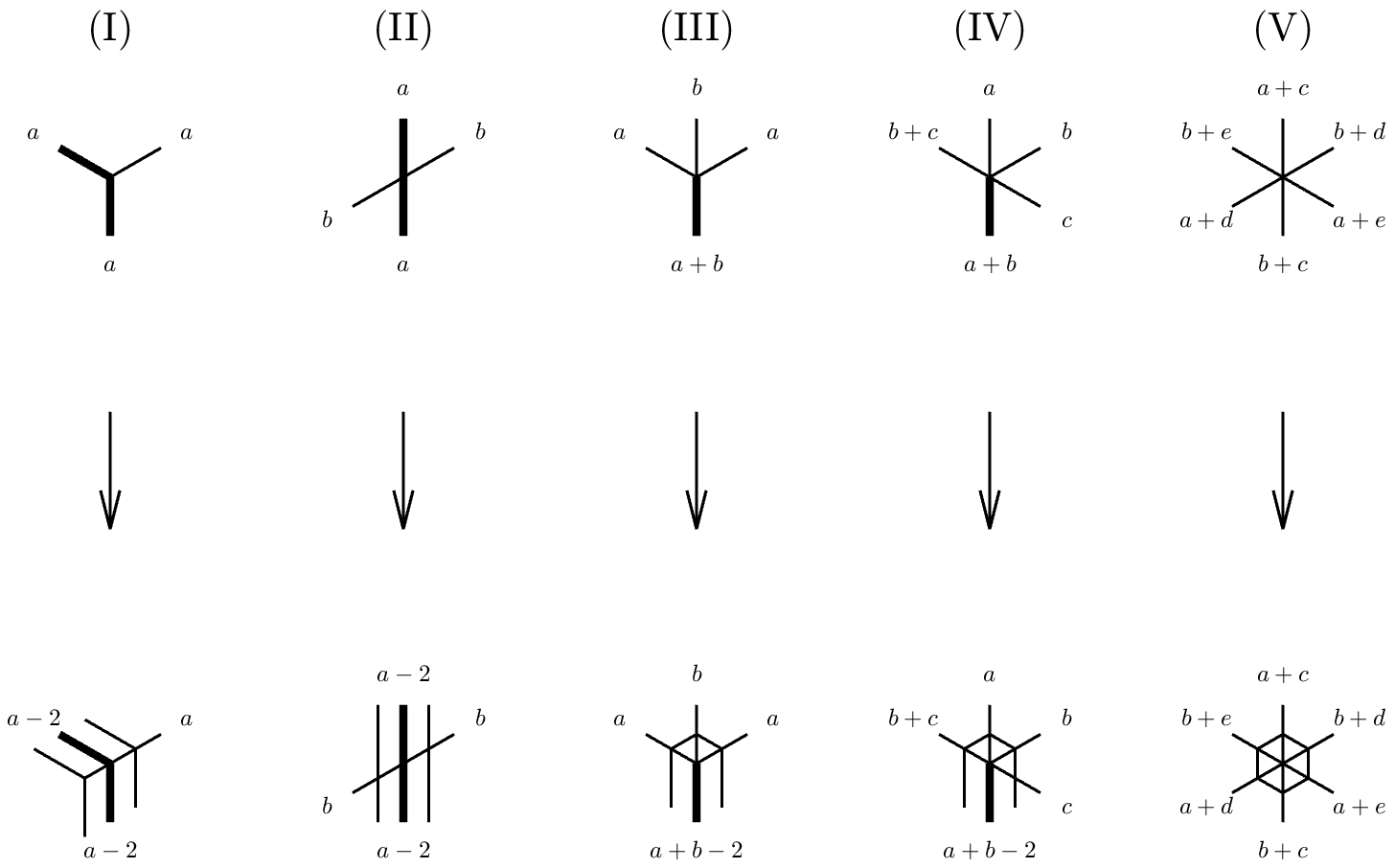}
\caption{Pre-existing molting techniques.}
\label{fig17}
\end{subfigure}

\begin{subfigure}[b]{\textwidth}
\centering
\includegraphics[scale=0.5]{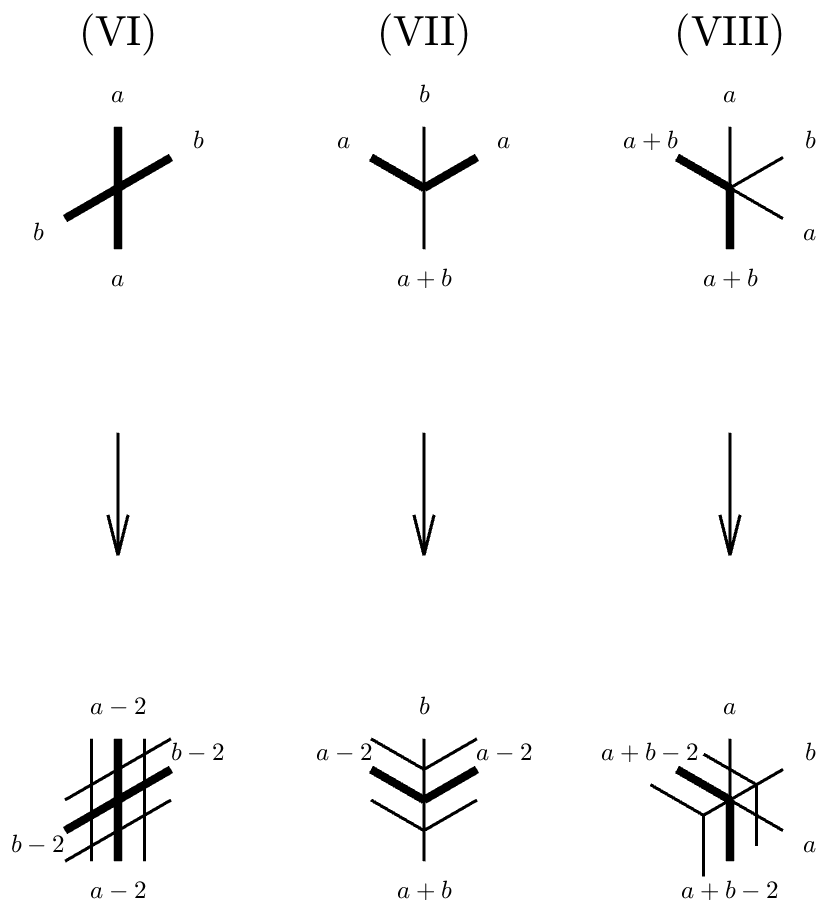}
\caption{New molting techniques.}
\label{fig18}
\end{subfigure}
\caption{Molting techniques.}
\label{fig1718}
\end{figure}

\begin{proof}
Choose $\epsilon >0$ so that
\begin{equation}
0< 2\epsilon < \min \{ {\sf length}(\widetilde{h} ; \widetilde{e}) \mid \widetilde{e} \in E_{\widetilde{\Gamma}_n}, \hspace{0.2em} {\sf length}(\widetilde{h} ; \widetilde{e}) \neq 0\},
\end{equation}
which is possible since there are finite number of values of ${\sf length}(\widetilde{h} ; \widetilde{e})$.

If there is a 6-valent white vertex $W$ in $\Gamma_n (h)$, let
\begin{equation}
H':=\{ \alpha \in H_{\Gamma_n} \mid  \rho_h (\alpha) = \{W\} \}.
\end{equation}
Inflating hexagons in $H'$ simultaneously by $\epsilon$, we have $\widetilde{h}' \in \MOB$, due to \citep[Lemma 10]{Tao99}. This contradicts Lemma \ref{lem3.4}. Hence, there is no 6-valent white vertex in $\Gamma_n (h)$.

We only need to prove that there is no white edges with multiplicity greater than $1$; then Figure \ref{fig23} follows from Figure \ref{fig46}. Let $m$ be the maximal multiplicity of white edges in $\Gamma_n (h)$. Assume that $m \geq 2$. We want to construct a trail\footnote{There is no edge repeated, but a vertex can be repeated.} in $\Gamma_n (h)$ satisfying the conditions as follows.
\begin{itemize}
\item The trail is composed of white vertices and white edges of multiplicity $m$.
\item Each vertex of the trail is one of the types on the top rows of Figures \ref{fig17} and \ref{fig18}, where bold edges are edges of the trail: (I), (II), (III), (IV), (VI), (VII) and (VIII).
\end{itemize}
We construct the trail by extending the endpoints, which are one of the four types: \textbf{Y}, crossing, rake and 5-valent. Let $W$ be an endpoint of the trail.

$\bullet$ $W$ is a \textbf{Y} : Connect another white edge of multiplicity $m$, possible due to Lemma \ref{lem4.6}. Then $W$ becomes the type (I) of Figure \ref{fig17}.

$\bullet$ $W$ is a \textit{crossing} : Connect the edge which is parallel, possible due to Lemma \ref{lem4.6}. If $W$ is chosen for the first time, then it is the type (II) of Figure \ref{fig17}. If it is the second time, then $W$ is the type (VI) of Figure \ref{fig18}.

$\bullet$ $W$ is a \textit{rake} : Then there are three cases: the trail is coming from an edge of multiplicity $a+b$, $a$ or $b$.
\begin{enumerate}
\item If the trail is coming from the edge of multiplicity $a+b$, then stop extending this endpoint. Then $W$ is the type (III) of Figure \ref{fig17}.
\item If the trail is coming from the edge of multiplicity $a$, then $a=m$. Then $a+b >m$, which means that $a+b$-multiplicity edge is in black, due to maximality of $m$. Due to Lemma \ref{lem4.6}, two edges of multiplicity $a$ are all in white, making it possible to extend the trail. Then $W$ is the type (VII) of Figure \ref{fig18}.
\item If the trail is coming from the edge of multiplicity $b$, then this edge is in white. According to Lemma \ref{lem4.6}, $(a+b)$-multiplicity edge is also in white, contradicting the maximality of $m$.
\end{enumerate}

$\bullet$ $W$ is a \textit{5-valent} : Due to Lemma \ref{lem4.6}, one of $(a+b)$-multiplicity edge and $(b+c)$-multiplicity edge is in white. Due to maximality, the trail is coming from one of these edges. Without loss of generality, assume that the trail is coming from $(b+c)$-multiplicity edge. Then there are three cases:

\begin{enumerate}
\item If $(a+b)$-multiplicity edge is in black, stop extending this endpoint of the trail. Then $W$ is the type (IV) of Figure \ref{fig17}.
\item If $(a+b)$-multiplicity edge is in white but $m>a+b$, stop extending this endpoint of the trail. Again, $W$ is the type (IV) of Figure \ref{fig17}.
\item Otherwise, there are two white edges of multiplicity $m$, making it possible to extend the trail. Then $W$ is the type (VIII) of Figure \ref{fig18}.
\end{enumerate} 

Since the graph $\Gamma_n$ is finite, either we have a closed trail, or it is no longer possible to extend the endpoints of the trail, leaving it an open trail. The vertices of the trail are one of the types on the top row of Figure \ref{fig17} and \ref{fig18} where only bold edges are contained in the trail. Write the vertices of the trail $v_0 , v_1 , \cdots , v_n$. If the trail is closed, then $v_0 = v_n$. Define
\begin{equation}\label{eqn4.10}
H_i := \{ \alpha \in H_{\Gamma_n} \mid \rho_h (\alpha) = \{ v_i \} \}, \quad H_{i,i+1} := \{ \alpha \in H_{\Gamma_n} \mid \rho_h (\alpha) = \{ v_i , v_{i+1}\} \}.
\end{equation}

Inflate all the hexagons by $\epsilon$ in $H_i$ and $H_{i,i+1}$ for $0 \leq i \leq n-1$. In addition, if the trail is open, then inflate all the hexagons in $H_n$ by $\epsilon$. It is possible that some hexagons are inflated twice \textit{i.e.} by $2\epsilon$, since $v_i = v_j$ may happen for $1 \leq i < j \leq n-1$. This happens when $v_i = v_j$ is the type (VI) of Figure \ref{fig18}.

Our claim is that by inflating all hexagons in $H_i$ and $H_{i,i+1}$ by $\epsilon$, we have $\widetilde{h}' \in \MOB$. Theorem \ref{lem4.2} immediately follows from the claim: if the claim is true, then it contradicts Lemma \ref{lem3.4}, since $\widetilde{h}' \in \MOB$ is constructed from a largest-lift $\widetilde{h}$ by inflating hexagons.

To prove our claim, we use the argument in \citep[Lemma 10]{Tao99} following four steps.
\begin{enumerate}
\item The image of $\widetilde{h}$ is a distortion of $\Gamma_n(h)$.
\item Choose $\delta >0$ greater than $4\epsilon$. On the image of $\widetilde{h}$, ``expand'' the length of all edges $\widetilde{e} \in E_{\widetilde{\Gamma}_n}$ by $\delta$. Then there is no degenerate edges. The hexagons of $H_{\widetilde{\Gamma}_n}$ corresponding to \eqref{eqn4.10} are marked in the gray region.
\item Inflate the hexagons in the gray region by $\epsilon$. Since $\delta > 4\epsilon$, we do not need to worry about ``spoke'' edges of hexagons being inflated. We check that each length is greater than or equal to $\delta$.
\item Since each length is greater than or equal to $\delta$, it is possible to ``shrink'' the length of all edges. The result is the image of $\widetilde{h}'$, which is a M{\"o}bius honeycomb.
\end{enumerate}

Consider the vertex $v_i$ of the trail, which is one of types in Figure \ref{fig17} and \ref{fig18}: (I), (II), (III), (IV), (VI), (VII) and (VIII). We follow four steps for each type.

$\bullet$ Type (I) : See Figure \citep[Figure 21]{Tao99}, depicting four pictures corresponding to four steps.

$\bullet$ Type (II) : See Figure \citep[Figure 22]{Tao99}.

$\bullet$ Type (III) : See Figure \citep[Figure 23]{Tao99}.

$\bullet$ Type (IV) : A.~Knutson and T.~Tao omitted the picture: see \citep[Lemma 10]{Tao99}.

$\bullet$ Type (VI) : Inflate hexagons as in type (II) twice, once for each direction.

$\bullet$ Type (VII) : See Figure \ref{fig19}. 

\begin{figure}
\centering
\includegraphics[scale=0.30]{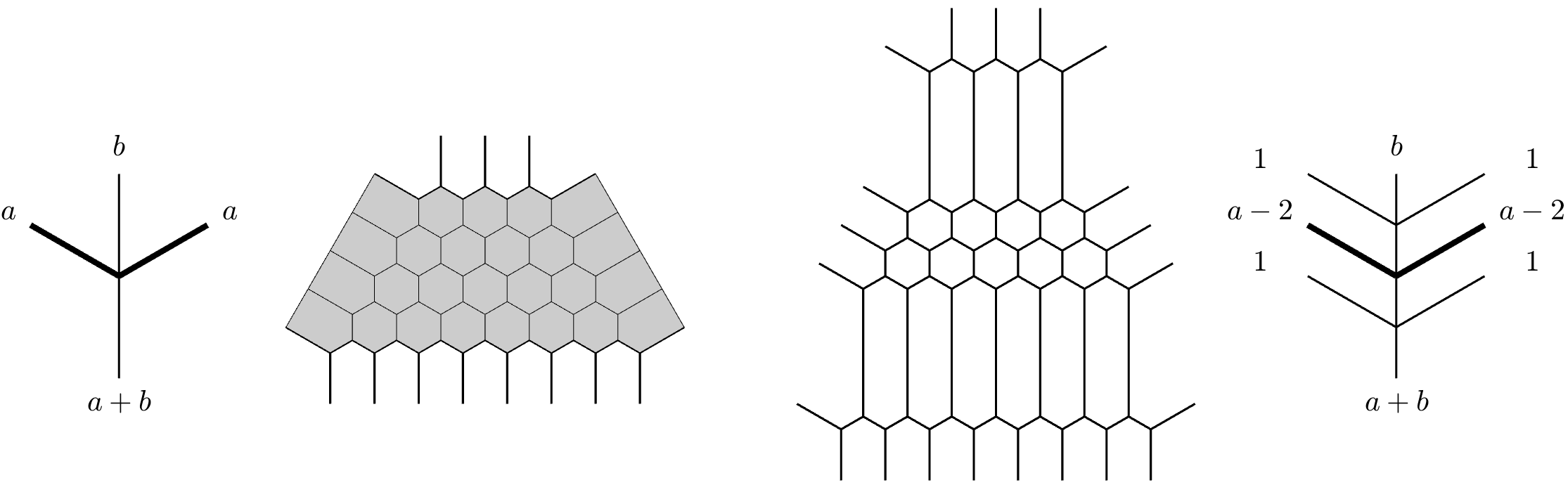}
\caption{Molting a 4-valent vertex by inflating gray regions.}
\label{fig19}
\end{figure}

$\bullet$ Type (VIII) : See Figure \ref{fig20}. 

\begin{figure}
\centering
\includegraphics[scale=0.30]{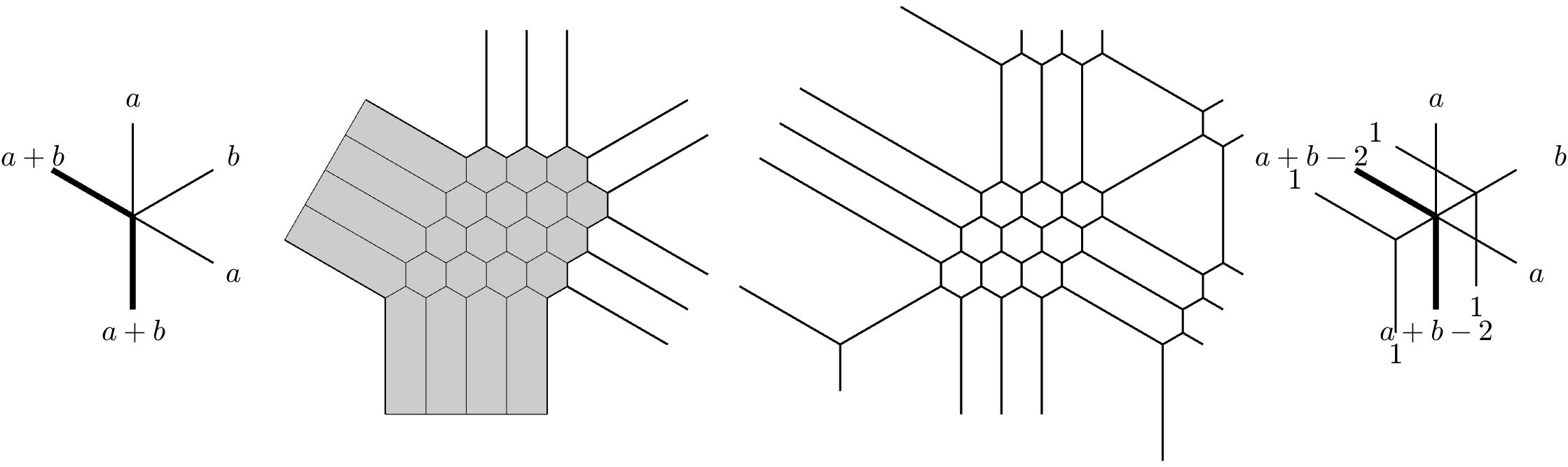}
\caption{Molting a 5-valent vertex by inflating gray regions.}
\label{fig20}
\end{figure}

Hence, inflating hexagons in \eqref{eqn4.10} leads to $\widetilde{h}' \in \MOB$, proving our claim.
\end{proof}

In \citep{Tao99}, the process of inflating hexagons simultaneously was called \textit{molting}, depicted in Figure \ref{fig17}.


\section{Loops in the M{\"o}bius strip}\label{sec5}

We continue working with largest-lifts $\widetilde{h} \in \MOB$. We sort out unfavorable white vertices in $\Gamma_n$ and connect them to construct loops. Since $\Gamma_n$ is embedded in the M{\"o}bius strip $B_\delta$, the loops may be non-orientable unlike \citep{Tao99}. Our claim is that such loops can be eliminated nicely, two at a time. Consequently, we prove Theorem \ref{thm3.2}, concluding that the Newell-Littlewood saturation holds.


\subsection{Fundamental groups of M{\"o}bius strips}

We list some well-known facts in algebraic topology. The fundamental group of the M{\"o}bius strip $B_\delta$ is $\mathbb{Z}$. Glue a disk to the boundary of the M{\"obius} strip and construct a real projective plane $\mathbb{RP}^{2}$. The fundamental group of $\mathbb{RP}^{2}$ is $\mathbb{Z}/2\mathbb{Z}$. The embedding $B_\delta \rightarrow \mathbb{RP}^{2}$ induces $\Pi_{1}(B_\delta) \rightarrow \Pi_{1}(\mathbb{RP}^{2})$, \textit{i.e.}, $\mathbb{Z} \rightarrow \mathbb{Z}/2\mathbb{Z}$. 

Intuitively, loops coiled around the M{\"o}bius strip for odd number of times are non-orientable loops whereas those coiled around for even number of times are orientable loops. Recall that two non-orientable loops must always intersect each other at least once. This follows from the fact that we have a disk by cutting $\mathbb{RP}^2$ along a non-orientable loop without self-intersection.

Let $G$ be a simple graph. In this paper, a list of vertices $(v_0, v_1 , \cdots, v_n)$ is called a \textbf{loop} in $G$ if $v_0 = v_n$ and $\{v_{i-1} , v_{i}\}$ are edges of $G$ for $1\leq i\leq n$. Vertices and edges are allowed to be repeated.

Let $\widetilde{h} \in \MOB$ and $h$ be its associated map. Let $C = (v_0 , v_1 , \cdots, v_k)$ be a loop in $\Gamma_n$. From $\rho_h : V_{\Gamma_n} \rightarrow V_{\Gamma_n (h)}$, construct a loop $C'$ by connecting vertices $\rho_h(v_0), \rho_h(v_1), \cdots, \rho_h (v_k)$. In other words, starting from a list of vertices $\rho_h(v_0), \rho_h(v_1), \cdots, \rho_h (v_k)$, whenever there is a pair of consecutive vertices which are the same, remove one of them. Then the remaining vertices $\rho_h (v_{i_0}) , \rho_h (v_{i_1}) , \cdots , \rho_h (v_{i_l})$ form the loop $C'$. Write $\rho_h (C):= C'$.

Note that $\widetilde{\Gamma}_n$ can be regarded as a planar graph embedded in $B$. Similarly, we regard $\Gamma_n$ as embedded in the M{\"o}bius strip $B_\delta$. Then as piecewise linear curves, we may define \textbf{orientable loops} and \textbf{non-orientable loops} in $\Gamma_n$. Similarly, note that $\Gamma_n (h)$ is also embedded in $B_\delta$ by the associated map $h : V_{\Gamma_n} \rightarrow B_\delta$. Therefore, it is possible to define \textbf{orientable loops} and \textbf{non-orientable loops} in $\Gamma_n (h)$. Then $C$ is an orientable loop in $\Gamma_n$ if and only if $\rho_h (C)$ is an orientable loop in $\Gamma_n (h)$. 

\begin{lemma}\label{lem5.0}
Let $C= (v_0 , v_1 , \cdots , v_k)$ be a loop in $\Gamma_n$. Then $C$ is orientable if and only if $k$ is an even integer.
\end{lemma}

To prove this lemma, we need to define orientation to each edge of $\Gamma_n$. Note that for each $a \in \R$, the plane $B$ is a union of
\begin{equation}
B = \{ (x,y,z) \in B \mid x-a \geq 0 \} \cup \{ (x,y,z) \in B \mid x-a \leq 0\}.
\end{equation}
The subset $x-a \geq 0$ is ``positive side'' of $(a,*,*)$ and $x-a \leq 0$ is ``negative side''.

Let $\widetilde{\alpha}$ and $\widetilde{\beta}$ be hexagons of $\widetilde{\Gamma}_n$ adjoined by an edge $\widetilde{e}$.\footnote{$\widetilde{\alpha}$ or $\widetilde{\beta}$ may be ``unbounded'' hexagons \textit{i.e.} hexagons assigned $0$ in Figure \ref{fig14}.} Let $\widetilde{h} \in \MOB$. Suppose $a = {\sf const}(\widetilde{h}; \widetilde{e})$ and $d(\widetilde{e}) = (0,-1,1)$. Without losing generality, assume 
\begin{equation}
\widetilde{h}(\widetilde{\alpha}) \subseteq \{ (x,y,z) \in B \mid x-a \geq 0 \}, \quad \widetilde{h}(\widetilde{\beta}) \subseteq \{ (x,y,z) \in B \mid x-a \leq 0 \}.
\end{equation}
Then we say that $\widetilde{\alpha}$ is on the \textbf{positive side} of $\widetilde{e}$ whereas $\widetilde{\beta}$ is on the \textbf{negative side} of $\widetilde{e}$. Due to \eqref{eqn3.9}, this is determined regardless of the choice of $\widetilde{h} \in \MOB$. Similarly, define positive and negative sides of $\widetilde{e}$ when $d(\widetilde{e})= (1,0,-1)$ or $d(\widetilde{e}) = (-1,1,0)$ by replacing $x$ with $y$ or $z$, respectively. 

Assign $+$ (resp. $-$) to a hexagon if it lies on the positive side (resp. negative side) of its adjoining edge. In this way, we assign signs to hexagons in $\widetilde{\Gamma}_n$ as in Figure \ref{fig58}. As a result, write $+ \rightarrow -$ clockwise for inward vertices and anti-clockwise for outward vertices. For instance, in Figure \ref{fig58}, $\widetilde{\alpha}_1$ is on the positive side while $\widetilde{\beta}_1$ is on the negative side with respect to the adjacent edge between them.

Let $\widetilde{\alpha}_2$ and $\widetilde{\beta}_2$ be another pair of adjacent hexagons satisfying  $p_h(\widetilde{\alpha}_1) = p_h(\widetilde{\alpha}_2)$ and $p_h(\widetilde{\beta}_1) = p_h(\widetilde{\beta}_2)$. As in Figure \ref{fig58}, $\widetilde{\alpha}_2$ is also on the positive side while $\widetilde{\beta}_2$ is on the negative side. Therefore, from the previous context, we say $\alpha:= p_h (\widetilde{\alpha})$ is on the \textbf{positive side} of $e := p_e (\widetilde{e})$ whereas $\beta:= p_h (\widetilde{\beta})$ is on the \textbf{negative side} of $e$. Hence, the signs assigned to edges of $\Gamma_n$ are well-defined as in Figure \ref{fig37}.

\begin{figure}
\centering
\begin{subfigure}[b]{\textwidth}
\centering
\includegraphics[scale=0.45]{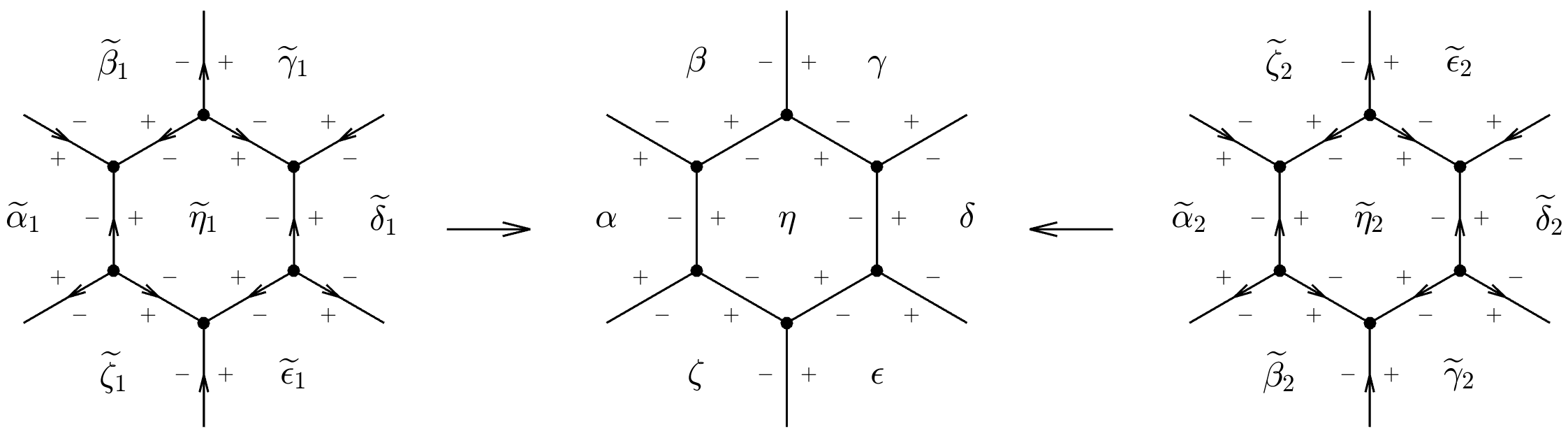}
\caption{Assigning signs to hexagons of $\widetilde{\Gamma}_n$ and $\Gamma_n$.}
\label{fig58}
\end{subfigure}

\begin{subfigure}[b]{\textwidth}
\centering
\includegraphics[scale = 0.5]{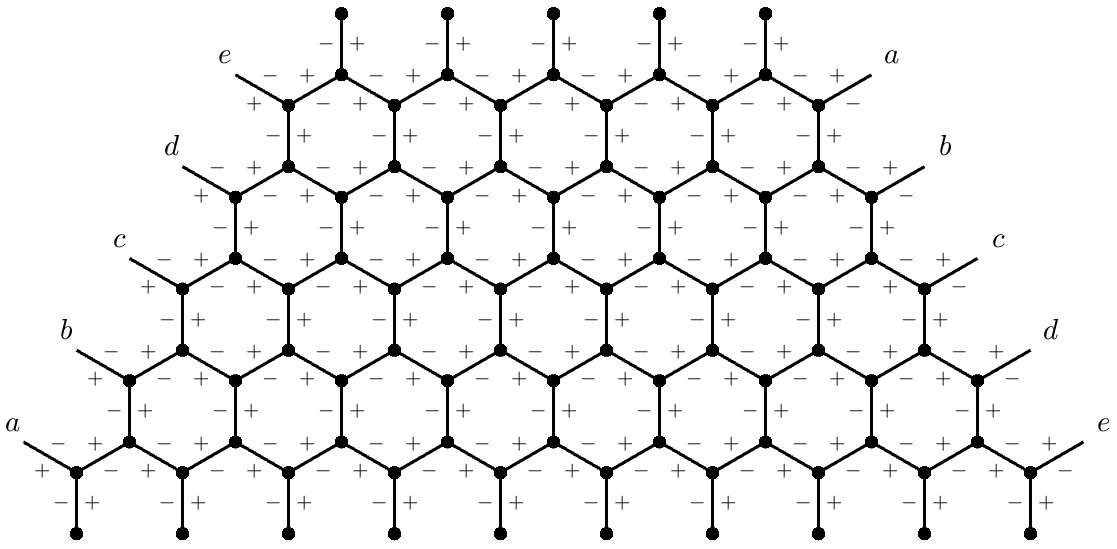}
\caption{Assigning signs to hexagons of the graph $\Gamma_{5}$.}
\label{fig37}
\end{subfigure}
\caption{Assigning signs to hexagons.}
\label{fig3758}
\end{figure}

\begin{proof}[Proof of Lemma \ref{lem5.0}]
For each edge $e_i:= \{v_{i-1}, v_i \} $ in $C$, define orientation as follows.
\begin{itemize}
	\item For even integer $1 \leq i \leq k$, set orientation of $e_i$ to the positive side.
	\item For odd integer $1 \leq i \leq k$, set orientation of $e_i$ to the negative side.
\end{itemize}
Indeed, for any loop in Figure \ref{fig37}, orientation of each edge should alternate between positive side and negative side. For the last edge $e_k$, its orientation should be on the positive side in order to comply with the orientation of $e_1$ on the negative side. This shows that the loop is orientable if and only if $k$ is even.
\end{proof}

\begin{figure}
\centering
\begin{subfigure}[b]{\textwidth}
\centering
\includegraphics[scale=0.45]{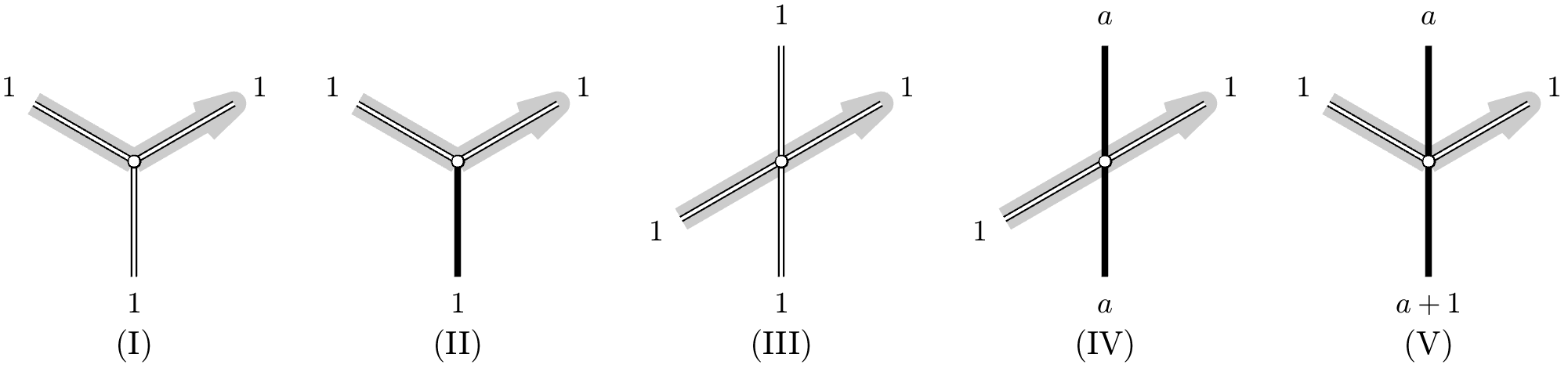}
\caption{Constructing a white loop in $\Gamma_n (h)$.}
\label{fig50}
\end{subfigure}

\begin{subfigure}[b]{\textwidth}
\centering
\includegraphics[scale=0.45]{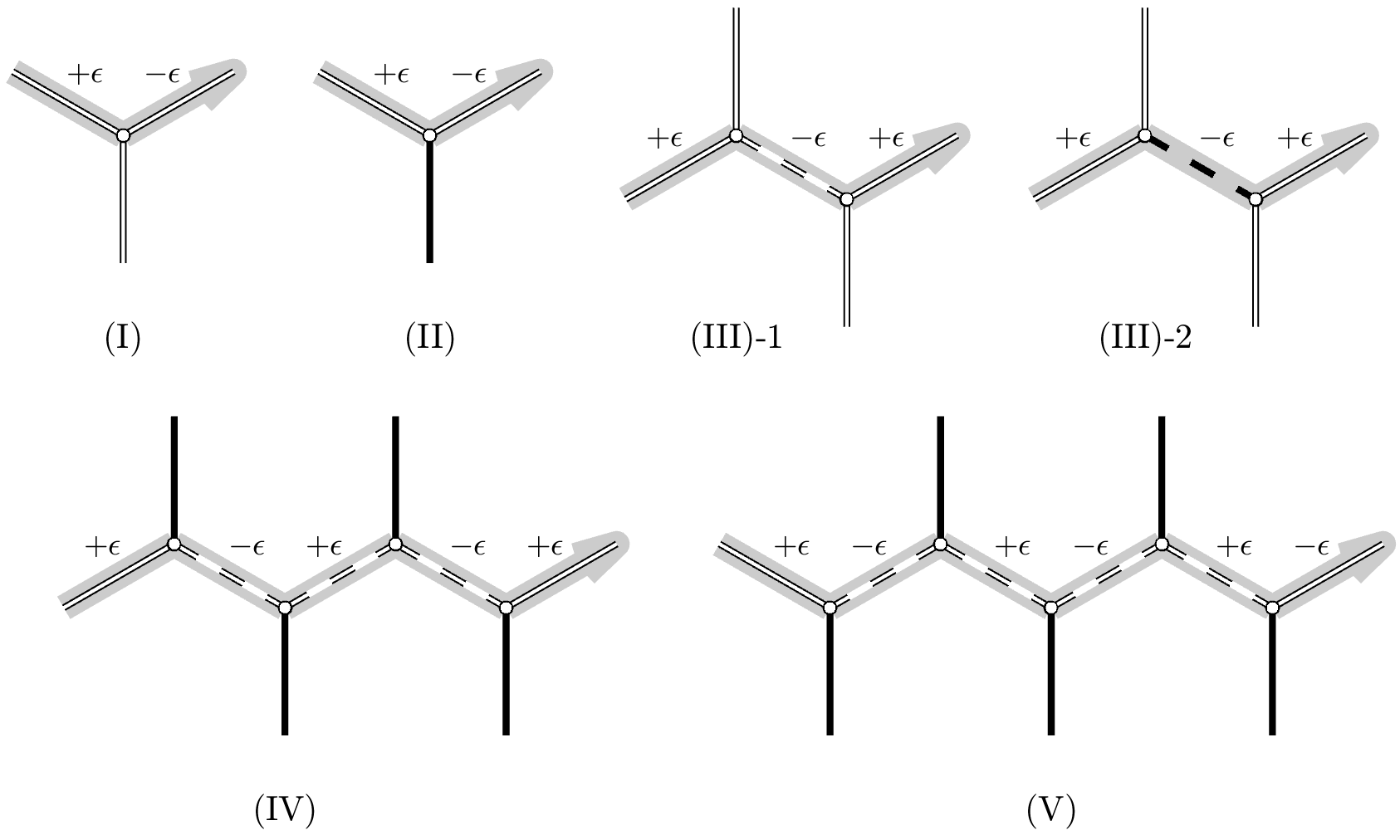}
\caption{Constructing a white loop in $\Gamma_n$.}
\label{fig49}
\end{subfigure}
\caption{A white loop.}
\label{fig4950}
\end{figure}

Let $\widetilde{h} \in \MOB$ be a largest-lift of $\xi \in \Z^{3n}$ and $h$ be its associated map. After coloring, there are five types of white vertices in $\Gamma_n(h)$ due to Theorem \ref{lem4.2}. We call a loop $C' = (w_0 , w_1 , \cdots , w_l)$ \textbf{a white loop} of $\Gamma_n (h)$ if it satisfies the following conditions.
\begin{itemize}
\item Its vertices and edges are all in white.
\item Whenever it encounters a crossing, it should go straight.
\end{itemize}
In other words, a white loop in $\Gamma_n (h)$ is constructed by following gray arrows in Figure \ref{fig50}. In addition, if a white loop $C'$ is a circuit\footnote{There is no edge repeated, but a vertex can be repeated.}, then $C'$ is called a \textbf{canonical white loop}.

Let $C$ be a loop in $\Gamma_n$. If $\rho_h (C)$ is a white loop (resp. canonical white loop) in $\Gamma_n (h)$, then we call $C$ a \textbf{white loop} (resp. \textbf{canonical white loop}) in $\Gamma_n$. A white loop in $\Gamma_n$ is constructed by following gray arrows in Figure \ref{fig49}. Here, degenerate edges are distinguished as dashed lines. Due to edge contraction, these degenerate edges are contracted to vertices. As a result, we have Figure \ref{fig50} from Figure \ref{fig49}. In particular, the white crossing type (III) in Figure \ref{fig50} is subdivided into types (III)-1 and (III)-2 in Figure \ref{fig49}, since the degenerate edge may be black or white.


\subsection{Sliding orientable loops}\label{sub5.2} Let $\widetilde{h} \in \MOB$. Suppose $\widetilde{A} \in V_{\widetilde{\Gamma}_n}$ is not a boundary vertex and let $\widetilde{e}_x , \widetilde{e}_y, \widetilde{e}_z \in E_{\widetilde{\Gamma}_n}$ be three protruding edges of $\widetilde{A}$
\begin{equation}\label{eqn5.5}
d(\widetilde{e}_x) = (0,-1,1), \quad d(\widetilde{e}_y) = (1,0,-1), \quad d(\widetilde{e}_z) = (-1,1,0).
\end{equation}
Then automatically, we have
\begin{equation}
\widetilde{h}(\widetilde{A}) = \left( {\sf const}(\widetilde{h};\widetilde{e}_x), {\sf const}(\widetilde{h};\widetilde{e}_y), {\sf const}(\widetilde{h};\widetilde{e}_z) \right).
\end{equation}
In other words, $ {\sf const}(\widetilde{h} ; \widetilde{e})$ has all the information of $\widetilde{h}$. One of the advantages of this approach is that we can extract statistics from $\Gamma_n$. Let $\widetilde{h}' \in \MOB$. Define
\begin{equation}
\varphi : E_{\Gamma_n} \rightarrow \R , \quad e \mapsto  {\sf const}(\widetilde{h}' ; \widetilde{e}) -{\sf const}(\widetilde{h} ; \widetilde{e}),
\end{equation}
for any $p_e (\widetilde{e}) = e$. This is well-defined by \eqref{eqn5.6} of Lemma \ref{lemZ.5}. Also, the sign of Figure \ref{fig37} is assigned so that whenever ${\sf const}(\widetilde{h}' ; \widetilde{e}) -{\sf const}(\widetilde{h} ; \widetilde{e}) > 0$, the image of $\widetilde{e}$ under $\widetilde{h}'$ lies on the positive side of that under $\widetilde{h}$, and vice versa.

Our strategy is to find an appropriate $\varphi$ from $\widetilde{h}$ to construct $\widetilde{h}'$. Define a map $\varphi : E_{\Gamma_n} \rightarrow \R$ so that for each vertex $v\in V_{\Gamma_n}$,
\begin{equation}\label{eqn5.2}
\sum_{v\text{ is an endpoint of }e} \varphi (e) = 0.
\end{equation}
Automatically, $\varphi (e) = 0$ if $e$ is connected to a boundary vertex.

For each $e \in E_{\Gamma_n}$, there are two hexagons (possibly unbounded) adjoined by $e$: $\alpha^+$ on the positive side and $\alpha^-$ on the negative side of $e$. Choose an endpoint of $e$ which is not a boundary vertex. Then this vertex is connected to two other edges: $f_1^+$ which  is a surrounding edge of $\alpha^-$ and $f_1^-$ of $\alpha^+$. If there is another endpoint of $e$ which is not a boundary vertex, choose $f_2^+$ and $f_2^-$ as well; see Figure \ref{fig59}. $\varphi$ is required to satisfy
\begin{equation}\label{eqn5.3}
{\sf length}(\widetilde{h} ; \widetilde{e}) \geq \frac{1}{2} \sum_i \left( \varphi(f_i^+) - \varphi(f_i^-) \right) .
\end{equation}

\begin{figure}
\centering
\begin{subfigure}[b]{\textwidth}
\centering
\includegraphics[scale=0.45]{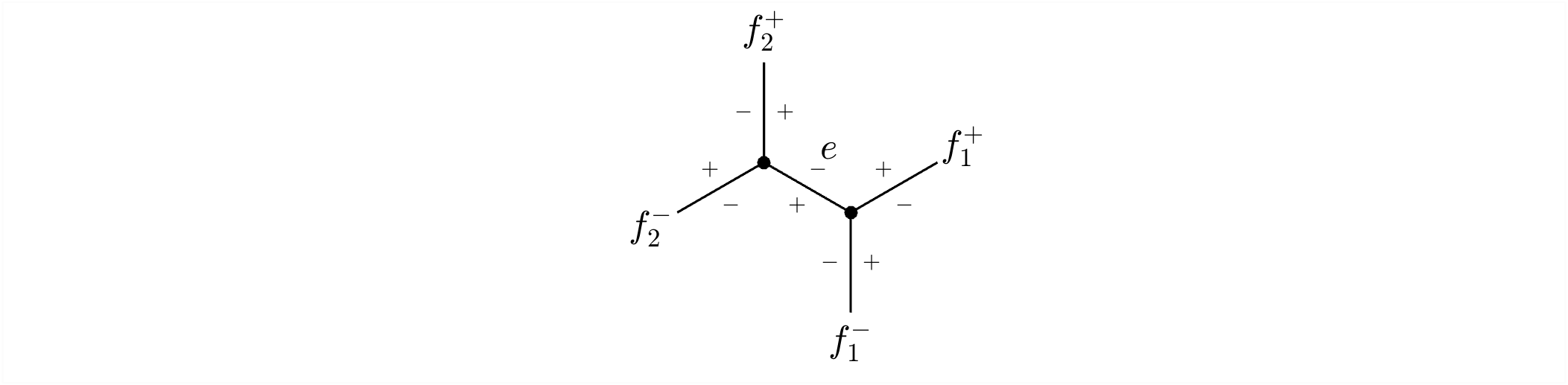}
\caption{Assigning $\varphi$ on edges of $\Gamma_n$.}
\label{fig59}
\end{subfigure}

\begin{subfigure}[b]{\textwidth}
\centering
\includegraphics[scale=0.45]{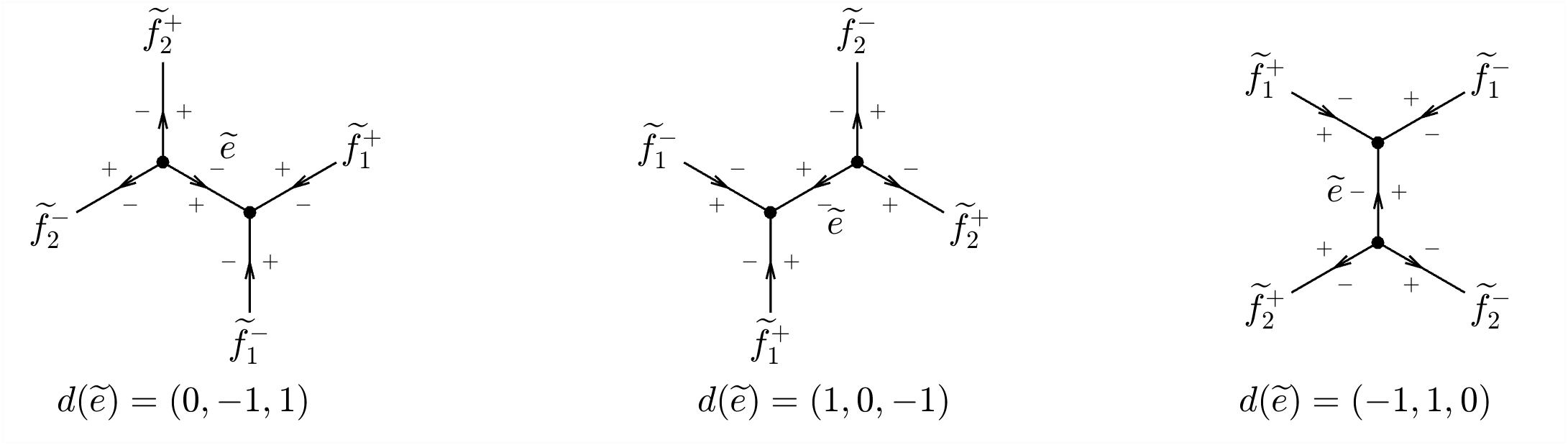}
\caption{${\sf head}(\widetilde{e}) = {\sf head}(\widetilde{f}_1^+) ={\sf head}(\widetilde{f}_1^-)$.}
\label{fig60}
\end{subfigure}

\begin{subfigure}[b]{\textwidth}
\centering
\includegraphics[scale=0.45]{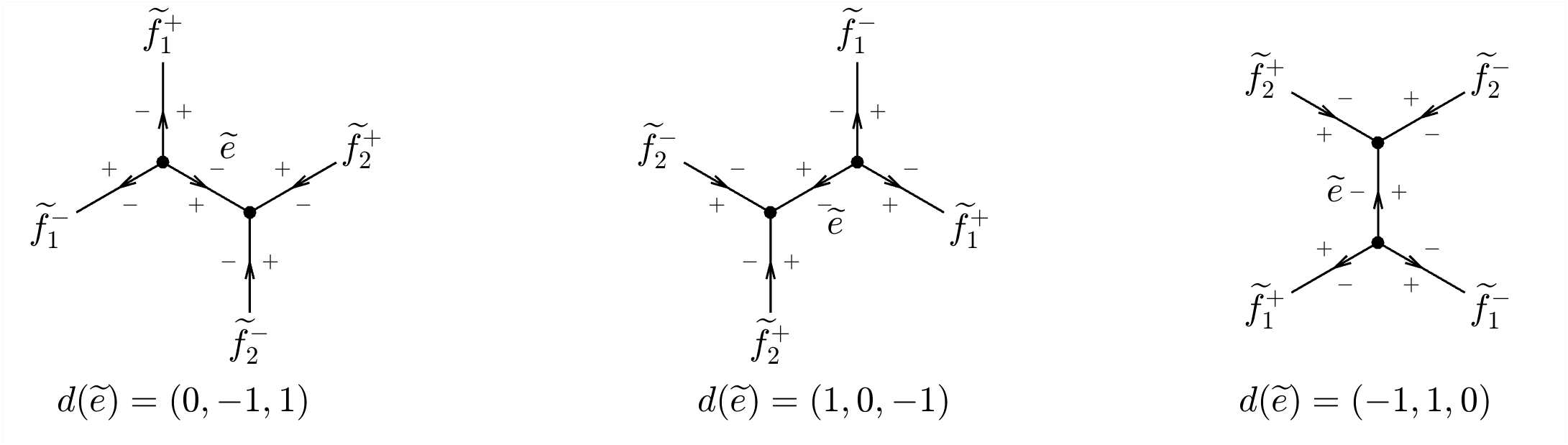}
\caption{${\sf head}(\widetilde{e}) = {\sf head}(\widetilde{f}_2^+) ={\sf head}(\widetilde{f}_2^-)$.}
\label{fig61}
\end{subfigure}

\caption{Sign of edges and $\varphi$.}
\label{fig596061}
\end{figure}

We define $\widetilde{h}_\varphi : V_{\widetilde{\Gamma}_n} \rightarrow B$ as follows: let $\widetilde{A} \in V_{\widetilde{\Gamma}_n}$. If $\widetilde{A}$ is a boundary vertex, then $\widetilde{h}_\varphi (\widetilde{A}) := \widetilde{h} (\widetilde{A})$. If $\widetilde{A}$ is not a boundary vertex, let $\widetilde{e}_x$, $\widetilde{e}_y$ and $\widetilde{e}_z$ be three protruding edges of $\widetilde{A}$ satisfying \eqref{eqn5.5}. Define
\begin{equation}\label{eqn5.4}
\widetilde{h}_\varphi (\widetilde{A}):= \widetilde{h}(\widetilde{A}) + (\varphi (e_x) , \varphi (e_y) , \varphi (e_z)). \quad \left( e_x = p_e (\widetilde{e}_x), \hspace{0.2em} e_y = p_e (\widetilde{e}_y), \hspace{0.2em} e_z = p_e (\widetilde{e}_z) \right)
\end{equation}
This defines the map $\widetilde{h}_\varphi : V_{\widetilde{\Gamma}_n} \rightarrow B$. Similar to \eqref{eqn5.6}, we have
\begin{equation}\label{eqn5.7}
{\sf const}(\widetilde{h}_\varphi ; \widetilde{e}) -{\sf const}(\widetilde{h} ; \widetilde{e}) = \varphi (e),
\end{equation}
for any $e = p_e(\widetilde{e})$.

\begin{lemma}\label{lem5.3}
Let $\widetilde{h} \in \MOB$ and $\varphi : E_{\Gamma_n} \rightarrow \R$ satifying \eqref{eqn5.2} and \eqref{eqn5.3}. Then $\widetilde{h}_\varphi $ is a M{\"o}bius honeycomb and $\partial \widetilde{h} = \partial \widetilde{h}_\varphi$.
\end{lemma}
\begin{proof}
To prove that $\widetilde{h}_\varphi$ satisfies \eqref{eqn3.9}, let $\widetilde{e}$ be an edge of $\widetilde{\Gamma}_n$. Assume that $\widetilde{e}$ is not connected to a boundary vertex. Denote the edges connected to ${\sf head}(\widetilde{e})$ as $\widetilde{f}_1^+, \widetilde{f}_1^-$ and the edges connected to  ${\sf tail}(\widetilde{e})$ as $\widetilde{f}_2^+, \widetilde{f}_2^-$. Let $\widetilde{f}_i^+$ (resp. $\widetilde{f}_i^-$) be on the negative side (resp. positive side) of $\widetilde{e}$ for $i=1,2$. Then we have three cases depicted in Figure \ref{fig60}. Write $e:= p_e(\widetilde{e})$, $f_i^+ := p_e(\widetilde{f}_i^+)$ and $f_i^- := p_e(\widetilde{f}_i^-)$. Then we have Figure \ref{fig59}.

Suppose $d(\widetilde{e}) = (0,-1,1)$. Then from the left picture of Figure \ref{fig60},
\begin{subequations}
\begin{equation}
\widetilde{h}_\varphi ({\sf head }(\widetilde{e}))  = \widetilde{h} ({\sf head }(\widetilde{e})) + (\varphi (e), \varphi (f_1^+), \varphi(f_1^-)),
\end{equation}
\begin{equation}
\widetilde{h}_\varphi ({\sf tail }(\widetilde{e}))  = \widetilde{h} ({\sf tail }(\widetilde{e})) + (\varphi (e), \varphi (f_2^-), \varphi(f_2^+)),
\end{equation}
\begin{equation}
\widetilde{h}_\varphi ({\sf head }(\widetilde{e})) - \widetilde{h}_\varphi ({\sf tail }(\widetilde{e})) = \widetilde{h} ({\sf head }(\widetilde{e})) - \widetilde{h} ({\sf tail }(\widetilde{e})) + (0, \varphi (f_1^+)-\varphi (f_2^-) , \varphi(f_1^-) - \varphi (f_2^+)).
\end{equation}
\end{subequations}
Since $\varphi (f_1^+)-\varphi (f_2^-) + \varphi(f_1^-) - \varphi (f_2^+) = 0$, this proves that $\widetilde{h}_\varphi ({\sf head }(\widetilde{e})) - \widetilde{h}_\varphi ({\sf tail }(\widetilde{e}))$ is parallel to $d(\widetilde{e}) = (0,-1,1)$. Compute the dot product by using \eqref{eqnZ.7} and Lemma \ref{lemZ.2}
\begin{equation}
d(\widetilde{e}) \cdot \left( \widetilde{h}_\varphi ({\sf head }(\widetilde{e})) - \widetilde{h}_\varphi ({\sf tail }(\widetilde{e})) \right)  = 2 \cdot {\sf length} (\widetilde{h} , \widetilde{e})  - \varphi (f_1^+) + \varphi (f_2^-) + \varphi(f_1^-) - \varphi (f_2^+)).
\end{equation}
Since we are assuming \eqref{eqn5.3}, the above value is non-negative. Therefore, \eqref{eqn3.9} is satisfied. Similarly, if $d(\widetilde{e}) = (1,0,-1)$, then from the middle picture of Figure \ref{fig60}
\begin{subequations}
\begin{equation}
\widetilde{h}_\varphi ({\sf head }(\widetilde{e}))  = \widetilde{h} ({\sf head }(\widetilde{e})) + ( \varphi(f_1^-), \varphi (e), \varphi (f_1^+)),
\end{equation}
\begin{equation}
\widetilde{h}_\varphi ({\sf tail }(\widetilde{e}))  = \widetilde{h} ({\sf tail }(\widetilde{e})) + ( \varphi(f_2^+), \varphi (e), \varphi (f_2^-)).
\end{equation}
\end{subequations}
If $d(\widetilde{e}) = (-1,1,0)$, then from the right picture of Figure \ref{fig60}
\begin{subequations}
\begin{equation}
\widetilde{h}_\varphi ({\sf head }(\widetilde{e}))  = \widetilde{h} ({\sf head }(\widetilde{e})) + (\varphi (f_1^+), \varphi(f_1^-), \varphi (e)),
\end{equation}
\begin{equation}
\widetilde{h}_\varphi ({\sf tail }(\widetilde{e}))  = \widetilde{h} ({\sf tail }(\widetilde{e})) + ( \varphi (f_2^-), \varphi(f_2^+), \varphi (e)).
\end{equation}
\end{subequations}
For each case, we can check that ${\sf head}(\widetilde{e}) - {\sf tail}(\widetilde{e})$ is on the same direction as $d(\widetilde{e})$. 
Lastly, if ${\sf head}(\widetilde{e})$ is a boundary vertex, then by replacing $\varphi (f_1^+)$ and $\varphi (f_1^-)$ with $0$ in all the equations above, we still have the same result. Similarly, if ${\sf tail}(\widetilde{e})$ is a boundary vertex, replace $\varphi (f_2^+)$ and $\varphi (f_2^-)$ with $0$. Hence, $\widetilde{h}_\varphi$ satisfies \eqref{eqn3.9}.

Since the images of boundary vertices under $\widetilde{h}$ and $\widetilde{h}_\varphi$ are the same by the definition, $\widetilde{h}_\varphi$ satisfies \eqref{eqn3.7}.

To check \eqref{eqn3.8}, let $\widetilde{A} \in V_{\widetilde{\Gamma}_n}$ from \eqref{eqn5.4}. Suppose $\widetilde{A}' \in V_{\widetilde{\Gamma}_n}$ satisfying $p_v (\widetilde{A}) = p_v (\widetilde{A}')$. If the protruding edges of $\widetilde{A}$ and $\widetilde{A}'$ are all outward or inward, then $\widetilde{A}'$ satisfies the same relation as in \eqref{eqn5.4}. If not, then
\begin{equation}
\widetilde{h}_\varphi (\widetilde{A}') = \widetilde{h}(\widetilde{A}') + (\varphi (e_y) , \varphi (e_x) , \varphi (e_z)).
\end{equation}
In other words, the positions of $\varphi (e_x)$ and $\varphi (e_y)$ are switched. Due to Lemma \ref{lemZ.1}, $\widetilde{h}_\varphi$ satisfies \eqref{eqn3.8}.
\end{proof}

\begin{figure}
\centering
\includegraphics[scale=0.55]{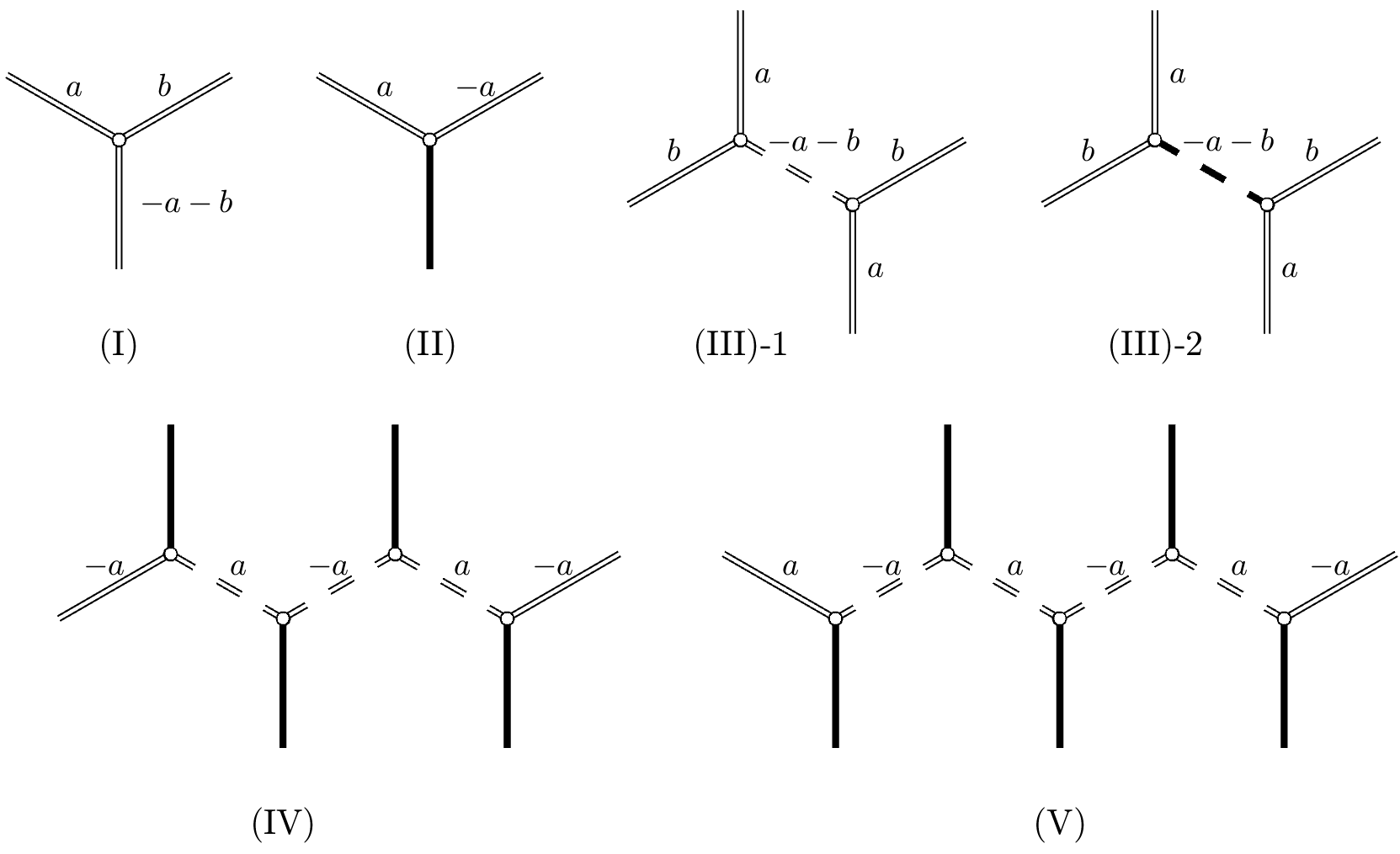}
\caption{Constructing a white loop to apply sliding.}
\label{fig48}
\end{figure}

\begin{figure}
\centering
\includegraphics[scale=0.33]{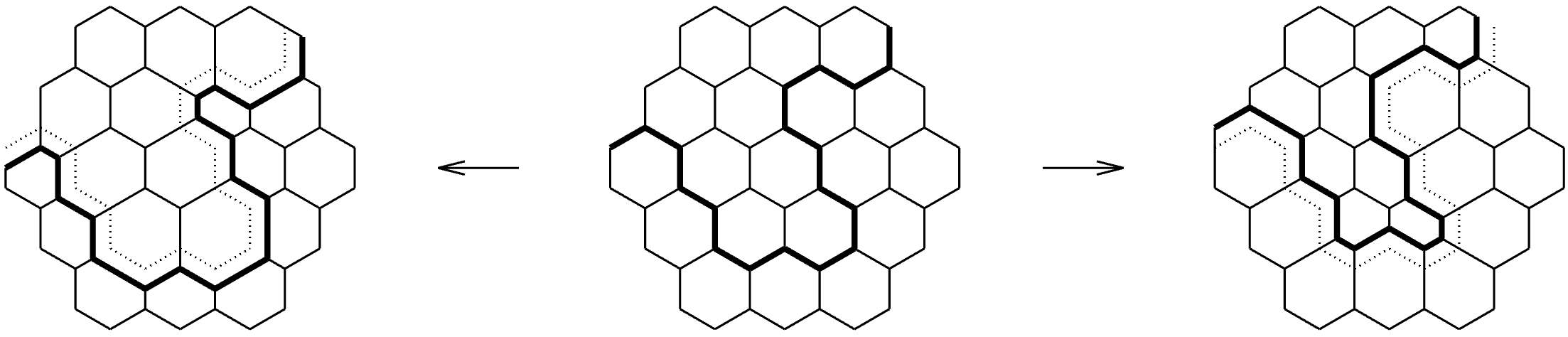}
\caption{Sliding an orientable loop.}
\label{fig25}
\end{figure}

\begin{lemma}\label{lem5.2}
Let $\widetilde{h} \in \MOB$ be a largest-lift of $\xi \in \Z^{3n}$. Let $C$ be a white loop in $\Gamma_n$. If there is an edge in $\Gamma_n$ which is used odd number of times\footnote{This condition is needed since we have an orientable loop by circling around a non-orientable one twice. We want to exclude such cases.} to construct $C$, then $C$ is non-orientable. In particular, all canonical white loops are non-orientable.
\end{lemma}

\begin{proof}
Suppose a white loop $C=(v_0 , v_1 , \cdots , v_k)$ in $\Gamma_n$ is orientable. Let $\epsilon >0$. Define $\varphi : E_{\Gamma_n} \rightarrow \R$ by following steps.
\begin{enumerate}
\item Initially, $\varphi (e) = 0$ for all $e \in E_{\Gamma_n}$.
\item For $i=1,2, \cdots , k$, add $\epsilon$ to $\varphi (\{v_{i-1} , v_i \})$ if $i$ is even. Subtract $\epsilon$ from $\varphi (\{v_{i-1} , v_i \})$ if $i$ is odd.
\end{enumerate}
As a result, $\varphi$ assigns real numbers to edges of the loop as in Figure \ref{fig48}. Indeed, this is possible since $k$ is an even integer due to Lemma \ref{lem5.0}. From Figure \ref{fig48}, \eqref{eqn5.2} is satisfied when a vertex is a white vertex. If a vertex is a black vertex, then the edges connected to it have $\varphi$ value zero. Therefore, $\varphi$ satisfies \eqref{eqn5.2}. 

To check $\eqref{eqn5.3}$, suppose we have a degenerate edge $e$. If $e$ is in white, then it should be one of (III)-1, (IV) or (V) types depicted in Figure \ref{fig48}. Then $\varphi(f_1^+) - \varphi (f_1^-) + \varphi (f_2^+) - \varphi (f_2^-) = 0$, leading to \eqref{eqn5.3}. If $e$ is in black, then it should be the type (III)-2 depicted in Figure \ref{fig48}, or it is connected to black edges with $\varphi$ value zero. Again, $\varphi(f_1^+) - \varphi (f_1^-) + \varphi (f_2^+) - \varphi (f_2^-) = 0$. 

Therefore, we only need to consider about the case when $e$ is non-degenerate in \eqref{eqn5.3}. In that case, since ${\sf length}(\widetilde{h} ; \widetilde{e})$, we can choose small enough $\epsilon$ to satisfy \eqref{eqn5.3}. Hence, $\varphi$ satisfies \eqref{eqn5.2} and \eqref{eqn5.3} so that $\widetilde{h}_\varphi \in \MOB$ and $\partial \widetilde{h} =\partial \widetilde{h}_\varphi$ due to Lemma \ref{lem5.3}.

On the other hand, $-\varphi : E_{\Gamma_n} \rightarrow \R$ also satisfies \eqref{eqn5.2} and \eqref{eqn5.3}: it assigns $-\epsilon$ instead of $\epsilon$. Therefore, $\widetilde{h}_{-\varphi} \in \MOB$ and $\partial \widetilde{h} =\partial \widetilde{h}_{-\varphi}$. By the definition \eqref{eqn5.4}, we have
\begin{equation}
\frac{1}{2} ( \widetilde{h}_\varphi + \widetilde{h}_{-\varphi} )= \widetilde{h}.
\end{equation}
Since there is an edge $e$ in $\Gamma_n$ which is chosen odd number of times to construct $C$, $\varphi (e) \neq 0$. Therefore, $\widetilde{h} \neq \widetilde{h}_\varphi$. However, this contradicts that $\widetilde{h}$ is a largest-lift, due to the Lemma \ref{lem3.4}. Therefore, the white loop constructed above should be non-orientable.
\end{proof}

In the proof of Lemma \ref{lem5.2}, constructing $\widetilde{h}_\varphi$ from an orientable loop can be understood as ``sliding'' the loop as in Figure \ref{fig25}. Here, we draw a picture of the image of $\widetilde{h}$ in $B$. Since the loop is orientable, it is possible to ``slide'' the loop as much as $\epsilon$, inward or outward. 

We now refine the result of Theorem \ref{lem4.2}.

\begin{theorem}\label{thm5.1}
Let $\widetilde{h} \in \MOB$ be a largest-lift of $\xi \in \Z^{3n}$. Then there is no white vertex connected to three non-degenerate white edges in $\Gamma_n$. In short, type (I) in Figure \ref{fig50} and Figure \ref{fig49} do not occur.
\end{theorem}

\begin{proof}
Suppose this is not true \textit{i.e.} there is a white vertex in $\Gamma_n (h)$ with three white edges connected to it. Starting from this vertex in $\Gamma_n (h)$, follow the gray arrows as in Figure \ref{fig50}, not using the same edge twice. In this way, we form a trail\footnote{There is no edge repeated, but a vertex can be repeated. The initial vertex and the terminal vertex are not necessarily the same. In this paper, a closed trail is a circuit.} until it is no longer possible to extend the both ends.

According to Theorem \ref{lem4.2} and Figure \ref{fig50}, all white vertices have even number of white edges protruding from them except for the type (I). Since the trail was constructed from the type (I) vertex, the trail cannot be a closed trail. Therefore, its initial vertex and terminal vertex should be the type (I), due to the fact that end points of Eulerian path always have odd numbers as their vertex degree. Write these initial and terminal vertices as $v_{int}$ and $v_{ter}$. Regarding the trail as a subgraph of $\Gamma_n (h)$, there are two cases.

(Case 1) There are three trails connecting $v_{int}$ and $v_{ter}$, all mutually edge disjoint : In other words, we have
\begin{subequations}
\begin{equation}
C_1 = (v_{int} , u_1 , u_2 , \cdots , u_k , v_{ter}),
\end{equation}
\begin{equation}
C_2 = (v_{int} , v_1 , v_2 , \cdots , v_l , v_{ter}),
\end{equation}
\begin{equation}
C_3 = (v_{int} , w_1 , w_2 , \cdots , w_m , v_{ter}).
\end{equation}
\end{subequations}
Then we have three canonical white loops in $\Gamma_n (h)$ by connecting $C_1$ and $C_2$, $C_2$ and $C_3$, $C_3$ and $C_1$. Since the fundamental group of a M{\"o}bius strip is $\Z / 2\Z$, one of them should be orientable loop, leading to contradiction with Lemma \ref{lem5.2}.

(Case 2) There is a trail connecting $v_{int}$ and $v_{ter}$, a circuit containing $v_{int}$ and a circuit containing $v_{ter}$, all mutually edge disjoint : In other words, we have
\begin{subequations}
\begin{equation}
C_1 = (v_{int} , u_1 , u_2 , \cdots , u_k , v_{int}),
\end{equation}
\begin{equation}
C_2 = (v_{int} , v_1 , v_2 , \cdots , v_l , v_{ter}),
\end{equation}
\begin{equation}
C_3 = (v_{ter} , w_1 , w_2 , \cdots , w_m , v_{ter}).
\end{equation}
\end{subequations}
Then $C_1$ and $C_3$ are canonical white loops in $\Gamma_n (h)$. Due to Lemma \ref{lem5.2}, $C_1$ and $C_3$ are non-orientable. Therefore, by connecting $C_1$, $C_2$ and $C_3$
\begin{equation}
C= ( v_{int} , u_1 , \cdots, u_k , v_{int} , v_1 , \cdots , v_l , v_{ter} , w_1 , \cdots , w_m , v_{ter}, v_l , \cdots , v_1 , v_{int}).
\end{equation}
Then $C$ is an orientable white loop. Due to Lemma \ref{lem5.2}, this leads to contradiction.\end{proof}

\begin{remark}\label{rmk5.1}
Lemma \ref{lem5.2} and Theorem \ref{thm5.1} give another proof of Theorem \ref{thm1.1}. In \citep{Tao99}, A.~Knutson and T.~Tao proved the special case when $\lambda , \mu , \nu$ are \textit{regular tableaux} \textit{i.e.} strictly decreasing partitions. Then they constructed a largest-lift globally, as a piecewise linear function ${\tt BDRY}(\tau_n) \rightarrow \HON$, where ${\tt BDRY}(\tau_n) := \partial (\HON)$. Using piecewise linearity, they proved that the theorem holds for the general case as well. On the other hand, our method can be used directly without assuming $\lambda , \mu , \nu$ are regular tableaux. Instead, we apply coloring technique on $\HON$ and use method in \citep{Tao99} only on white vertices and white edges. Due to Theorem \ref{thm5.1}, we can construct a canonical white loop from Figure \ref{fig4950} in $\Delta_n$, and it is non-orientable by Lemma \ref{lem5.2}. Since $\Delta_n$ is a planar graph, this proves that there are no white vertices or white edges, leading to Theorem \ref{thm2.2}.
\end{remark}


\subsection{Breaking non-orientable loops}

As a corollary of \citep[Theorem 2]{Tao99}, A.~Knutson and T.~Tao proved that a largest lift $g \in \HON$ maps vertices to lattice points if $\partial g \in \Z^{3n}$. In other words, a largest lift is the construction of $g$ in Theorem \ref{thm2.2}, leading to Theorem \ref{thm1.1}.

Theorem \ref{thm5.1} is analogous to \citep[Theorem 2]{Tao99}, classifying the image of a largest-lift as in Figure \ref{fig4950}. However, a largest lift in $\MOB$ has slightly different property, comparing to a largest-lift in $\HON$ of \citep{Tao99}. Not only does it map vertices of $\widetilde{\Gamma}_n$ to lattice points, but also to \textbf{half lattice points}: they are points $(x,y,z) \in B$ such that two of the coordinates $x,y,z$ are half integers  whereas the other one is an integer.

Similarly, define a \textbf{half lattice line} in $B$ if the constant coordinate $a$ in \eqref{eqn2.3} is a half integer. Recall that in Figure \ref{fig1}, lattice points and lattice lines are colored in black. We may add half lattice points and half lattice lines and color them in white which is illustrated in the Figure \ref{fig29}.

\begin{figure}
\centering
\includegraphics[scale=0.35]{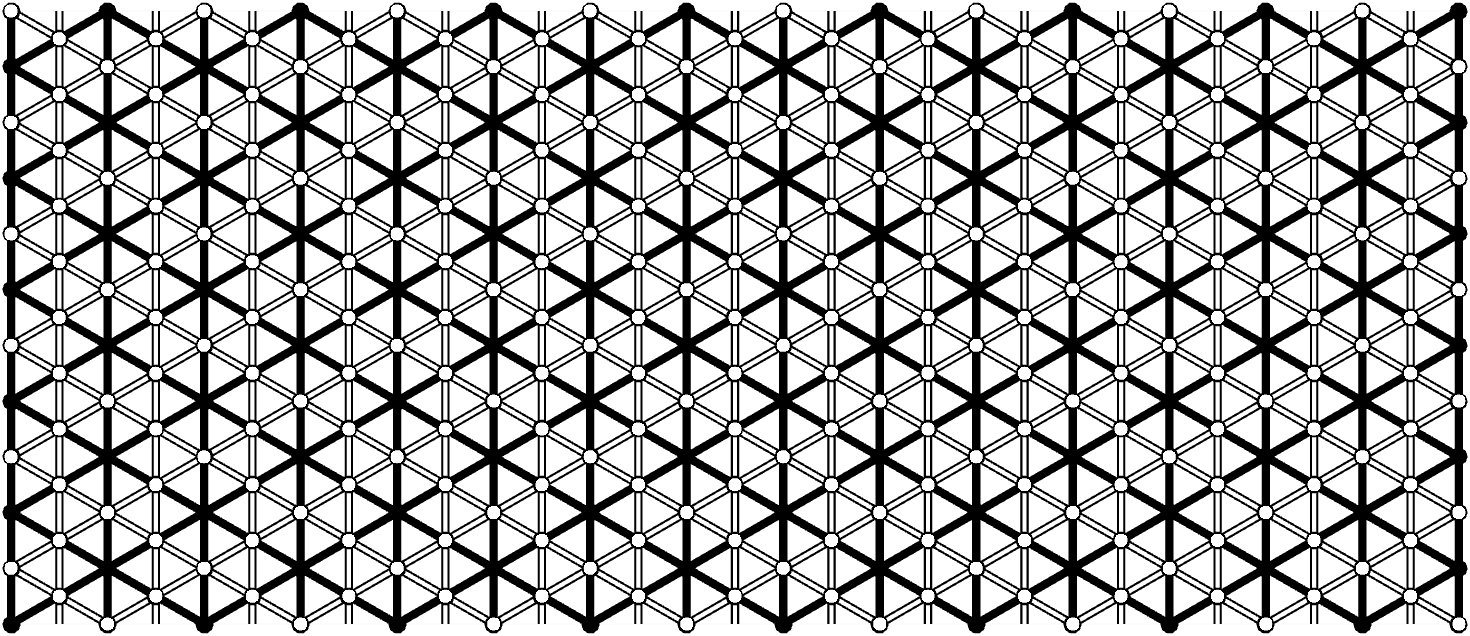}
\caption{Half lattice points, half lattice lines in the vector space $B$.}
\label{fig29}
\end{figure}

Note that we color the graph $\widetilde{\Gamma}_n$ in black and white in Subsection \ref{sub4.2}. By the definition of coloring, black vertices and black edges of $\widetilde{\Gamma}_n$ are mapped to black points and black line segments in $B$ under $\widetilde{h} \in \MOB$. Our next claim is that when $\widetilde{h}$ is a largest-lift, white vertices and white edges of $\widetilde{\Gamma}_n$ are mapped to white points and white line segments in $B$ under $\widetilde{h}$ as well.

\begin{theorem}\label{thm5.2}
Let $\widetilde{h} \in \MOB$ be a largest-lift of $\xi \in \Z^{3n}$. Then all white vertices in $\widetilde{\Gamma}_n$ are mapped to half lattice points under $\widetilde{h}$. 
\end{theorem}

\begin{figure}
\centering
\includegraphics[scale=0.33]{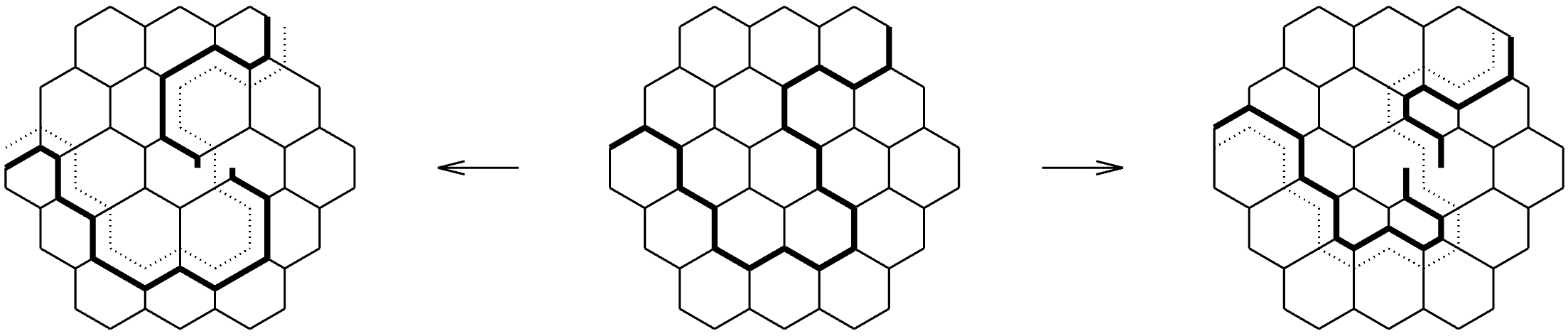}
\caption{Breaking a non-orientable loop.}
\label{fig26}
\end{figure}

Before proving the theorem, we discuss about the intuition of the proof. Let $C$ be a white loop, which is non-orientable due to Lemma \ref{lem5.2}. Moving edges of $C$ parallelly, one of the edges $e$ will be ``broken'', as depicted in Figure \ref{fig26}. Continue moving the edges until one of the edges lies on a lattice line. Then all edges of $C$ lie on lattice lines, due to Theorem \ref{thm5.1}. In particular, two pieces of the broken edge $e$ also lie on the lattice lines of constant coordinate $a$ and $a+1$ where $a \in \Z$. This means that the constant coordinate of $e$ was $a+ \frac{1}{2}$, proving the theorem. The process described in Figure \ref{fig26} is called ``breaking''.

\begin{proof}[Proof of Theorem \ref{thm5.2}]
For any $a \in \R$, write
\begin{equation}
\Z + a  := \{x \in \R \mid x= m+ a, m \in \Z \}.
\end{equation}
Let $e \in E_{\Gamma_n}$. Due to Lemma \ref{lemZ.5}, there exists $a \in \R$ such that for any $\widetilde{e} \in E_{\widetilde{\Gamma}_n}$ satisfying $p_e (\widetilde{e}) = e$, ${\sf const}(\widetilde{h} ; \widetilde{e}) \in \Z + a$.

Suppose there is a vertex $\widetilde{A}$ of $\widetilde{\Gamma}_n$ which is not mapped to lattice points or half lattice points under $\widetilde{h}$. Then there is a non-degenerate edge $\widetilde{e}_1$ of $\widetilde{\Gamma}_n$ such that ${\sf const}(\widetilde{h} ; \widetilde{e}_1) \in \Z -a$ where $a$ is neither an integer nor a half integer. Let $e_1 := p_e (\widetilde{e}_1)$. Due to Theorem \ref{thm5.1}, we may construct a white loop $C = (v_0, v_1, \cdots , v_k)$ in $\Gamma_n$ where $e_1 = \{v_0 , v_1 \}$. Let $e_i := \{v_{i-1}, v_i \}$ and $\widetilde{e}_i \in E_{\widetilde{\Gamma}_n}$ be chosen so that $p_e(\widetilde{e}_i) = e_i$ for $2 \leq i \leq k$. Let $f_i \in E_{\Gamma_n}$ be chosen so that $e_{i}$, $e_{i+1}$ and $f_i$ are protruding from a vertex and $\widetilde{f}_i \in E_{\widetilde{\Gamma}_n}$ such that $p_e (\widetilde{f}_i) = f_i$. Our claim is as follows.
\begin{itemize}
\item If $\widetilde{e}_i$ is a non-degenerate edge and $i$ is an odd integer, then ${\sf const}(\widetilde{h} ; \widetilde{e}_i) \in \Z - a$.
\item If $\widetilde{e}_i$ is a non-degenerate edge and $i$ is an even integer, then ${\sf const}(\widetilde{h} ; \widetilde{e}_i) \in \Z + a$.
\end{itemize}  
We use induction on $i$ to prove our claim. Let $e_i$ be non-degenerate where $i \geq 2$. Then there exists a positive integer $j$ such that $e_{i-1}, e_{i-2} , \cdots , e_{i-j+1}$ are degenerate whereas $e_{i-j}$ is non-degenerate. Due to Theorem \ref{thm5.1}, the non-degenerate edges $e_i$ and $e_{i-j}$ appear in $\Gamma_n$ as the both ends of the gray paths depicted in Figure \ref{fig49}, except for type (I). Then we have two cases below.

(Case 1) $j$ is an odd integer. Then we have type (II) or (V) in Figure \ref{fig49}. Then ${\sf const}(\widetilde{h}; \widetilde{f}_{i-j}) = {\sf const}(\widetilde{h}; \widetilde{f}_{i-j+1}) = \cdots = {\sf const}(\widetilde{h}; \widetilde{f}_{i-1}) \in \Z$. Therefore, ${\sf const}(\widetilde{h} ; \widetilde{e}_{i-j}) + {\sf const}(\widetilde{h} ; \widetilde{e}_i) \in \Z$, proving that the statement holds for $i$.

(Case 2) $j$ is an even integer. Then we have type (III) or (VI) in Figure \ref{fig49}. Then ${\sf const}(\widetilde{h}; \widetilde{e}_{i-j}) = {\sf const}(\widetilde{h}; \widetilde{e}_{i-j+2}) = {\sf const}(\widetilde{h}; \widetilde{e}_{i-j+4})  = \cdots = {\sf const}(\widetilde{h}; \widetilde{e}_{i})$, proving that the statement holds for $i$.

Hence, our claim is true. By Lemma \ref{lem5.2}, the loop $C$ is non-orientable \textit{i.e.} $k$ is an odd integer. By our claim, $\Z + a = \Z -a$, leading to contradiction.
\end{proof}

\begin{corollary}\label{coro5.2}
Let $\widetilde{h} \in \MOB$ be a largest-lift of $\xi \in \Z^{3n}$. Then the type (III)-1 in Figure \ref{fig50} and Figure \ref{fig49} do not occur.
\end{corollary}

\begin{proof}
This is true since coordinates of a half lattice point are composed of two half integers and one integer.
\end{proof}

Define the \textbf{total length of $\widetilde{h}$} by
\begin{equation}\label{eqn3.27}
{\sf ltotal}(\widetilde{h}):= \sum_{e \in E_{\Gamma_n}} {\sf length}(\widetilde{h} ; \widetilde{e}).
\end{equation}
Here, for each $e \in E_{\Gamma_n}$, choose a representative $\widetilde{e} \in E_{\widetilde{\Gamma}_n}$ such that $p_e(\widetilde{e}) = e$ and add ${\sf length} (\widetilde{h} ; \widetilde{e})$ to the sum. This is well defined due to \eqref{eqn4.4}. From Lemma \ref{lem5.1}, for any $\widetilde{h} \in \MOB$ with $\partial \widetilde{h} = \xi$,
\begin{equation}\label{eqn5.10}
{\sf ltotal}(\widetilde{h}) = \frac{1}{2}\sum_{j=1}^{3n} \xi_j.
\end{equation}

\begin{theorem}\label{thm5.3}
Let $\widetilde{h} \in \MOB$ be a largest-lift of $\xi \in \Z^{3n}$. In addition, assume that $\sum_{i=1}^{3n} \xi_i \equiv 0 \modtwo$. Then the number of white vertices in $\Gamma_n$ is even.
\end{theorem}

\begin{corollary}\label{coro5.1}
Let $\widetilde{h} \in \MOB$ be a largest-lift of $\xi \in \Z^{3n}$. In addition, assume that $\sum_{i=1}^{3n} \xi_i \equiv 0 \modtwo$. Then the number of canonical white loops in $\Gamma_n (h)$ is even.
\end{corollary}
\begin{proof}
Due to Lemma \ref{lem5.0} and \ref{lem5.2}, the number of vertices in a canonical white loop is odd. From the construction of white loops depicted in Figure \ref{fig49}, a canonical white loop is composed of vertices of types (II), (III)-2, (IV) and (V). Observe
\begin{itemize}
\item Vertices of types (II), (IV) and (V) are used exactly once.
\item Vertices of type (III)-2 are used exactly twice.
\end{itemize}
Combined with Theorem \ref{thm5.3}, this proves the corollary.
\end{proof}

We first give explanation of Theorem \ref{thm5.3} by using ``breaking''. As illustrated in Figure \ref{fig26}, we apply ``breaking'' on all canonical white loops so that all edges lie on the lattice lines. Then all edges have integer lengths except for broken edges, which have half integer lengths. On the other hand, the total length of edges before ``breaking'' is an integer due to \eqref{eqn5.10}. ``Breaking'' does not change the total length of edges, as depicted in Figure \ref{fig26}. Therefore, the total length of edges is still an integer. Hence, there should be even number of broken edges, implying that there are even number of canonical white loops and Theorem \ref{thm5.3}.

\begin{figure}
\centering
\begin{subfigure}[b]{\textwidth}
\centering
\includegraphics[scale=0.55]{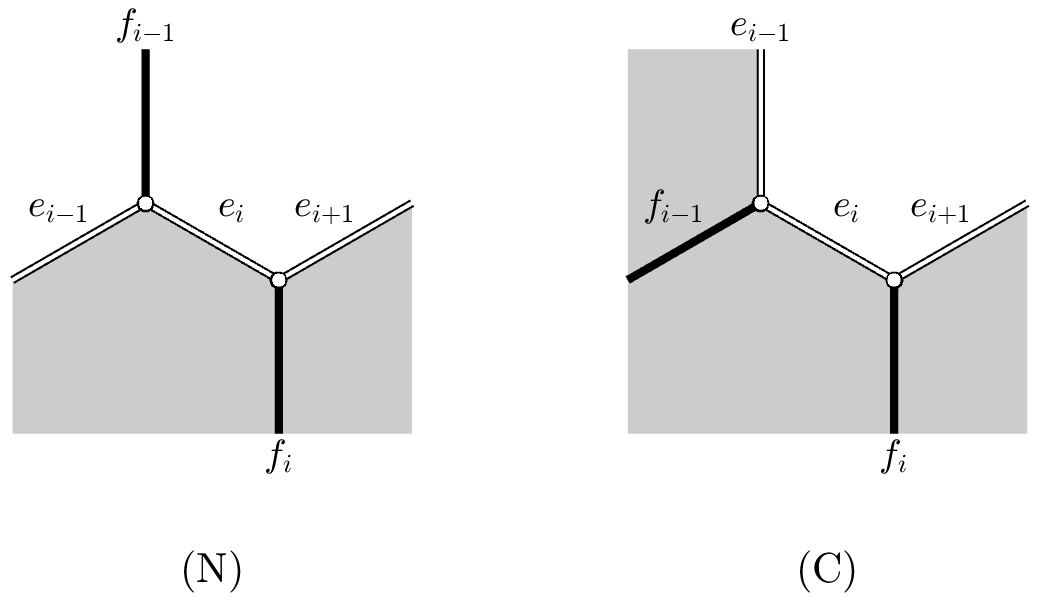}
\caption{White edge $e_i$: (N) type and (C) type.}
\label{fig52}
\end{subfigure}

\begin{subfigure}[b]{\textwidth}
\centering
\includegraphics[scale=0.55]{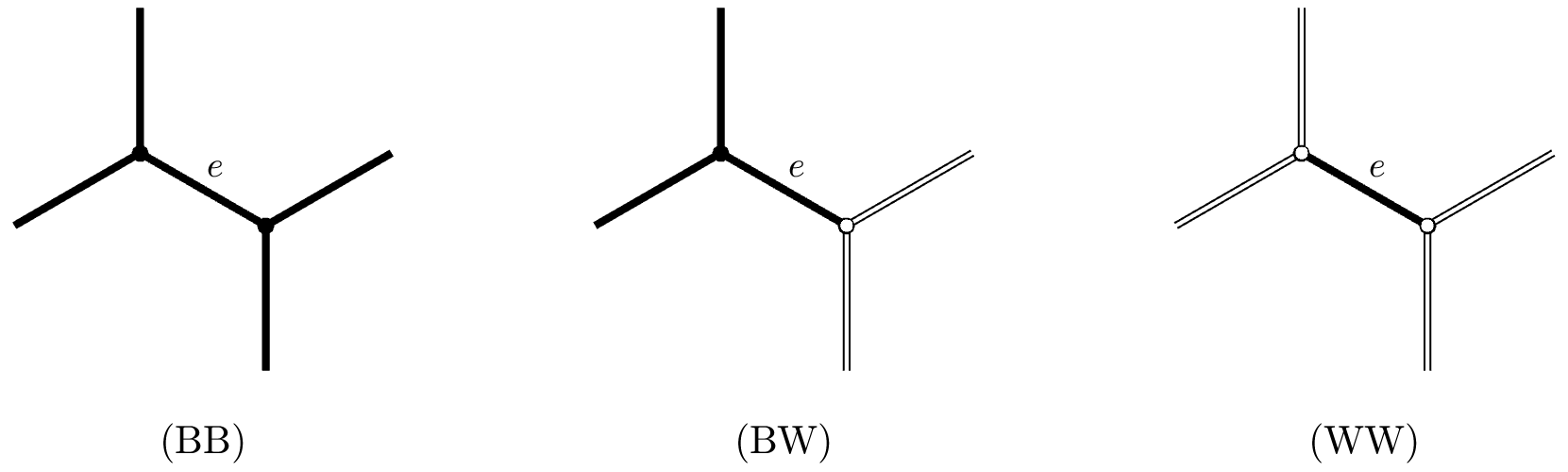}
\caption{Black edge $e$: (BB) type, (BW) type and (WW) type.}
\label{fig53}
\end{subfigure}
\caption{Five types of edges in $\Gamma_n$.}
\label{fig5253}
\end{figure}

\begin{proof}[Proof of Theorem \ref{thm5.3}]
Due to Theorem \ref{thm5.2}, a white vertex in $\Gamma_n$ is not a boundary vertex and is connected to exactly two white edges and one black edge. Therefore, we can construct cycles in $\Gamma_n$ composed of white vertices and white edges, mutually disjoint. (Not necessarily white loops.)

Let $C = (v_0 , v_1 , \cdots,  v_k)$ be one of those cycles. Write the edges $e_i := \{ v_{i-1} , v_i \}$ and choose $\widetilde{e}_i$ so that $e_i = p_e(\widetilde{e}_i)$. Then our claim is that
\begin{equation}\label{eqn5.21}
\sum_{i=1}^k {\sf length}(\widetilde{h}; \widetilde{e}_i) \in \Z.
\end{equation} 

To show this, observe that each $e_i$ is one of two types as in Figure \ref{fig52}.
\begin{itemize}
\item $e_i$ is of (N) type. In other words, $e_{i-1} , e_i , e_{i+1}$ form up a shape \textbf{\textit{N}}.
\item $e_i$ is of (C) type. In other words, $e_{i-1} , e_i , e_{i+1}$ form up a shape \textbf{\textit{C}}.
\end{itemize}
Write $(a_1, a_2, a_3):= \widetilde{h}({\sf head}(\widetilde{e}_i))$ and $(b_1, b_2, b_3):= \widetilde{h}({\sf tail}(\widetilde{e}_i))$. Since $\widetilde{e}_i$ are in white, so are ${\sf head}(\widetilde{e}_i)$ and ${\sf tail}(\widetilde{e}_i)$. By Theorem \ref{thm5.2}, white vertices are on half lattice points, which means that two out of three coordinates are half integers while the other one is an integer. Then
\begin{itemize}
\item If there exists $j$ such that $a_j, b_j \in \Z$, then $e_i$ is of (N) type.
\item Otherwise, $e_i$ is of (C) type.
\end{itemize}
Then we have
\begin{itemize}
\item If $e_i$ is of (N) type, then ${\sf length}(\widetilde{h} ; \widetilde{e}_i)$ is an integer.
\item If $e_i$ is of (C) type, then ${\sf length}(\widetilde{h} ; \widetilde{e}_i)$ is a half integer.
\end{itemize}
On the other hand, recall that we write edges protruding from $v_i$ as $e_i$, $e_{i+1}$ and $f_i$. Here, $f_i$ can be regarded as giving orientation to the loop $C$. Then
\begin{itemize}
\item If $e_i$ is of (N) type, then $f_{i-1}$ and $f_i$ are on different sides with respect to $C$.
\item  If $e_i$ is of (C) type, then $f_{i-1}$ and $f_i$ are on the same side with respect to $C$.
\end{itemize}

Therefore, we have
\begin{itemize}
\item If $C$ is orientable, then there are even number of (N) type $e_i$. Due to Lemma \ref{lem5.0}, $k$ is even. Therefore, there are even number of (C) type $e_i$.
\item If $C$ is non-orientable, then there are odd number of (N) type $e_i$. Due to Lemma \ref{lem5.0}, $k$ is odd. Therefore, there are even number of (C) type $e_i$.
\end{itemize}

Either way, we conclude that there are even number of (C) type edges. Therefore, we have \eqref{eqn5.21}, showing that
\begin{equation}
\sum_{\text{white edge }e}{\sf length}(\widetilde{h} ; \widetilde{e}) \in \Z,
\end{equation}
where $p_e (\widetilde{e}) = e$. Due to \eqref{eqn5.10} and the condition $\sum_{i=1}^{3n} \xi_i \equiv 0 \modtwo$,
\begin{equation}\label{eqn5.22}
\sum_{\text{black edge }e}{\sf length}(\widetilde{h} ; \widetilde{e}) \in \Z,
\end{equation}
where $p_e (\widetilde{e}) = e$.

For each black edge, there are three types as in Figure \ref{fig53}.
\begin{itemize}
\item If both endpoints of black edge $e$ are in black, then it is of (BB) type and ${\sf length}(\widetilde{h}, \widetilde{e}) $ is an integer.
\item If both endpoints of black edge $e$ are in white, then it is of (WW) type and ${\sf length}(\widetilde{h}, \widetilde{e}) $ is an integer.
\item If endpoints of black edge $e$ are in different colors, then it is of (BW) type and ${\sf length}(\widetilde{h}, \widetilde{e}) $ is a half integer.
\end{itemize}

Suppose the number of white vertices in $\Gamma_n$ is odd. Then the number of (BW) type edges is odd. Then we have
\begin{equation}
\sum_{\text{black edge }e}{\sf length}(\widetilde{h} ; \widetilde{e}) \in \Z + 0.5,
\end{equation}     
which contradicts \eqref{eqn5.22}. Hence, the number of white vertices in $\Gamma_n$ is even.
\end{proof}

The above theorem is the reason why we need $|\lambda| + |\mu| + |\nu| \equiv 0 \modtwo$ so that Newell-Littlewood saturation holds. It ensures that there are even number of canonical white loops in $B_\delta$, enabling us to pair them up.


\subsection{Proof of the main theorem}\label{sub5.4} 

As before, assume $\widetilde{h} \in \MOB$ is chosen so that it is a largest-lift of $\xi \in \Z^{3n}$ with $\sum_{i=1}^{3n} \xi_i \equiv 0 \modtwo$. 

\begin{proof}[Proof of Theorem \ref{thm3.2}]

Due to Theorem \ref{thm5.2}, for each edge $\widetilde{e}$ of $\widetilde{\Gamma}_n$,
\begin{itemize}
\item If $\widetilde{e}$ is in white, then ${\sf const}(\widetilde{h}; \widetilde{e}) \in \Z +0.5$.
\item If $\widetilde{e}$ is in black, then ${\sf const}(\widetilde{h} ; \widetilde{e}) \in \Z$.
\end{itemize}

We want to construct $\varphi : E_{\Gamma_n} \rightarrow \R$ so that for each edge $e$ of $\Gamma_n$,
\begin{itemize}
\item If $e$ is in white, then $\varphi (e)$ is $0.5$ or $-0.5$.
\item If $e$ is in black, then $\varphi (e)$ is $1$, $-1$ or $0$. 
\end{itemize}
If $\varphi$ satisfies \eqref{eqn5.2} and \eqref{eqn5.3}, then $\widetilde{g}:= \widetilde{h}_\varphi$ is a M{\"o}bius honeycomb and $\partial \widetilde{g} = \xi$ by Lemma \ref{lem5.3}. From \eqref{eqn5.7}, ${\sf const}(\widetilde{g} ; \widetilde{e}) \in \Z$ for all $\widetilde{e} \in E_{\widetilde{\Gamma}_n}$, proving the existence of $\widetilde{g}$ in Theorem \ref{thm3.2}.

To construct such $\varphi$, we first set it $\varphi \equiv 0$. Due to Corollary \ref{coro5.1}, it is possible to pair up canonical white loops. Let $C$ and $C'$ be a pair of canonical white loops in $\Gamma_n$
\begin{equation}
C = (W_0, W_1 , \cdots, W_k), \quad C' = (U_0 , U_1 ,  \cdots,  U_l).
\end{equation}
Due to Lemma \ref{lem5.2}, $C$ and $C'$ are non-orientable. Then by the fact from algebraic topology, they should intersect. Observing Figure \ref{fig49}, the intersection is only possible at crossing of type (III)-2 due to Corollary \ref{coro5.2}. Shift the vertices of $C$ and $C'$ so that $U_0 = W_1$ and $W_0 = U_1$ as in Figure \ref{fig54}. Re-define $\varphi$ by adding $(-1)^i \cdot 0.5$ as in Figure \ref{fig54}
\begin{subequations}
\begin{equation}
\varphi ( \{ W_{i-1} , W_{i} \}) \leftarrow  \varphi ( \{ W_{i-1} , W_{i} \}) + (-1)^i \cdot 0.5 \quad (2 \leq i \leq k),
\end{equation}
\begin{equation}
\varphi ( \{ U_{i-1} , U_{i} \}) \leftarrow \varphi ( \{ U_{i-1} , U_{i} \})  +  (-1)^i \cdot 0.5 \quad (2 \leq i \leq l).
\end{equation}
\end{subequations}

Since $k$ and $l$ are odd integers due to Lemma \ref{lem5.0} and \ref{lem5.2}, $\varphi$ still satisfies \eqref{eqn5.2}. This process is called \textbf{double breaking} at the intersection point of $C$ and $C'$. As a result, the degenerate edge $\{U_0 , U_1 \} = \{W_0,W_1\}$ becomes non-degenerate. It is important that reversing signs is not permitted unlike Figure \ref{fig48}. Continue applying double breaking on all pairs of canonical white loops. Then each non-boundary vertex of $\Gamma_n$ is one of the types in Figure \ref{fig55}, degenerate edges depicted as dashed lines and canonical white loops depicted as gray paths.

\begin{figure}
\centering
\includegraphics[scale=0.45]{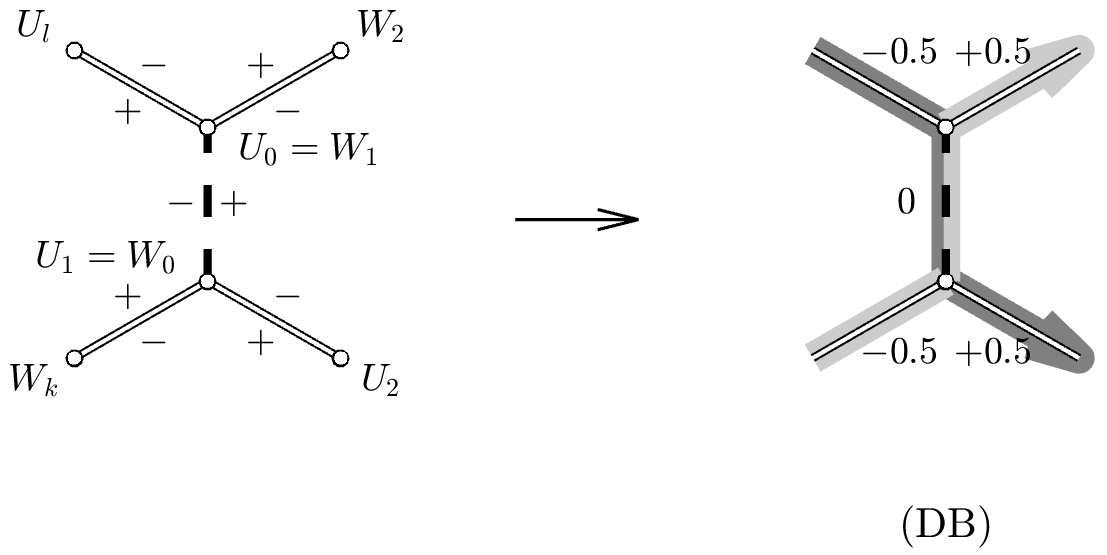}
\caption{Double breaking a pair of canonical white loops.}
\label{fig54}
\end{figure}

\begin{figure}
\centering
\includegraphics[scale=0.45]{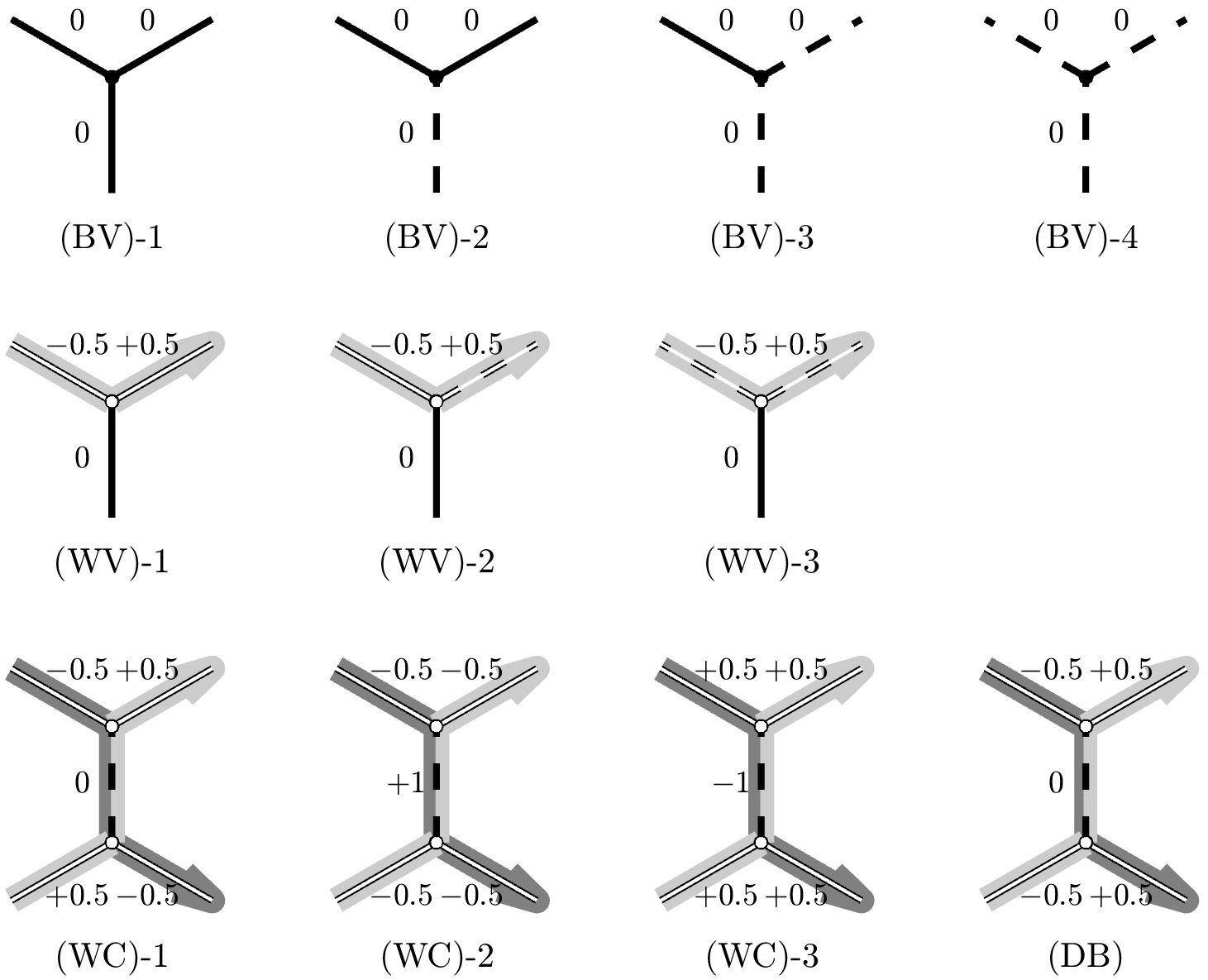}
\caption{Possible cases of vertices in $\Gamma_n$ after assigning $\varphi$.}
\label{fig55}
\end{figure}

We need to prove that $\varphi$ satisfies \eqref{eqn5.3}. Let $e$ be an edge of $\Gamma_n$ with $f_1^+ , f_1^- , f_2^+, f_2^-$ assigned as in Figure \ref{fig59}. Then there are four cases.

(Case 1) $e$ is a non-degenerate black edge : Since $e$ is non-degenerate, $\varphi (e) = 0$. According to Figure \ref{fig53}, there are three types of $e$: (BB), (BW) and (WW) types. If $e$ is (BB) type, then $\varphi (f_i^+) = \varphi(f_i^-) = 0$ for $i=1,2$ so that \eqref{eqn5.3} is satisfied. If $e$ is (BW) type, then ${\sf length}(\widetilde{h} ; \widetilde{e}) \geq 0.5 $ and the values of $\varphi (f_i^+)$ and $\varphi(f_i^-)$ should be one of $0 , 0 , -0.5 , +0.5$. Therefore, we again have \eqref{eqn5.3}. Lastly, if $e$ is of (WW) type, then ${\sf length}(\widetilde{h} ; \widetilde{e}) \geq 1 $ since $e$ is non-degenerate. Again, since the values of $\varphi (f_i^+)$ and $\varphi(f_i^-)$ should be one of $-0.5 , +0.5 , -0.5 , +0.5$, \eqref{eqn5.3} is satisfied.

(Case 2) $e$ is a degenerate black edge : Then the endpoints of $e$ are in the same color. If the color is black, then $\varphi (f_i^+) = \varphi(f_i^-) = 0$ for $i=1,2$. If the color is white, then $e$ is one of the types (WC)-1, (WC)-2, (WC)-3 and double breaking (DB). If $e$ is one of the types (WC)-1, (WC)-2 and (WC)-3, then both sides of \eqref{eqn5.3} are zeros. If $e$ is of type (DB), then the right-hand side of \eqref{eqn5.3} is negative due to assignment of $\varphi$ in Figure \ref{fig54}.

(Case 3) $e$ is a degenerate white edge : According to Figure \ref{fig49}, $e$ is type (IV) or (V) since (III)-1 is no longer possible. Therefore, \eqref{eqn5.3} is satisfied since both sides are zeros.

(Case 4) $e$ is a non-degenerate white edge : According to Figure \ref{fig52}, $e$ is one of the types (N) or (C). If $e$ is of type (N), then ${\sf length} (\widetilde{h} ; \widetilde{e}) \geq 1$. If \eqref{eqn5.3} is violated, then $\varphi (f_i^+), \varphi(f_i^-) $ should be one of $+0.5 , -0.5 , +1 , -1$, which is impossible. Therefore, \eqref{eqn5.3} is satisfied.

However, if $e$ is of type (C), then there is a case when \eqref{eqn5.3} is violated. To find the exceptional case, assume that $\eqref{eqn5.3}$ is violated for $e$. Since $e$ is of type (C), ${\sf length} (\widetilde{h} ; \widetilde{e}) \in \Z + 0.5$. Assuming that \eqref{eqn5.3} is violated, ${\sf length} (\widetilde{h} ; \widetilde{e}) = 0.5$. Without losing generality, assume that $f_1^+, f_2^+$ are white edges and $f_1^- , f_2^-$ are black edges. From Figure \ref{fig49}, $f_1^+$ and $f_2^+$ are non-degenerate. On the other hand, if one of $f_1^-$ and $f_2^-$ is non-degenerate, then the value of $\varphi$ is zero, so that \eqref{eqn5.3} is satisfied. Therefore, $f_1^-$ and $f_2^-$ are both degenerate black edges. Hence, the worst scenario when \eqref{eqn5.3} is violated is
\begin{equation}
\varphi(f_1^+) = \varphi(f_2^+) = +0.5, \quad \varphi(f_1^-) = \varphi(f_2^-) = -1, \quad \varphi(e) = +0.5, \quad {\sf length}(\widetilde{h} ; \widetilde{e}) = 0.5.
\end{equation}
Since $f_1^-$ and $f_2^-$ are degenerate black edges, their end points are in the same color, which is white. Therefore, the endpoints of $f_1^-$ and $f_2^-$ are connected to non-degenerate white edges. Since ${\sf length}(\widetilde{h}; \widetilde{e}) = 0.5$, we have a \textbf{white triangle} of size $0.5$ depicted in Figure \ref{fig56}. Here, canonical white loops are depicted as gray paths. 

\begin{figure}
\centering
\includegraphics[scale=0.45]{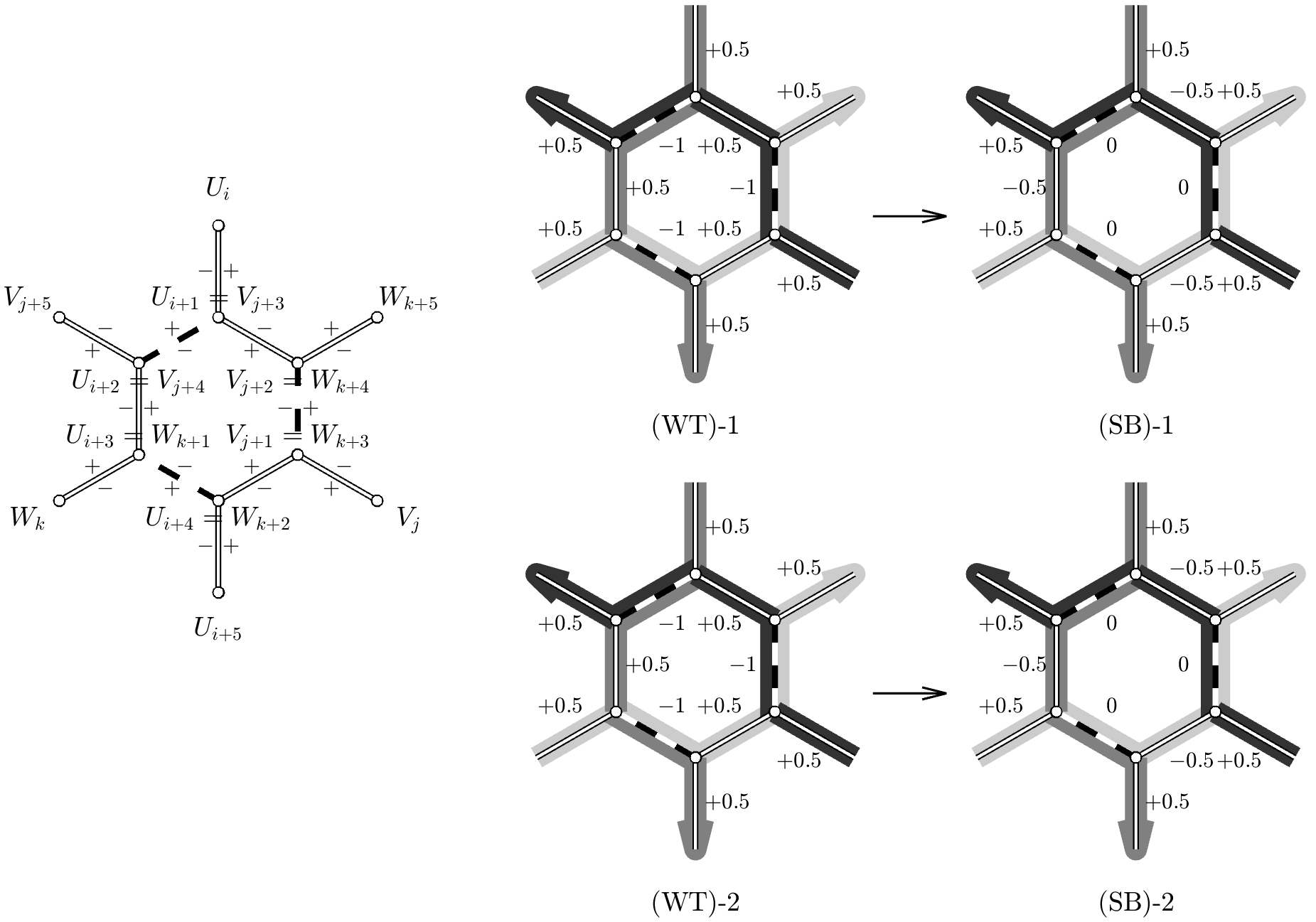}
\caption{Sixfold breaking in $\Gamma_n$.}
\label{fig56}
\end{figure}

Whenever we have a white triangle (WT)-1 or (WT)-2 as in Figure \ref{fig56}, modify $\varphi$ so that the assigned values are changed into (SB)-1 or (SB)-2. We call this process \textbf{sixfold breaking}. Note that white triangles (WT)-1 and (WT)-2 cannot be adjacent to each other. Therefore, we can modify $\varphi$ ``locally'' without affecting others. After the sixfold breaking, $\varphi$ satisfies \eqref{eqn5.3}. Apply Lemma \ref{lem5.3} to construct $\widetilde{h}_\varphi$ which is the construction of $\widetilde{g}$.
\end{proof}

\appendix

\section{Basic properties}\label{secZ}

In this section, we prove basic properties of M{\"o}bius honeycombs directly from \eqref{eqn3.9}, \eqref{eqn3.7} and \eqref{eqn3.8}.

Let $j,k \in \Z$ and $1 \leq j \leq 3n$. Construct a path ${\sf DPath}_j^{(2k)}$ in $\widetilde{\Gamma}_n$ between boundary vertices $\widetilde{A}_{0, j+3k}$ and $\widetilde{B}_{n,j+3k+n}$, disregarding the directions of edges:
\begin{equation}
{\sf DPath}_j^{(2k)} := ( \widetilde{A}_{0,j+3k}, \widetilde{B}_{0,j+3k},\widetilde{A}_{1,j+3k+1}, \widetilde{B}_{1,j+3k+1}, \cdots , \widetilde{A}_{n,j+3k+n}, \widetilde{B}_{n,j+3k+n} ).
\end{equation}
Similarly, define
\begin{equation}
{\sf DPath}_j^{(2k+1)} := ( \widetilde{B}_{n,j+3k+2n}, \widetilde{A}_{n,j+3k+2n},\widetilde{B}_{n-1,j+3k+2n}, \widetilde{A}_{n-1,j+3k+2n}, \cdots , \widetilde{B}_{0,j+3k+2n}, \widetilde{A}_{0,j+3k+2n} ).
\end{equation}

For each $k \in \Z$, ${\sf DPath}_j^{(k)}$ is called \textbf{$j$th diagonal path}. In Figure \ref{fig42}, four $3$rd diagonal paths are depicted as bold lines, each one contained in a trapezoid.

For fixed $j$, ${\sf DPath}_j^{(k)}$ are identified to each other by equivalence relation in $V_{\widetilde{\Gamma}_n}$. See Figure \ref{fig45}.

\begin{figure}
\centering
\begin{subfigure}[b]{\textwidth}
\centering
\includegraphics[scale = 0.44]{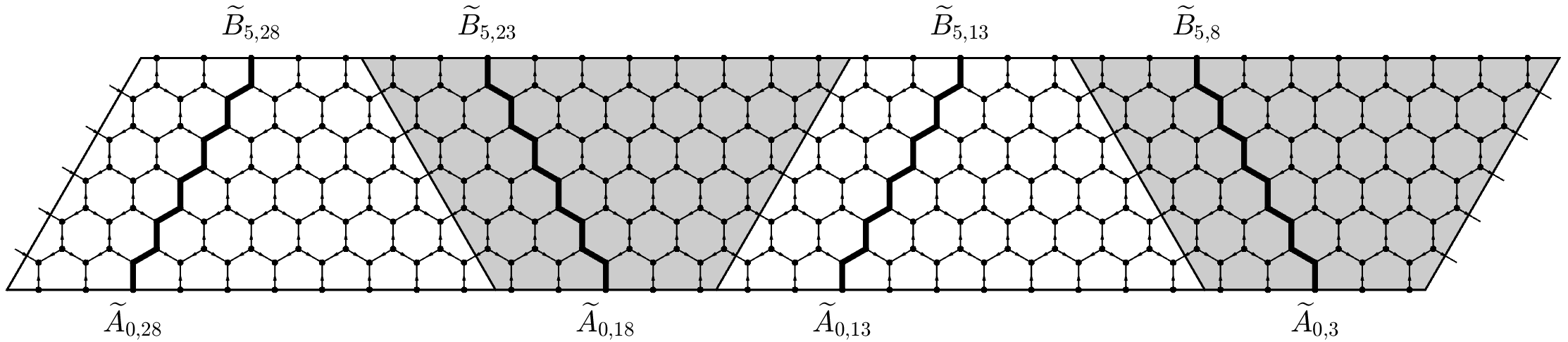}
\caption{$3$rd diagonal paths in $\widetilde{\Gamma}_5$.}
\label{fig42}
\end{subfigure}

\begin{subfigure}[b]{\textwidth}
\centering
\includegraphics[scale = 0.44]{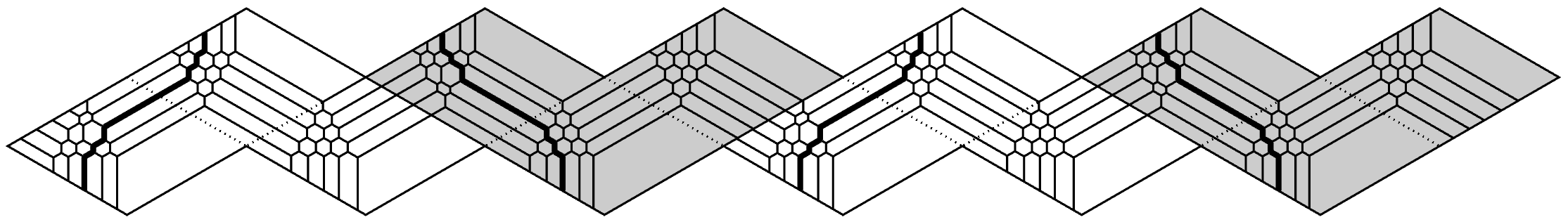}
\caption{Images of $3$rd diagonal paths in $\widetilde{B}_\delta$.}
\label{fig45}
\end{subfigure}
\caption{$j$th diagonal paths.}
\label{fig4245}
\end{figure}

\begin{lemma}\label{lemZ.0}
Let ${\sf DPath}_j^{(k)} = (\widetilde{W}_1 , \widetilde{W}_2 , \cdots , \widetilde{W}_{2n+2})$. Let $\widetilde{h} \in \MOB$. Then there exists $c \in \Z$ such that
\begin{equation}
\widetilde{h}(\widetilde{W}_1) , \widetilde{h}(\widetilde{W}_2) , \cdots , \widetilde{h}(\widetilde{W}_{2n+2}) \in \left( D_\delta^{(c)} \cup D_\delta^{(c+1)} \right).
\end{equation}
In particular,
\begin{equation}
\widetilde{h}(\widetilde{W}_1) \in D_{\delta}^{(c)}, \quad \widetilde{h}(\widetilde{W}_{2n+2}) \in D_{\delta}^{(c+1)}.
\end{equation}
\end{lemma}
\begin{proof}
Note that $\widetilde{W}_1$ and $\widetilde{W}_{2n+2}$ are boundary vertices, so $\widetilde{h}(\widetilde{W}_1)$ and $\widetilde{h}(\widetilde{W}_{2n+2})$ are contained in the boundary of $\widetilde{B}_\delta$ due to \eqref{eqn3.7} and \eqref{eqn3.8}. In particular, choose $c \in \Z$ so that $\widetilde{h}(\widetilde{W}_1) \in D_\delta^{(c)}$.

On the other hand, from \eqref{eqn3.9}, the following vectors
\begin{equation}\label{eqnZ.1}
\widetilde{h}(\widetilde{W}_1)-\widetilde{h}(\widetilde{W}_{2}), \hspace{0.5em} \widetilde{h}(\widetilde{W}_3)-\widetilde{h}(\widetilde{W}_{4}), \hspace{0.5em} \widetilde{h}(\widetilde{W}_5)-\widetilde{h}(\widetilde{W}_{6}),  \cdots , \widetilde{h}(\widetilde{W}_{2n+1})-\widetilde{h}(\widetilde{W}_{2n+2}),
\end{equation}
are in the same direction. Similarly, so are
\begin{equation}\label{eqnZ.2}
\widetilde{h}(\widetilde{W}_2)-\widetilde{h}(\widetilde{W}_{3}), \hspace{0.5em} \widetilde{h}(\widetilde{W}_4)-\widetilde{h}(\widetilde{W}_{5}), \hspace{0.5em} \widetilde{h}(\widetilde{W}_6)-\widetilde{h}(\widetilde{W}_{7}),  \cdots , \widetilde{h}(\widetilde{W}_{2n})-\widetilde{h}(\widetilde{W}_{2n+1}).
\end{equation}
Since $\widetilde{h}(\widetilde{W}_{2n+2})$ should be contained in the boundary of $\widetilde{B}_\delta$, we have $\widetilde{h}(\widetilde{W}_{2n+2}) \in D_{\delta}^{(c+1)}$. This proves the lemma.
\end{proof}

\begin{lemma}\label{lem3.1}
Let $\widetilde{A}_{i,j} , \widetilde{B}_{i,j} \in V_{\widetilde{\Gamma}_n}$. Let $j = kn+r$ where $k , r \in \Z$ and $1 \leq r \leq n$. Then for all $\widetilde{h} \in \MOB$,
\begin{align}
\widetilde{h}(\widetilde{A}_{i,j}) , \widetilde{h}(\widetilde{B}_{i,j}) & \in D_\delta^{(2k-4)}  \quad (r>i),\\
\widetilde{h}(\widetilde{A}_{i,j}) , \widetilde{h}(\widetilde{B}_{i,j}) & \in D_\delta^{(2k-5)}  \quad (r \leq i).
\end{align}
\end{lemma}
\begin{proof}
From \eqref{eqn3.7}, the lemma holds for boundary vertices $\widetilde{A}_{0,j}$ when $1 \leq j \leq 3n$. By Lemma \ref{lemZ.0}, we conclude that the lemma holds for all boundary vertices of $\widetilde{\Gamma}_n$.

Let $\widetilde{A}_{i,j} , \widetilde{B}_{i,j} \in V_{\widetilde{\Gamma}_n}$ be arbitrarily chosen. Find $k,r \in \Z$ so that $j=kn+r$ and $1\leq r \leq n$. Then there are two diagonal paths
\begin{subequations}
\begin{equation}
( \widetilde{A}_{0,j-i} , \widetilde{B}_{0, j-i} , \cdots , \widetilde{A}_{i,j} , \widetilde{B}_{i,j} , \cdots, \widetilde{A}_{n,j-i+n} , \widetilde{B}_{n,j-i+n} ),
\end{equation}
\begin{equation}
(\widetilde{B}_{n,j} , \widetilde{A}_{n,j} , \cdots , \widetilde{B}_{i,j} , \widetilde{A}_{i,j} , \cdots, \widetilde{B}_{0,j}, \widetilde{A}_{0,j}).
\end{equation}
\end{subequations}
Since $\widetilde{B}_{n,j}$ and $\widetilde{A}_{0,j}$ are boundary vertices, apply the lemma to have
\begin{equation}
\widetilde{h} (\widetilde{B}_{n,j}) \in D_\delta^{(2k-5)}, \quad \widetilde{h} (\widetilde{A}_{0,j}) \in D_\delta^{(2k-4)}.
\end{equation}
This implies
\begin{equation}
\widetilde{h}(\widetilde{A}_{i,j}) , \widetilde{h}(\widetilde{B}_{i,j}) \in (D_\delta^{(2k-5)} \cup D_\delta^{(2k-4)}).
\end{equation}

Assume $r>i$. Then again, since $\widetilde{A}_{0,j-i}$ and $\widetilde{B}_{n,j-i+n}$ are boundary vertices, apply the lemma, taking account of $r > i$
\begin{equation}
\widetilde{h} (\widetilde{A}_{0,j-i}) \in D_\delta^{(2k-4)}, \quad \widetilde{h} (\widetilde{B}_{n,j-i+n}) \in D_\delta^{(2k-3)}.
\end{equation}
Again, this leads to 
\begin{equation}
\widetilde{h}(\widetilde{A}_{i,j}) , \widetilde{h}(\widetilde{B}_{i,j}) \in (D_\delta^{(2k-4)} \cup D_\delta^{(2k-3)}).
\end{equation}
In other words, by assuming $r>i$, we have $\widetilde{h}(\widetilde{A}_{i,j}) , \widetilde{h}(\widetilde{B}_{i,j}) \in D_\delta^{(2k-4)} $, satisfying the lemma. Similarly, the lemma holds if $r \leq i$.
\end{proof}

Lemma \ref{lem3.1} can be expressed in terms of vectors in $B$.

\begin{lemma}\label{lemZ.1}
Let $\widetilde{h} \in \MOB$ and write $(x,y,z) := \widetilde{h}(\widetilde{A}_{i,j})$. Then the image of vertices which are equivalent to $\widetilde{A}_{i,j}$ i.e.
\begin{align}
\widetilde{h}(\widetilde{A}_{i,j+3k}) & =(x+ 3k \delta , \hspace{0.2em} y+3k\delta , \hspace{0.2em} z-6k\delta), \\
\widetilde{h}(\widetilde{B}_{-i+n, -i+j+2n+3k}) &= (y+(3k-2)\delta , \hspace{0.2em} x+(3k-1)\delta , \hspace{0.2em} z-(6k-3)\delta ).
\end{align}
\end{lemma}
\begin{proof}
This follows from Lemma \ref{lem3.1} and \eqref{eqn3.5}.
\end{proof}

For each $\widetilde{e} \in E_{\widetilde{\Gamma}_n}$, define ${\sf const}(\widetilde{h} ; \widetilde{e})$ as in \eqref{eqnA.2}.

\begin{lemma}\label{lemZ.5}
Let $\widetilde{h} , \widetilde{h}' \in \MOB$ and $\widetilde{e}_1 , \widetilde{e}_2 \in E_{\widetilde{\Gamma}_n}$ such that $p_e (\widetilde{e}_1) = p_e (\widetilde{e}_2) $. Then we have
\begin{equation}\label{eqn5.6}
{\sf const}(\widetilde{h}' ; \widetilde{e}_1) - {\sf const}(\widetilde{h}' ; \widetilde{e}_2) = {\sf const}(\widetilde{h} ; \widetilde{e}_1) - {\sf const}(\widetilde{h} ; \widetilde{e}_2).
\end{equation}
\end{lemma}

\begin{proof}
Let $\widetilde{e}_1 = (\widetilde{P}_1 , \widetilde{Q}_1)$ and $\widetilde{e}_2 = (\widetilde{P}_2 , \widetilde{Q}_2)$. Write $(x_P, y_P , z_P):= \widetilde{h}(\widetilde{P}_1)$ and $(x_Q , y_Q , z_Q):= \widetilde{h}(\widetilde{Q}_1)$. We have two cases as follows.

(Case 1) $p_v (\widetilde{P}_1) = p_v (\widetilde{P}_2)$ and $p_v (\widetilde{Q}_1) = p_v (\widetilde{Q}_2)$: Using Lemma \ref{lemZ.1}, there exists $k \in \Z$ such that
\begin{subequations}\label{eqnA.6}
\begin{equation}
\widetilde{h}(\widetilde{P}_2) = (x_P+3k\delta , y_P + 3k\delta , z_P - 6k\delta), 
\end{equation}
\begin{equation}
\widetilde{h}(\widetilde{Q}_2) = (x_Q+3k\delta , y_Q + 3k\delta , z_Q - 6k\delta).
\end{equation}
\end{subequations}
Therefore, the right-hand side of \eqref{eqn5.6} is $-3k\delta$ if $d(\widetilde{e}_1) \neq (-1,1,0)$. If $d(\widetilde{e}_1) = (-1,1,0)$, then the right-hand side of \eqref{eqn5.6} is $6k\delta$. Since the computed value does not depend on $\widetilde{h}$, we conclude that both sides of \eqref{eqn5.6} are equal.

(Case 2) $p_v (\widetilde{P}_1) = p_v (\widetilde{Q}_2)$ and $p_v (\widetilde{Q}_1) = p_v (\widetilde{P}_2)$: Using Lemma \ref{lemZ.1}, there exists $k \in \Z$ such that
\begin{subequations}\label{eqnA.7}
\begin{equation}
\widetilde{h}(\widetilde{P}_2) = (y_Q + (3k-2)\delta, x_Q+(3k-1)\delta , z_Q - (6k-3)\delta),
\end{equation}
\begin{equation}
\widetilde{h}(\widetilde{Q}_2) = (y_P + (3k-2)\delta, x_P+(3k-1)\delta , z_P - (6k-3)\delta).
\end{equation}
\end{subequations}
Therefore, the right-hand side of \eqref{eqn5.6} is $-(3k-1)\delta$ if $d(\widetilde{e}_1) = (0,-1,1)$. If $d(\widetilde{e}_1) = (1,0,-1)$, then the right-hand side of \eqref{eqn5.6} is $-(3k-2)\delta$. Lastly, if $d(\widetilde{e}_1) = (-1,1,0)$, then it is $(6k-3)\delta$. Since the computed value does not depend on $\widetilde{h}$, we conclude that both sides of \eqref{eqn5.6} are equal.
\end{proof}

From Lemma \ref{lemZ.1}, $\widetilde{h} \in \MOB$ is determined by $3n(n+1) = |V_{\Gamma_n}|$ number of points in $B$. Let $\lambda , \mu , \nu \in \Parn$. Choose $\delta \in \N$ so that $\delta > |\lambda|, |\mu|, |\nu|$. Write
\begin{subequations}
\begin{align}
\widetilde{h}(\widetilde{A}_{i,j}) &= (a_{i,j,1},a_{i,j,2},a_{i,j,3}) \quad (0\leq i \leq n, 1 \leq j \leq n+i) \\
\widetilde{h}(\widetilde{B}_{i,j}) &= (b_{i,j,1},b_{i,j,2},b_{i,j,3}) \quad (0\leq i \leq n, 1 \leq j \leq n+i)
\end{align}
\end{subequations}
Then $\widetilde{A}_{i,j}$ and $\widetilde{B}_{i,j}$ are representatives of $V_{\Gamma_n}$.

We now list the conditions of $(a_{i,j,k},b_{i,j,k}) \in \R^{9n(n+1)}$ so that $\widetilde{h}$ is a M{\"o}bius honeycomb. First of all, since the image of $\widetilde{h} \in \MOB$ should be contained in $B$,
\begin{subequations}\label{eqnB.2}
\begin{align}
a_{i,j,1}+ a_{i,j,2} + a_{i,j,3} = 0, \quad (1 \leq i \leq n , 1 \leq j \leq n+i) \\
b_{i,j,1} + b_{i,j,2} + b_{i,j,3} = 0. \quad (0 \leq i \leq n-1 , 1 \leq j \leq n+i)
\end{align}
\end{subequations}

From \eqref{eqn3.9},
\begin{subequations}\label{eqnB.3}
\begin{align}
& a_{i,j,3} = b_{i,j,3}, \quad a_{i,j,1}\leq b_{i,j,1}, \quad (0 \leq i \leq n, 1 \leq j \leq n+i) \\
& a_{i+1,j,2} = b_{i,j,2}, \quad a_{i+1,j,3} \leq b_{i,j,3} \quad (0 \leq i \leq n-1, 1 \leq j \leq n+i) \\
& a_{i+1,j+1,1} = b_{i,j,1}, \quad a_{i+1,j+1,2} \leq b_{i,j,2} \quad (0 \leq i \leq n-1, 1 \leq j \leq n+i)
\end{align}
\end{subequations}

From \eqref{eqn3.7}.
\begin{subequations}\label{eqnB.1}
\begin{align}
(a_{0,j,1},a_{0,j,2},a_{0,j,3}) &= (-2\delta, -\lambda_j-2\delta ,\lambda_j + 4\delta), \quad (1 \leq j \leq n) \\
(b_{n,j,1},b_{n,j,2},b_{n,j,3}) &= (-\mu_j -3\delta , -2\delta , \mu_j +5 \delta), \quad (1 \leq j \leq n) \\
(b_{n,j,1},b_{n,j,2},b_{n,j,3}) &= (-\nu_{j-n}-2\delta , -\delta , \nu_{j-n}+3\delta). \quad (n+1 \leq j \leq 2n)
\end{align}
\end{subequations}

Lastly, we need to take account of \eqref{eqn3.8}. Using \eqref{eqn3.5}, we have
\begin{equation}\label{eqnB.4}
a_{i,1,1} = a_{n-i,2n-i,2}-2 \delta, \quad a_{i,1,3} \geq a_{n-i,2n-i,3}+3 \delta. \quad (0 \leq i \leq n)
\end{equation}

We define a polytope $P_{\lambda,\mu, \nu , \delta} \subset \R^{9n(n+1)}$ using \eqref{eqnB.2}, \eqref{eqnB.3}, \eqref{eqnB.1} and \ref{eqnB.4}. 

By Theorem \ref{thm3.1}, we have a corollary.

\begin{corollary}\label{coroC.1}
Let $n \in \N$ and $\lambda , \mu , \nu \in \Parn$. Choose $\delta \in \N$ so that $\delta > |\lambda|, |\mu|, |\nu|$. Then
\begin{equation}
|P_{\lambda , \mu ,\nu, \delta} \cap \Z^{9n(n+1)}| = N_{\lambda, \mu ,\nu}.
\end{equation}
\end{corollary}

\begin{proof}
Due to Lemma \ref{lemZ.1}, $\widetilde{h}$ can be constructed from $a_{i,j,k}$ and $b_{i,j,k}$.
\end{proof}

\begin{lemma}\label{lemZ.4}
Let $\widetilde{W} \in V_{\widetilde{\Gamma}_n}$ and $\widetilde{h} \in \MOB$. Suppose $\widetilde{h} ( \widetilde{W} )$ is contained in the boundary of $\widetilde{B}_\delta$. Then there exists a boundary vertex $\widetilde{W}'$ such that $\widetilde{h} ( \widetilde{W} ) = \widetilde{h} ( \widetilde{W}' )$.
\end{lemma}

\begin{proof}
Let $c \in \Z$ be chosen so that $\widetilde{h}(\widetilde{W}) \in D_\delta^{(c)}$. Using Lemma \ref{lemZ.0}, choose ${\sf DPath}_j^{(k)} = (\widetilde{W}_1 , \widetilde{W}_2 , \cdots , \widetilde{W}_{2n+2})$ passing through $\widetilde{W}$ so that
\begin{equation}
\widetilde{h}(\widetilde{W}_1) , \widetilde{h}(\widetilde{W}_2) , \cdots , \widetilde{h}(\widetilde{W}_{2n+2}) \in \left( D_\delta^{(c-1)} \cup D_\delta^{(c)} \right),
\end{equation}
\begin{equation}
\widetilde{h}(\widetilde{W}_1) \in D_{\delta}^{(c-1)}, \quad \widetilde{h}(\widetilde{W}_{2n+2}) \in D_{\delta}^{(c)}.
\end{equation}
Note that $\widetilde{h}(\widetilde{W}_{2n+2})$ is contained in the boundary of $D_\delta^{(c)}$. Since vectors in \eqref{eqnZ.1} (resp. vectors in \eqref{eqnZ.2}) are in same direction, we conclude that $\widetilde{h} ( \widetilde{W} ) = \widetilde{h} ( \widetilde{W}_{2n+2})$.
\end{proof}

\begin{lemma}\label{lem3.0}
Let $\widetilde{h}_1 , \widetilde{h}_2 \in \MOB$ such that $\partial \widetilde{h}_1 = \xi$ and $ \partial \widetilde{h}_2 = \zeta$. Let $c_1 , c_2 \in \R_{\geq 0}$ and $m = c_1 + c_2$. Then $\widetilde{h}:=(c_1 \cdot \widetilde{h}_1 + c_2 \cdot \widetilde{h}_2)$ satisfies
\begin{equation}
\widetilde{h} \in {\tt M \ddot{O} BIUS}(\widetilde{\tau}_n , m  \delta), \quad \partial \widetilde{h} = c_1 \xi + c_2 \zeta.
\end{equation}
\end{lemma}
\begin{proof}
Due to Lemma \ref{lem2.1}, $\widetilde{h}$ is a configuration, satisfying \eqref{eqn3.9}.

To check \eqref{eqn3.7}, recall for each $1\leq j \leq n$, there exists $4\delta \leq \xi_j , \zeta_j \leq 5\delta$ such that $\widetilde{h}_1(\widetilde{A}_{0,j}) = (-2\delta , 2\delta - \xi_j , \xi_j )$ and $\widetilde{h}_1(\widetilde{A}_{0,j}) = (-2\delta , 2\delta - \zeta_j , \zeta_j )$. Then $\widetilde{h}(\widetilde{A}_{0,j}) = (-2m\delta , 2m\delta - (c_1 \xi_j + c_2 \zeta) , (c_1 \xi_j+ c_2 \zeta_j) )$, where $4m\delta \leq (c_1 \xi_j + c_2 \zeta) \leq 5m\delta$. Similarly, we may check the cases when $n+1 \leq j \leq 2n$ and $2n+1 \leq j \leq 3n$, concluding that $\widetilde{h}$ satisfies \eqref{eqn3.7} for $m\delta$. Moreover, $\partial \widetilde{h}  = c_1 \xi + c_2 \zeta$.

To check \eqref{eqn3.8}, let $(x_l , y_l , z_l):=\widetilde{h}_l (\widetilde{A}_{i,j})$ for $l=1,2$. Using Lemma \ref{lemZ.1},
\begin{align}
\widetilde{h}_l(\widetilde{A}_{i,j+3k}) & =(x_l+ 3k \delta , \hspace{0.2em} y_l+3k\delta , \hspace{0.2em} z_l-6k\delta), \\
\widetilde{h}_l(\widetilde{B}_{-i+n, -i+j+2n+3k}) &= (y_l+(3k-2)\delta , \hspace{0.2em} x_l+(3k-1)\delta , \hspace{0.2em} z_l-(6k-3)\delta ).
\end{align}

Write $(x,y,z) = c_1 \cdot (x_1 , y_1 , z_1) + c_2 \cdot (x_2, y_2 , z_2)$. Using $\widetilde{h} = c_1 \cdot \widetilde{h}_1 + c_2 \cdot \widetilde{h}_2 $,
\begin{align}
\widetilde{h}(\widetilde{A}_{i,j+3k}) & =(x+ 3k m\delta , \hspace{0.2em} y+3k m\delta , \hspace{0.2em} z-6k m\delta), \\
\widetilde{h}(\widetilde{B}_{-i+n, -i+j+2n+3k}) &= (y+(3k-2)m\delta , \hspace{0.2em} x+(3k-1)m\delta , \hspace{0.2em} z-(6k-3)m\delta ).
\end{align}
This proves that $\widetilde{h}$ satisfies \eqref{eqn3.8}.
\end{proof}

Recall that in \eqref{eqn3.26} and \eqref{eqn3.28}, we define ${\sf length}$ and ${\sf perimeter}$ for $\widetilde{h} \in \MOB$. We want to generalize the concepts by $\MOB \subseteq B^{V_{\widetilde{\Gamma}_n}}$. Let $\widetilde{e} \in E_{\widetilde{\Gamma}_n}$. For each $\widetilde{h} \in B^{V_{\widetilde{\Gamma}_n}}$, define
\begin{equation}\label{eqnZ.7}
\widehat{{\sf length}}(\widetilde{h} ; \widetilde{e}):= \frac{1}{2} \left( \widetilde{h}({\sf head}(\widetilde{e})) - \widetilde{h}({\sf tail}(\widetilde{e})) \right) \cdot d(\widetilde{e})
\end{equation}
Let $\widetilde{e}_1, \cdots , \widetilde{e}_6$ be six edges surrounding $\widetilde{\alpha} \in H_{\widetilde{\Gamma}_n}$. For each $\widetilde{h} \in B^{V_{\widetilde{\Gamma}_n}}$, define
\begin{equation}
\widehat{{\sf perimeter}}(\widetilde{h} ; \widetilde{\alpha}):= \sum_{i=1}^6 \widehat{{\sf length}}(\widetilde{h} ; \widetilde{e}_i).
\end{equation}

Lastly, define for each $\widetilde{h} \in B^{V_{\widetilde{\Gamma}_n}}$
\begin{equation}
\widehat{\partial} (\widetilde{h}) := \left( \widetilde{h}(\widetilde{A}_{0,1}) \cdot (0,0,1) , \hspace{0.5em} \widetilde{h}(\widetilde{A}_{0,2}) \cdot (0,0,1), \cdots , \widetilde{h}(\widetilde{A}_{0,3n}) \cdot (0,0,1) \right).
\end{equation}

\begin{lemma}\label{lemZ.2}
Let $\widetilde{e} \in E_{\widetilde{\Gamma}_n}$, $\widetilde{\alpha} \in H_{\widetilde{\Gamma}_n}$ and $\widetilde{h} \in B^{V_{\widetilde{\Gamma}_n}}$.
\begin{enumerate}
\item The maps below are $\R$-linear.
\begin{subequations}
\begin{equation}
B^{V_{\widetilde{\Gamma}_n}} \rightarrow \R , \hspace{0.5em} \widetilde{h} \mapsto \widehat{{\sf length}}(\widetilde{h} ; \widetilde{e}),
\end{equation}
\begin{equation}
B^{V_{\widetilde{\Gamma}_n}} \rightarrow \R , \hspace{0.5em} \widetilde{h} \mapsto \widehat{{\sf perimeter}}(\widetilde{h} ; \widetilde{\alpha}),
\end{equation}
\begin{equation}
\widehat{\partial}  : B^{V_{\widetilde{\Gamma}_n}} \rightarrow \R.
\end{equation}
\end{subequations}
\item If $\widetilde{h} \in \MOB$, then
\begin{subequations}
\begin{equation}
\widehat{{\sf length}}(\widetilde{h} ; \widetilde{e}) = {\sf length}(\widetilde{h} ; \widetilde{e}),
\end{equation}
\begin{equation}
\widehat{{\sf perimeter}}(\widetilde{h} ; \widetilde{\alpha}) = {\sf perimeter}(\widetilde{h} ; \widetilde{\alpha}),
\end{equation}
\begin{equation}
\widehat{\partial}(\widetilde{h}) = \partial (\widetilde{h}).
\end{equation}
\end{subequations}
\end{enumerate}
\end{lemma}
\begin{proof}
Straightforward from the definitions.
\end{proof}

\begin{lemma}\label{lemZ.3}
Let $n , \delta \in \N$. Then the following map is $\R$-linear:
\begin{equation}
\widehat{\iota} :  B^{V_{\widetilde{\Gamma}_n}} \rightarrow \R^{\frac{3}{2}n(n-1)} \times \R^{3n}, \quad \widetilde{h}  \mapsto \left( (p_{i,j})_{1 \leq i \leq n-1, \hspace{0.2em} 1 \leq j \leq n+i} , (\xi_j)_{1\leq j \leq 3n} \right),
\end{equation}
where $p_{i,j} := \widehat{{\sf perimeter}}(\widetilde{h} ; \widetilde{\alpha}_{i,j})$ and $(\xi_j)_{1\leq j \leq 3n} := \widehat{\partial} ( \widetilde{h})$. In particular, if $\widetilde{h} \in \MOB$, then
\begin{equation}
\widehat{\iota}(\widetilde{h}) = \iota (\widetilde{h}).
\end{equation}
\end{lemma}
\begin{proof}
Direct consequences from Lemma \ref{lemZ.2}.
\end{proof}

\begin{lemma}\label{lemZ.6}
Let $n,\delta \in \N$ and $\widetilde{h} \in \MOB$. Then we have
\begin{equation}
p_e(\widetilde{e}_1) = p_e(\widetilde{e}_2) \quad \Rightarrow \quad {\sf length}(\widetilde{h}; \widetilde{e}_1) = {\sf length}(\widetilde{h}; \widetilde{e}_2).
\end{equation}
\end{lemma}

\begin{proof}
Let $\widetilde{e}_1 = (\widetilde{P}_1 , \widetilde{Q}_1)$ and $\widetilde{e}_2 = (\widetilde{P}_2 , \widetilde{Q}_2)$. Write $(x_P, y_P , z_P):= \widetilde{h}(\widetilde{P}_1)$ and $(x_Q , y_Q , z_Q):= \widetilde{h}(\widetilde{Q}_1)$. Using Lemma \ref{lemZ.1}, we have \eqref{eqnA.6} or \eqref{eqnA.7}. Either way, we have ${\sf length}(\widetilde{h}; \widetilde{e}_1) = {\sf length}(\widetilde{h}; \widetilde{e}_2)$.
\end{proof}

Using Lemma \ref{lemZ.6}, the map
\begin{equation}
E_{\Gamma_n} \rightarrow \R , \quad e \mapsto {\sf length}(\widetilde{h} ; \widetilde{e}),
\end{equation}
where $p_e(\widetilde{e}) =e$, is well-defined.

\begin{lemma}\label{lem5.1}
Let $\delta \in \N$ and $\widetilde{h} \in \MOB$ and $ (\xi_1 , \cdots , \xi_{3n}) := \partial \widetilde{h}$. Then
\begin{equation}\label{eqn5.1}
\sum_{e \in E_{\Gamma_n}} {\sf length}(\widetilde{h} ; \widetilde{e})= \frac{1}{2}\sum_{1\leq j \leq 3n}\xi_j,
\end{equation}
where $p_e(\widetilde{e}) = e$.
\end{lemma}

\begin{proof}
In $\Gamma_n$, there are two types of edges: vertical edges and non-vertical edges. Let $e \in E_{\Gamma_n}$ and $\widetilde{e} \in E_{\widetilde{\Gamma}_n}$ satisfying $p_e (\widetilde{e}) = e$. If $e$ is vertical, then $\widetilde{e}$ is a vertical representative. Otherwise, $\widetilde{e}$ is a non-vertical representative.
	
Note that $\widetilde{B}_{\delta}$ in Figure \ref{fig9} is just two copies of $B_{\delta}$ in Figure \ref{fig11}. If we collect all of the non-vertical edges in Figure \ref{fig9}, then the sum of lengths is $6n\delta$. Therefore, the sum of lengths of non-vertical representatives is $3n\delta$.

Next, consider the $j$th diagonal path ${\sf DPath}_j^{(0)}$ which is
\begin{equation}
{\sf DPath}_j^{(0)} = (\widetilde{A}_{0,j} , \widetilde{B}_{0,j} , \widetilde{A}_{1,j+1}, \widetilde{B}_{1,j+1}, \cdots, \widetilde{A}_{n,j+n}, \widetilde{B}_{n,j+n}).
\end{equation}
We want to compute the length of edges consisting a $j$th diagonal path ${\sf DPath}_j^{(k)}$. Due to Lemma \ref{lemZ.6}, it is sufficient to compute ${\sf DPath}_j^{(0)}$ \textit{i.e.}
\begin{equation}\label{eqnA.8}
\sum_{i=0}^n {\sf length}\left( \widetilde{h} ; (\widetilde{A}_{i,i+j},\widetilde{B}_{i,i+j}) \right) + \sum_{i=1}^{n} {\sf length} \left( \widetilde{h} ; (\widetilde{A}_{i,i+j},\widetilde{B}_{i-1,i+j-1}) \right).
\end{equation}
Since $d \left( (\widetilde{A}_{i,i+j},\widetilde{B}_{i,i+j}) \right) = (-1,1,0)$, $\widetilde{h}(\widetilde{B}_{i,i+j})-\widetilde{h}(\widetilde{A}_{i,i+j}) = (-c,c,0)$ for some $c \geq 0$. Then
\begin{equation}
{\sf length}\left( \widetilde{h} ; (\widetilde{A}_{i,i+j},\widetilde{B}_{i,i+j}) \right) = \left( \widetilde{h}(\widetilde{B}_{i,i+j})-\widetilde{h}(\widetilde{A}_{i,i+j}) \right) \cdot (0,1,0).
\end{equation} 
Similarly, we have
\begin{equation}
{\sf length} \left( \widetilde{h} ; (\widetilde{A}_{i,i+j},\widetilde{B}_{i-1,i+j-1}) \right) = \left( \widetilde{h}(\widetilde{A}_{i,i+j})-\widetilde{h}(\widetilde{B}_{i-1,i+j-1}) \right) \cdot (0,1,0).
\end{equation} 
Therefore, \eqref{eqnA.8} is simplified into
\begin{equation}
\left( \widetilde{h}(\widetilde{B}_{n,j+n}) - \widetilde{h}(\widetilde{A}_{0,j}) \right) \cdot (0,1,0).
\end{equation}
Here, $\widetilde{h}(\widetilde{A}_{0,j}) \cdot (0,1,0)$ can be computed from \eqref{eqn3.7}. Also, $\widetilde{h}(\widetilde{B}_{n,j+n}) \cdot (0,1,0)$ can be computed since $\widetilde{h}(\widetilde{B}_{n,j+n})$ is contained in one of the lines $(*, -\delta , *)$, $(*,0,*)$ or $(*, \delta , *)$ due to Lemma \ref{lemZ.1}. Indeed, this is depicted in depicted in Figure \ref{fig9}. Therefore, the length of $j$th diagonal path, \eqref{eqnA.8}, is $\xi_j - 3\delta$ if $1 \leq j\leq n$, $\xi_j - \delta$ if $n+1 \leq j\leq 2n$ and $\xi_j + \delta$ if $2n+1 \leq j\leq 3n$. 

To summarize, the sum of the lengths of the $j$th diagonal paths from $j=1$ to $j=3n$ is $-3n\delta+\sum_{j=1}^{3n} \xi_j$. During the calculation, a vertical representative occurs twice whereas a non-vertical representative occurs once. Therefore, the sum of lengths of vertical representatives is $-3n\delta+\frac{1}{2}\sum_{1\leq j \leq 3n}\xi_{j}$. In other words, the total length of $\widetilde{h}$ is $\frac{1}{2}\sum_{1\leq j \leq 3n}\xi_{j}$.
\end{proof}

\begin{lemma}\label{lemZ.7}
Let $\widetilde{e} \in E_{\widetilde{\Gamma}_n}$ and $\widetilde{f}_1^+ , \widetilde{f}_1^-, \widetilde{f}_2^+, \widetilde{f}_2^- \in E_{\widetilde{\Gamma}_n}$ be its adjacent edges assigned as in Figure \ref{fig60}. Let $\widetilde{h} \in \MOB$. Then
\begin{subequations}
\begin{equation}\label{eqnZ.3}
{\sf length}(\widetilde{h} ; \widetilde{e}) = {\sf const}(\widetilde{h} ; \widetilde{f}_1^-) - {\sf const}(\widetilde{h} ; \widetilde{f}_2^+) ,
\end{equation}
\begin{equation}\label{eqnZ.4}
{\sf length}(\widetilde{h} ; \widetilde{e})  = {\sf const}(\widetilde{h} ; \widetilde{f}_2^-) - {\sf const}(\widetilde{h} ; \widetilde{f}_1^+),
\end{equation}
\begin{equation}\label{eqnZ.5}
{\sf length}(\widetilde{h} ; \widetilde{e}) = {\sf const}(\widetilde{h} ; \widetilde{f}_1^-) + {\sf const}(\widetilde{h} ; \widetilde{f}_2^-) + {\sf const}(\widetilde{h} ; \widetilde{e}),
\end{equation}
\begin{equation}\label{eqnZ.6}
{\sf length}(\widetilde{h} ; \widetilde{e}) = -{\sf const}(\widetilde{h} ; \widetilde{f}_1^+) - {\sf const}(\widetilde{h} ; \widetilde{f}_2^+) - {\sf const}(\widetilde{h} ; \widetilde{e}).
\end{equation}
\end{subequations}
\end{lemma}

\begin{proof}
Without losing generality, assume that $d(\widetilde{e}) = (0,-1,1)$. Let ${\sf tail}(\widetilde{e}) = \widetilde{P}$ and ${\sf head}(\widetilde{e}) = \widetilde{Q}$. Then
\begin{subequations}
\begin{equation}
\widetilde{h}(\widetilde{P}) = ({\sf const}(\widetilde{h} ; \widetilde{e}), {\sf const}(\widetilde{h} ; \widetilde{f}_2^-), {\sf const}(\widetilde{h} ; \widetilde{f}_2^+)),
\end{equation}
\begin{equation}
\widetilde{h}(\widetilde{Q}) = ({\sf const}(\widetilde{h} ; \widetilde{e}), {\sf const}(\widetilde{h} ; \widetilde{f}_1^+), {\sf const}(\widetilde{h} ; \widetilde{f}_1^-)).
\end{equation}
\end{subequations}
From $d(\widetilde{e}) = (0,-1,1)$, there exists $c \geq 0$ such that $\widetilde{h}(\widetilde{Q}) - \widetilde{h}(\widetilde{P}) = (0,-c,c)$. In particular, $c = {\sf length}(\widetilde{h}; \widetilde{e})$. This proves \eqref{eqnZ.3} and \eqref{eqnZ.4}. Together with
\begin{subequations}
\begin{equation}
{\sf const}(\widetilde{h} ; \widetilde{f}_1^-) + {\sf const}(\widetilde{h} ; \widetilde{f}_1^+) + {\sf const}(\widetilde{h} ; \widetilde{e}) = 0,
\end{equation}
\begin{equation}
{\sf const}(\widetilde{h} ; \widetilde{f}_2^-) + {\sf const}(\widetilde{h} ; \widetilde{f}_2^+) + {\sf const}(\widetilde{h} ; \widetilde{e}) = 0,
\end{equation}
\end{subequations}
\eqref{eqnZ.3} and \eqref{eqnZ.4} lead to \eqref{eqnZ.5} and \eqref{eqnZ.6} as well.
\end{proof}

\section{Existence of largest-lifts}\label{secA}

In this section, we prove Lemma \ref{lemA.1}.

Consider a hexagon $\widetilde{\alpha} \in H_{\widetilde{\Gamma}_n}$, surrounded by six edges $\widetilde{e}_j$; see the middle picture of the Figure \ref{figA}. Let $\widetilde{h} \in \MOB$. Use Lemma \ref{lemZ.7} and compute
\begin{subequations}\label{eqnA.4}
\begin{align}
& {\sf length}(\widetilde{h} ; \widetilde{e}_1) = {\sf const}(\widetilde{h} ; \widetilde{e}_6)+{\sf const}(\widetilde{h} ; \widetilde{e}_1)+{\sf const}( \widetilde{h} ; \widetilde{e}_2), \\
& {\sf length}(\widetilde{h} ; \widetilde{e}_2) = -{\sf const}(\widetilde{h} ; \widetilde{e}_1)-{\sf const}(\widetilde{h} ; \widetilde{e}_2)-{\sf const}(\widetilde{h} ; \widetilde{e}_3), \\
& {\sf length}(\widetilde{h} ; \widetilde{e}_3))  = {\sf const}(\widetilde{h} ; \widetilde{e}_2)+{\sf const}( \widetilde{h} ; \widetilde{e}_3)+ {\sf const}( \widetilde{h} ; \widetilde{e}_4), \\
& {\sf length}(\widetilde{h} ; \widetilde{e}_4)  = -{\sf const}(\widetilde{h} ; \widetilde{e}_3)-{\sf const}( \widetilde{h} ; \widetilde{e}_4)-{\sf const}(\widetilde{h} ; \widetilde{e}_5), \\
& {\sf length}(\widetilde{h} ; \widetilde{e}_5) = {\sf const}(\widetilde{h} ; \widetilde{e}_4)+{\sf const}(\widetilde{h} ; \widetilde{e}_5)+{\sf const}( \widetilde{h} ; \widetilde{e}_6), \\
& {\sf length}(\widetilde{h} ; \widetilde{e}_6)  = -{\sf const}(\widetilde{h} ; \widetilde{e}_5)-{\sf const}( \widetilde{h} ; \widetilde{e}_6)-{\sf const}(\widetilde{h} ; \widetilde{e}_1).
\end{align}
\end{subequations}

Consequently,
\begin{equation}\label{eqnA.5}
{\sf perimeter}(\widetilde{h} ; \widetilde{\alpha}) = \sum_{i=1}^6 (-1)^i \cdot {\sf const}(\widetilde{h} ; \widetilde{e}_i).
\end{equation}

\begin{figure}
\centering
\begin{subfigure}[b]{\textwidth}
\centering
\includegraphics[scale=0.45]{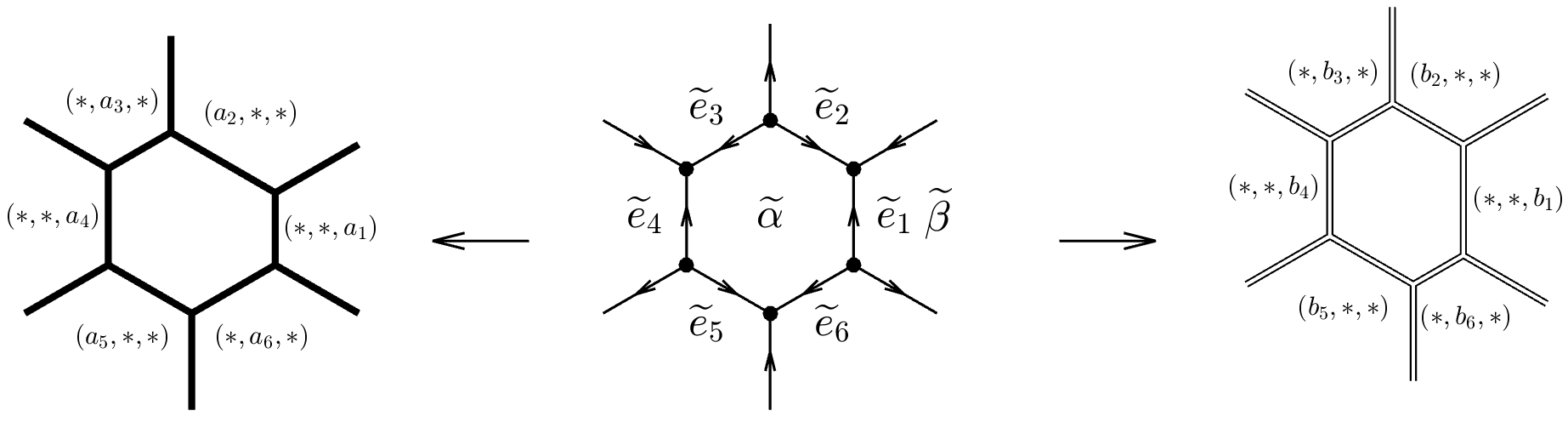}
\caption{The images of a hexagon $\widetilde{\alpha}$ under $\widetilde{h}_1$ and $\widetilde{h}_2$.}
\label{figA}
\end{subfigure}
\begin{subfigure}[b]{\textwidth}
\centering
\includegraphics[scale=0.45]{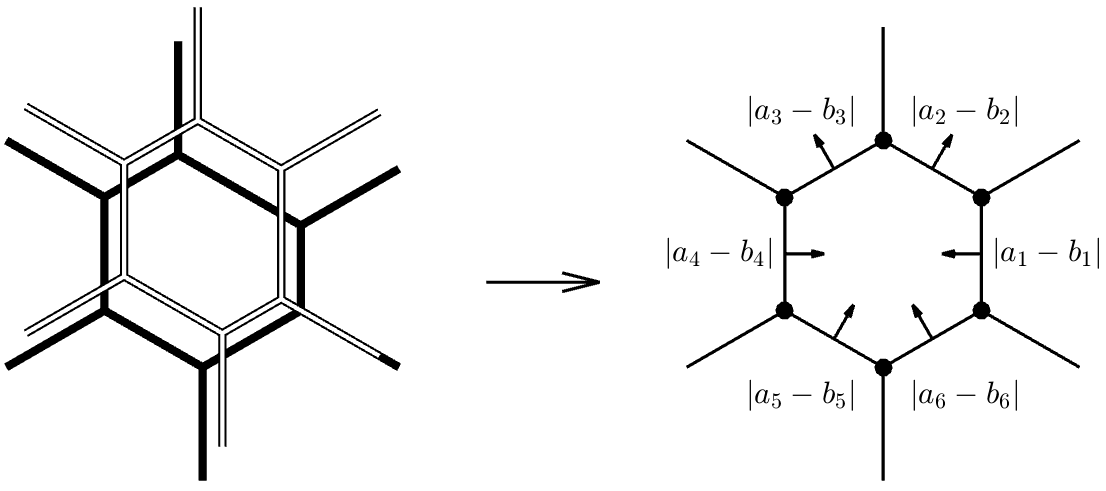}
\caption{Assigning arrows with weight to each edge.}
\label{figB}
\end{subfigure}
\caption{${\sf arrow}$ and ${\sf arrowsum}$.}
\label{figAB}
\end{figure}

Next, for each $\widetilde{h}_1 , \widetilde{h}_2 \in \MOB$ and $\widetilde{e} \in E_{\widetilde{\Gamma}_n}$, define
\begin{equation}
{\sf arrow} (\widetilde{h}_1 ; \widetilde{h}_2 ; \widetilde{e}):= {\sf const}(\widetilde{h}_2 ; \widetilde{e}) - {\sf const}(\widetilde{h}_1 ; \widetilde{e}).
\end{equation}
Then by computation,
\begin{equation}
p_e(\widetilde{e}) = p_e(\tilde{e}') \quad \Rightarrow \quad {\sf arrow} (\widetilde{h}_1 ; \widetilde{h}_2 ; \widetilde{e}) = {\sf arrow} (\widetilde{h}_1 ; \widetilde{h}_2 ; \tilde{e}' ).
\end{equation}
In other words, it makes sense to define ${\sf arrow}$ in $\Gamma_n$. In addition, define
\begin{equation}\label{eqnA.11}
{\sf arrowsum} (\widetilde{h}_1 ; \widetilde{h}_2 ; \widetilde{\alpha}):= \sum_{i=1}^6 (-1)^i \cdot {\sf arrow}(\widetilde{h}_1 ; \widetilde{h}_2 ; \widetilde{e}_i).
\end{equation}
Automatically from \eqref{eqnA.5},
\begin{equation}
{\sf arrowsum}(\widetilde{h}_1 ; \widetilde{h}_2 ; \widetilde{\alpha}) = {\sf perimeter}(\widetilde{h}_2 ; \widetilde{\alpha}) - {\sf perimeter}(\widetilde{h}_1 ; \widetilde{\alpha}).
\end{equation}

\begin{figure}
\centering
\includegraphics[scale=0.5]{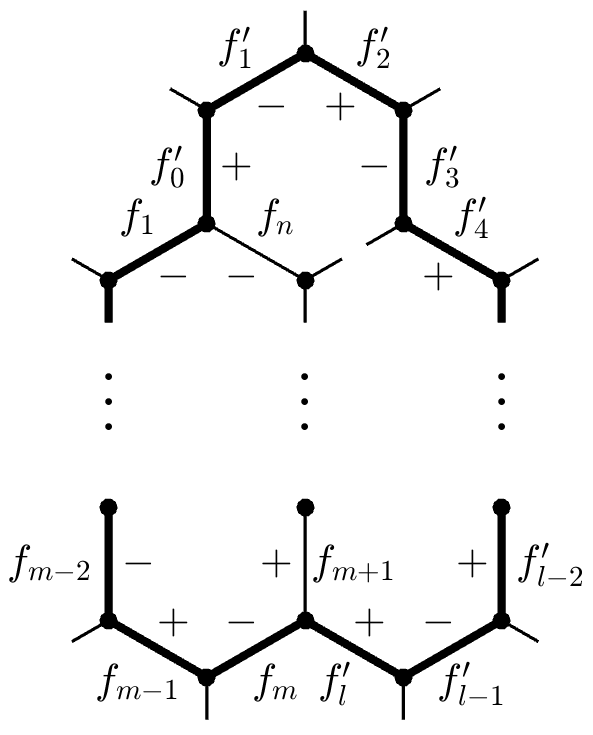}
\caption{Constructing an orientable loop with alternating sign of ${\sf arrow}$.}
\label{figD}
\end{figure}

\begin{proof}[Proof of Lemma \ref{lemA.1}]
Suppose there exist $\widetilde{h}_1 , \widetilde{h}_2 \in \MOB$ such that
\begin{equation}
\widetilde{h}_1 \neq \widetilde{h}_2 , \quad \iota ( \widetilde{h}_1 ) = \iota ( \widetilde{h}_2).
\end{equation}
Then there exists $f_1 \in E_{\Gamma_n}$ such that ${\sf arrow}(\widetilde{h}_1 ; \widetilde{h}_2 ; \widetilde{f}_1) \neq 0$ for $p_e (\widetilde{f}_1) = f_1$. Without losing generality, assume that 
\begin{equation}\label{eqnA.9}
{\sf arrow}(\widetilde{h}_1 ; \widetilde{h}_2 ; \widetilde{f}_1) < 0.
\end{equation}
Let $\widetilde{A}$ be an endpoint of $\widetilde{f}_1$. Since we are assuming $\partial \widetilde{h}_1 = \partial \widetilde{h}_2$, $\widetilde{A}$ is not a boundary vertex. Therefore, there are two more edges connected to $\widetilde{A}$: write them as $\widetilde{f}'$ and $\widetilde{f}''$. Then
\begin{equation}
{\sf const}(\widetilde{h}_1 ; \widetilde{f}_1) + {\sf const}(\widetilde{h}_1 ; \widetilde{f}') + {\sf const}(\widetilde{h}_1 ; \widetilde{f}'') = {\sf const}(\widetilde{h}_2 ; \widetilde{f}_1) + {\sf const}(\widetilde{h}_2 ; \widetilde{f}') + {\sf const}(\widetilde{h}_2 ; \widetilde{f}'')  = 0.
\end{equation}
In other words, 
\begin{equation}
{\sf arrow}(\widetilde{h}_1 ; \widetilde{h}_2 ; \widetilde{f}_1) + {\sf arrow}(\widetilde{h}_1 ; \widetilde{h}_2 ; \widetilde{f}') + {\sf arrow}(\widetilde{h}_1 ; \widetilde{h}_2 ; \widetilde{f}'') = 0.
\end{equation}
Due to \eqref{eqnA.9}, it is possible to choose $\widetilde{f}_2$ between $\widetilde{f}'$ and $\widetilde{f}''$ so that
\begin{equation}
{\sf arrow}(\widetilde{h}_1 ; \widetilde{h}_2 ; \widetilde{f}_2) > 0.
\end{equation}
Write $f_2 := p_e (\widetilde{f}_2)$. Select the endpoint of $\widetilde{f}_2$ aside from $\widetilde{A}$. Again, write the other two edges connected to the endpoint as $\widetilde{f}'$ and $\widetilde{f}''$ to find $\widetilde{f}_3$ satisfying ${\sf arrow}(\widetilde{h}_1 ; \widetilde{h}_2 ; \widetilde{f}_3)$.

In this way, there exists $\widetilde{f}_1 , \widetilde{f}_2 , \widetilde{f}_3 \cdots \in E_{\widetilde{\Gamma}_n}$ such that $\widetilde{f}_i$ and $\widetilde{f}_{i+1}$ are connected and
\begin{equation}\label{eqnA.10}
(-1)^i \cdot {\sf arrow}(\widetilde{h}_1 ; \widetilde{h}_2 ; \widetilde{f}_i) > 0.
\end{equation}
Write $f_i := p_e(\widetilde{f}_i)$. Since $\Gamma_n$ is a finite graph, $f_i$ forms up a loop, eventually. Assume that $f_1 , f_2 , \cdots , f_n$ form a loop $C$ in $\Gamma_n$ without self-intersection. 

We want to choose $C=(f_1 , f_2 , \cdots , f_n)$ so that $C$ is orientable. Suppose it is not. Due to Lemma \ref{lem5.0}, $n$ is an odd integer. Therefore,
\begin{equation}\label{eqnA.15}
{\sf arrow}(\widetilde{h}_1 ; \widetilde{h}_2 ; \widetilde{f}_1) <0, \quad {\sf arrow}(\widetilde{h}_1 ; \widetilde{h}_2 ; \widetilde{f}_n) <0.
\end{equation}
Let $f_0'$ be the edge connected to $f_1$ and $f_n$. Choose $p_e(\widetilde{f}_0') = f_0'$. Then from \eqref{eqnA.15},
\begin{equation}
{\sf arrow}(\widetilde{h}_1 ; \widetilde{h}_2 ; \widetilde{f}_0') >0.
\end{equation}
Again, choose $f_0' , f_1' , f_2' ,  \cdots$ in $E_{\Gamma_n}$ so that for each $p_e(\widetilde{f}_i') = f_i'$,
\begin{equation}
(-1)^i \cdot {\sf arrow}(\widetilde{h}_1 ; \widetilde{h}_2 ; \widetilde{f}_i') > 0.
\end{equation}

Since $\Gamma_n$ is a finite graph, there exists the minimal $l \in \N$ such that $f_l'$ meets one of $f_i$ or $f_i'$. If $f_l'$ share an endpoint with $f_i'$ and $f_{i+1}'$, then $f_{i+1}' , f_{i+2}' , \cdots, f_l ' $ form an orientable loop, as desired. Indeed, if not, then there are two non-orientable loops without intersection, leading to contradiction. 

Otherwise, $f_l'$ share an endpoint with $f_m$ and $f_{m+1}$. In Figure \ref{figD}, we can see the loop $(f_1 , f_2  ,  \cdots , f_n)$ and edges $f_0' , f_1' , \cdots , f_l'$ in $\Gamma_n$. Suppose $m$ is an odd integer \textit{i.e.} ${\sf arrow}(\widetilde{h}_1 ; \widetilde{h}_2 ; \widetilde{f}_m)$ is negative. Then construct
\begin{equation}
C = (f_0' , f_1 ' , f_2' , \cdots , f_{l-1}' , f_l' , f_m , f_{m-1} , \cdots , f_2 , f_1)
\end{equation}
The loop $C$ is depicted as a bold line in Figure \ref{figD}. Then the values of ${\sf arrow}$ alternate, meaning that there are even number of edges consisting the loop $C$. By Lemma \ref{lem5.0}, it is an orientable loop.

If $m$ is an even integer, construct the loop $C$ by
\begin{equation}
C = (f_0' , f_1 ' , f_2' , \cdots , f_{l-1}' , f_l' , f_{m+1} , f_{m+2} , \cdots , f_{n-1},  f_n).
\end{equation}
Either way, we have an orientable loop $C = (f_0'', f_1'', \cdots , f_l'')$ without self-intersection, the values ${\sf arrow}$ alternating.

Collect all hexagons $\alpha \in H_{\Gamma_n}$ which are in the interior region of the loop $C$. Write the subset of such hexagons as $H'$. Consider
\begin{equation}\label{eqnA.12}
\sum_{p_h(\widetilde{\alpha}) \in H'} {\sf arrowsum}(\widetilde{h}_1 ; \widetilde{h}_2 ; \widetilde{\alpha}).
\end{equation}
Here, for each $\alpha \in H'$, choose one of $\widetilde{\alpha} \in H_{\widetilde{\Gamma}_n}$ such that $p_h(\widetilde{\alpha}) = \alpha$ and add corresponding ${\sf arrowsum}$ to the summation.

According to \eqref{eqnA.11}, \eqref{eqnA.12} is an alternating sum of ${\sf arrow}(\widetilde{h}_1 ; \widetilde{h}_2 , \widetilde{e})$. Let two hexagons $\alpha_1 , \alpha_2$ be adjoined by $e \in E_{\Gamma_n}$. Let $p_e(\widetilde{e}) = e$. If $\alpha_1 , \alpha_2 \in H'$, then both values ${\sf arrow}(\widetilde{h}_1 ; \widetilde{h}_2 ; \widetilde{e})$ and $-{\sf arrow}(\widetilde{h}_1 ; \widetilde{h}_2 ; \widetilde{e})$ appear in the computation of \eqref{eqnA.12}, cancelling out each other. Indeed, in Figure \ref{figA}, hexagons $\widetilde{\alpha}$ and $\widetilde{\beta}$ are adjoined by $\widetilde{e}_1$. As a surrounding edge of $\widetilde{\beta}$, $\widetilde{e}_1$ can be denoted as $\tilde{e}_4'$.

This means that \eqref{eqnA.12} involves summation of $\widetilde{f}_i''$ consisting the loop $C$. Also, ${\sf arrow}(\widetilde{h}_1 ; \widetilde{h}_2 ; \widetilde{f}_i'')$ should alternate, since $f_i''$ and $f_{i+1}''$ are connected and on the same hexagon. Therefore,
\begin{equation}\label{eqnA.13}
\sum_{p_h(\widetilde{\alpha}) \in H'} {\sf arrowsum}(\widetilde{h}_1 ; \widetilde{h}_2 ; \widetilde{\alpha}) = \pm \left( \sum_{i=1}^l (-1)^i \cdot {\sf arrow}(\widetilde{h}_1 ; \widetilde{h}_2 ; \widetilde{f}_i'') \right).
\end{equation}
This is non-zero due to \eqref{eqnA.10}.

On the other hand, from \eqref{eqnA.11},
\begin{equation}\label{eqnA.14}
\sum_{p_h(\widetilde{\alpha}) \in H'} {\sf arrowsum}(\widetilde{h}_1 ; \widetilde{h}_2 ; \widetilde{\alpha}) = \sum_{p_h(\widetilde{\alpha}) \in H'} {\sf perimeter}( \widetilde{h}_2 ; \widetilde{\alpha}) - {\sf perimeter}( \widetilde{h}_1 ; \widetilde{\alpha}) = 0.
\end{equation}
Hence, \eqref{eqnA.13} and \eqref{eqnA.14} lead to contradiction, proving that $\iota$ is injective.
\end{proof}

\section*{Acknowledgements}
We thank Shiliang Gao and Alexander Yong  for helpful remarks on drafts of this preprint. We are grateful to Jiyang Gao for pointing out the exceptional case of ``white triangle of size $0.5$''. We also thank Allen Knutson for helpful conversations. This work was partially supported by UIUC Campus Research Board grant RB24025.


\end{document}